\documentclass[a4paper, 12pt]{amsart}
\usepackage{amsmath,amsthm,amssymb}
\usepackage[T1]{fontenc}
\usepackage{url}
\usepackage{amsmath,amsthm,amssymb}
\usepackage{verbatim}
\usepackage{mathrsfs}
\usepackage{graphics}
\usepackage{graphicx}
\usepackage{subfigure}
\usepackage{hyperref}
\usepackage{mathtools}
\usepackage{color}
\usepackage{tikz-cd}
\usepackage{enumerate}
\usepackage[numbered]{bookmark}
\usepackage{soul}
\allowdisplaybreaks[1]
\newtheorem{theorem}{Theorem}[section]
\newtheorem{proposition}[theorem]{Proposition}
\newtheorem{lemma}[theorem]{Lemma}
\newtheorem{corollary}[theorem]{Corollary}
\theoremstyle{remark}
\newtheorem{remark}[theorem]{Remark}

\numberwithin{equation}{section}

\newcommand{\R}{{\mathbb{R}}}
\newcommand{\A}{{\mathcal{A}}}

\newcommand{\Z}{{\mathbb{Z}}}
\newcommand{\N}{{\mathbb{N}}}

\newcommand{\sgn}{\operatorname{sgn}}

\newcommand{\Leb}{\operatorname{Leb}}
\newcommand{\Aff}{\operatorname{Aff}}

\makeatletter
\newcommand{\customlabel}[2]{\protected@write\@auxout{}{\string\newlabel{#1}{{#2}{\thepage}{#2}{#1}{}}}\hypertarget{#1}{#2}}
\makeatother

\begin{document}

\title[On measures and semiconjugacies for AIETs]
{On measures and semiconjugacies for affine interval exchange transformations}

\author[P.\ Berk]{Przemys\l aw Berk}
\address{Faculty of Mathematics and Computer Science, Nicolaus
	Copernicus University, ul. Chopina 12/18, 87-100 Toru\'n, Poland}
\email{zimowy@mat.umk.pl}

\author[K.\ Fr\k{a}czek]{Krzysztof Fr\k{a}czek}
\address{Faculty of Mathematics and Computer Science, Nicolaus
	Copernicus University, ul. Chopina 12/18, 87-100 Toru\'n, Poland}
\email{fraczek@mat.umk.pl}

\author[\L.\ Kotlewski]{\L ukasz Kotlewski}
\address{Faculty of Mathematics and Computer Science, Nicolaus
	Copernicus University, ul. Chopina 12/18, 87-100 Toru\'n, Poland}
\email{kotlewskilukasz@mat.umk.pl}

\author[F.\ Trujillo]{Frank Trujillo}
\address{Centre de Recerca Matem\`atica, Campus de Bellaterra,
	Edifici C, 08193 Bellaterra, Barcelona, Spain}
\email{ftrujillo@crm.cat}

\date{\today}
\subjclass[2000]{}
\thanks{}

\begin{abstract}
In this article, we study affine interval exchange transformations (AIETs) which are semi-conjugated to interval exchange transformations (IETs) of hyperbolic periodic type. More precisely, we study the Hausdorff dimension of their invariant measures, as well as the Hausdorff dimension of conformal measures of self-similar interval exchange transformations, and implicit relations between them. Among the highlights of this paper, we provide a precise formula for the Hausdorff dimension when the vector of the logarithm of slopes is of central-stable type with respect to the renormalization matrix. This dimension turns out to be strictly between $0$ and $1$. Moreover, we study the regularity of the semi-conjugacy between an AIET and an IET in the periodic case, deriving explicit formulas for their supremal H\"older exponents.
\end{abstract}

\maketitle

\section{Introduction}
One of the main objects investigated during the study of flows on surfaces of finite genus $g\ge 1$ is \emph{generalized interval exchange transformations} (GIETs), that is, piecewise smooth (and increasing) bijections of the interval, with $2g+\kappa-1$ continuity intervals, where $\kappa$ is the number of singularities of the flow. They appear naturally as the first return maps to a Poincaré section of such a flow, i.e., as returns to an interval which is transversal to the foliation induced by this flow. The main object investigated within the scope of this article is \emph{affine interval exchange transformations} (AIETs), which are piecewise linear GIETs, with positive derivative on each piece. If additionally, on each piece the derivative is 1, then we deal with an \emph{interval exchange transformation} (IET). {If the flow whose Poincaré section we are studying preserves the area form, then the first return map is an IET.} Very often, the study of GIETs reduces to the study of AIETs {as the former are, under quite general assumptions, conjugated to AIETs (see e.g. \cite{berk_rigidity_2024})}.

To formulate the results of this article, we first recall some basic notions regarding AIETs.
Let $f: I \to I$ be an affine interval exchange transformation on the interval $I$ of $d\ge 2$ intervals.
Let $\A$ be an alphabet of $d$ elements and let $(I_{\alpha})_{\alpha\in\A}$ be the collection of intervals exchanged by $f$.
Then $f$ is uniquely determined by a triple of parameters:

\begin{itemize}
 \item the \emph{permutation} $\pi=(\pi_t,\pi_b)$, where $\pi_\varepsilon:\A\to\{1,\ldots,d\}$ for $\varepsilon\in\{t,b\}$ are bijections, which govern the order of exchanged intervals before and after the action of $f$;
 \item the \emph{length vector} $\lambda\in \Lambda^{\A}$, where $\Lambda^{\A}:=\{v\in\R_+^{\A}\mid \sum_{\alpha\in\A} v_{\alpha}=1 \}$, i.e. $\lambda_{\alpha}:=|I_\alpha|$, for every $\alpha\in\A$;
 \item the \emph{log-slope vector} $\omega\in \R^{\A}$, where $\omega_\alpha:= \log Df|_{I_{\alpha}}$, for every $\alpha\in\A$.
\end{itemize}

Thus we identify $f= (\pi,\lambda,\omega)$. If $f=T$ is an IET, then we often omit the log-slope vector and just write $T= (\pi,\lambda)$.

One of the main tools to study the ergodic properties of AIETs is the \emph{Rauzy-Veech} induction. Namely, by setting $(\pi^{(0)},\lambda^{(0)},\omega^{(0)}):=f$, we define $\mathcal{RV}(f):=(\pi^{(1)},\lambda^{(1)},\omega^{(1)})$
as the first return map to the longer of intervals $I\setminus I_\alpha$ or $I\setminus f(I_{\beta})$, where $\alpha=\pi_t^{-1}(d)$ and $\beta=\pi_b^{-1}(d)$. If $I\setminus I_\alpha$ is the longer of the two, then we say that $f$ is of \emph{top type}, $\beta$ is a \emph{winner} of the induction step, and $\alpha$ is a \emph{loser}. If, on the other hand, $I\setminus f(I_{\beta})$ is the longer of the two, then we say that $f$ is of \emph{bottom type}, $\alpha$ is a \emph{winner} of the induction step, and $\beta$ is a \emph{loser}. This procedure can be continued indefinitely as long as the intervals before and after the exchange are of different lengths. We say then that $f$ is \emph{infinitely renormalizable}. Assuming that $f$ has this property, we obtain a sequence of AIETs $(\pi^{(k)},\lambda^{(k)},\omega^{(k)})_{k\in\N}$, where $(\pi^{(k)},\lambda^{(k)},\omega^{(k)})=\mathcal{RV}^k(f)$, for every $k\in\N$.

A \emph{Rauzy class} $C$ is any minimal subset of the set of permutations of the alphabet $\A$, which is invariant under the action of $\mathcal{RV}$. Note that any permutation $\pi\in C$ has exactly two possible images via the induced action of $\mathcal{RV}$. A \emph{Rauzy graph} $\mathcal C$ is a directed labeled graph with vertices being the elements of $C$ and the arrows leading from permutations to the possible images via $\mathcal{RV}$, which are labeled $t$ or $b$ depending on the type of renormalization. In particular, each vertex is a starting point and an endpoint to exactly $2$ arrows labeled by $t$ and $b$. For a given infinitely renormalizable AIET $f$, the infinite path $\gamma(f)=\gamma$ in $\mathcal C$ obtained by connecting the consecutive entries of $(\pi^{(k)})_{k\in\N}$ is called the \emph{(combinatorial) rotation number} of $f$. Moreover, we say that $\gamma$ is \emph{$\infty$-complete} if every symbol $\alpha\in\A$ is a winner infinitely many times. In all presented definitions, an AIET can be replaced by a GIET.

It is well-known that the combinatorial rotation number of any infinitely renormalizable IET is $\infty$-complete (see e.g. \cite[Proposition 4]{yoccoz_echanges_2005}). Moreover, there exists at least one infinitely renormalizable IET associated with any $\infty$-complete path, and this correspondence is one-to-one, when restricted to uniquely ergodic IETs (see Corollary~5 and Proposition~6 in \cite{yoccoz_echanges_2005}). By \cite[Proposition 7]{yoccoz_echanges_2005}, any infinitely renormalizable GIET (and, in particular, any AIET) $f$, whose combinatorial rotation number is $\infty$-complete, is semi-conjugate (via a non-decreasing surjective map $h$) to an IET $T$ with the same combinatorial rotation number; this is, $h\circ f=T\circ h$. Moreover, if $T$ is uniquely ergodic, then so is $f$. For every IET $T=(\pi,\lambda)$ and every $\omega\in \R^{\A}$, we denote by $\operatorname{Aff}(T,\omega)$ the (possibly empty) set of AIETs semi-conjugated to $T$ with the log-slope $\omega$. {Let us point out that $\langle \omega, \lambda \rangle = 0$ if, and only if $\operatorname{Aff}(T,\omega) \neq \emptyset$ (see \cite[Proposition 2.3]{marmi_affine_2010}). Moreover, the semi-conjugacy $h$ is a conjugacy if the AIET has no \emph{wandering intervals}, i.e., intervals that are disjoint with all its iterates (see Remark after the proof of Proposition~7 in \cite{yoccoz_echanges_2005})}.

The main goal of this article is to investigate the regularity of the (semi-)conjugacy $h$ between the AIET $f$ and the IET $T$, as described above. Moreover, we are interested in the Hausdorff dimension of the invariant measure of the considered AIET. Note that these notions are closely tied. Indeed, if, for example, $h$ is a $C^1$-diffeomorphism, then the unique invariant measure of $f$ is $(h^{-1})_*Leb$ and it is equivalent to the Lebesgue measure, hence its Hausdorff dimension is $1$.

To be more precise, we want to study the (semi-)conjugacy and the invariant measures for AIETs with \emph{periodic} (combinatorial) rotation number, i.e., AIETs whose rotation number is an $\infty$-complete periodic path in the corresponding Rauzy graph. We refer to such AIET simply as being of \emph{periodic type}. Recall that an IET $T=(\pi,\lambda)$ is \emph{self-similar}, if there exists $N\ge 1$ such that $\mathcal{RV}^N(T)=T$, up to rescaling. In particular, a periodic type AIET is semi-conjugate to a self-similar IET.

With the notation as in the previous paragraph, with every self-similar IET $T$, we can associate a primitive \emph{self-similarity matrix} $M=M(T)\in SL_{\A}(\Z)$, which describes visits of the intervals exchanged by $\mathcal{RV}^N(T)$, where $N$ denotes the period of the combinatorial rotation number, to the intervals exchanged by $T$, before the first return to the domain of $\mathcal{RV}^N(T)$. In particular,
\begin{equation}
\label{eq:Perron}
 \lambda M=e^{\rho_T}\lambda,
\end{equation}
where $\rho_T$ is the logarithm of the Perron-Frobenius eigenvalue of $M$. It is worth mentioning that the matrix $M$ can actually be identified with the classical \emph{Kontsevich-Zorich} cocycle on the space of all interval exchange transformations (see e.g., \cite{yoccoz_interval_2010}).

{It follows from \cite[Proposition 7.6]{yoccoz_interval_2010} that, for self-similar IETs, the associated self-similarity matrix $M$ acts as the identity on $\textup{Ker}(\Omega_\pi)$, where $\Omega_\pi$ is the \emph{translation matrix} associated to $T$ (see \eqref{eq:exchange_matrix} for a definition of this matrix). In particular, the minimal number of unit eigenvalues of $M$ is bigger than or equal to $\dim (\textup{Ker}(\Omega_\pi))$.

Moreover, it is a classical fact (see, e.g., \cite{zorich_finite_1996}) that $M$ is a symplectic matrix with respect to the symplectic structure determined by the translation matrix when restricted to an $M$-invariant $2g$-dimensional subspace of $\R^{\A}$ complementary to $\textup{Ker}(\Omega_\pi)$, where $$g = \frac{d - \dim(\textup{Ker}(\Omega_\pi))}{2}.$$
As a consequence, if $M$ has $\nu$ as an eigenvalue, then $\nu^{-1}$ is also an eigenvalue of $M$.}
In particular, the number of eigenvalues of modulus greater than $1$ is equal to the number of eigenvalues of modulus smaller than $1$.

Throughout this work, we make additional assumptions on the matrix $M$, namely, we will assume that $M$ has exactly $\kappa-1$ unit eigenvalues and exactly $g$ eigenvalues of modulus greater than (and smaller than) $1$ {with different moduli}. {Notice that in this case, since all non-unit eigenvalues have different moduli, all eigenvectors of the matrix $M$ are real. Moreover, taking the square of the matrix $M$ (instead of $M$) we have that, in addition, all eigenvalues of the matrix are positive.}

Let us mention that, in the setting of first return maps of flows on surfaces mentioned at the beginning of the introduction, $g$ is exactly the genus of the underlying surface, while $\kappa$ is the number of singularities of the flow.
It is worth mentioning that $\kappa-1$ is the minimal number of unit eigenvalues of $M$, for IETs obtained as the first return maps of flows with $\kappa$ singularities. This follows from the geometry of the considered surface and the fact that $M$ is a symplectic matrix.
Let $f$ be an AIET of periodic type with an $\infty$-complete rotation number $\gamma$, semi-conjugated to a self-similar IET $T$. If the matrix $M=M(T)$ has exactly $g \geq 1$ real simple eigenvalues greater than (and smaller than) $1$ with different moduli and $\kappa-1$ unit ones, we say that $f$ is of \emph{hyperbolic periodic type}.

Recall that since $T$ is self-similar, it is uniquely ergodic and thus $f$ is also uniquely ergodic.

If $f$ is of hyperbolic periodic type, then the eigenvectors of $M$ form a basis of $\R^{\A}$. Moreover, we have exactly $g$ expanding, $g$ contracting, and $\kappa-1$ invariant (fixed) eigenvectors.
Let $\omega\in\R^{\A}$ {be a vector such that $\langle \omega,\lambda\rangle=0$. Then, by \eqref{eq:Perron}, $\omega$ cannot have the Perron-Frobenius eigenvector as a component in its basis decomposition.} We say that $\omega$ is of

\begin{itemize}
 \item \emph{unstable type}, if in the basis decomposition it has at least one expanding (non-Perron-Frobenius) eigenvector,
 \item \emph{central-stable type}, if in the basis decomposition it has at least one fixed but no expanding eigenvectors.
 \item \emph{stable type}, if in the basis decomposition it has only contracting eigenvectors.
\end{itemize}

It is noteworthy that this classification is coherent with the classical Oseledets filtration of the Kontsevich-Zorich cocycle, when applied to the invariant discrete measure, supported on the $\mathcal{RV}$-orbit of $T$. As we illustrate in this article, the properties of the (semi-)conjugacy depend on where the log-slope vector $\omega$ of $f$ belongs in the above classification.

\smallskip

{
The invariant measures of piecewise affine circle homeomorphisms were first studied by Herman \cite{herman_sur_1979} who proved that for a map in this class, with exactly two \emph{break points} (i.e., discontinuities of the derivative) and irrational rotation number, its unique invariant probability measure is absolutely continuous with respect to Lebesgue if and only if the break points lie on the same orbit. For maps with finitely many break points and with irrational rotation number of bounded type, Liousse \cite{liousse_nombre_2005} showed that, if the slopes satisfy an explicit generic condition (that depends only on the number of break points), the unique invariant probability measure is singular with respect to Lebesgue. More recently, in \cite{trujillo_hausdorff_2024}, the last author proved that for a \emph{typical} (in a combinatorial rotation number sense, see, e.g., \cite{berk_rigidity_2024}) piecewise affine circle map with a finite number of break points that lie in distinct orbits, the unique invariant probability measure is singular with respect to Lebesgue and in fact has zero Hausdorff dimension. In particular, any homeomorphism that conjugates such a map to a rigid rotation is not H\"older continuous. However, the dimensional properties of singular invariant measures for non-typical maps remain open, for example, for periodic type piecewise affine circle maps whose slopes satisfy the generic condition in \cite{liousse_nombre_2005}.

In the AIET setting, the works of Cobo \cite{cobo_piece-wise_2002} and of Ulcigrai with the last author \cite{trujillo_affine_2024} yield the following dichotomy for invariant measures. A typical AIET is either smoothly conjugate to an IET, or its unique invariant probability measure is singular with respect to Lebesgue. The former occurs when the log-slope vector of the AIET belongs to the stable space in the Oseledets filtration of the Kontsevich-Zorich cocycle. The stable and the unstable cases were treated in \cite{cobo_piece-wise_2002}, and the central case in \cite{trujillo_affine_2024}. In the central scenario, the authors also conclude the absence of wandering intervals. Moreover, although not explicitly stated in these articles, the arguments therein yield the same results for AIETs of hyperbolic periodic type, with the classification into stable, central, and unstable as above. To the best of our knowledge, there are no results on the dimension of these singular measures for AIETs that do not arise from circle maps.

More generally, existing results on GIETs concern either piecewise smooth circle maps or GIETs that are smoothly conjugate to IETs. In the circle setting, the results of Khanin and Koci\'c \cite{khanin_hausdorff_2017}, and of the last author \cite{trujillo_hausdorff_2024}, show that the unique invariant probability measure of a typical piecewise smooth circle map has Hausdorff dimension zero. Let us mention that the singularity with respect to Lebesgue of this measure had been previously established in various contexts for maps with one or two break points, see \cite{dzhalilov_invariant_1998, dzhalilov_circle_2006, dzhalilov_singular_2009}. For GIETs not necessarily arising from circle maps, the smooth linearization results of Marmi, Moussa, and Yoccoz \cite{marmi_linearization_2012}, as well as the rigidity results of Ghazouani \cite{ghazouani_local_2021} and of Ghazouani with Ulcigrai \cite{ghazouani_priori_2023}, identify large classes of GIETs that are smoothly conjugate to uniquely ergodic IETs. In these cases, the unique invariant probability measure is equivalent to the Lebesgue measure and has Hausdorff dimension one.

Finally, let us mention that in contrast to sufficiently regular piecewise smooth circle maps, where Denjoy's theorem ensures minimality (see, for example, \cite[Theorem 6.5.5]{herman_sur_1979}), large classes of AIETs admit wandering intervals \cite{camelier_affine_1997, cobo_piece-wise_2002, bressaud_persistence_2010, marmi_affine_2010, cobo_wandering_2018}. In such cases, the invariant probability measures are supported in the complement of the union of all wandering intervals.

The main result of this work, which we state below, provides explicit values for the dimension of the unique invariant probability measure of any AIET of hyperbolic periodic type.}

\begin{theorem}\label{thm: main1}
 Let $f\in \operatorname{Aff}(T,\omega)$ be an AIET of hyperbolic periodic type, semi-conjugated to a self-similar IET $T=(\pi,\lambda)$, with the vector of logarithms of slopes $\omega$. Let $\mu_f$ be the invariant measure of $f$. Then
 \begin{enumerate}
 \item if $\omega$ is of stable type, then $\operatorname{dim}_H(\mu_f)=1$;
 \item if $\omega$ is of central-stable type, then $0<\operatorname{dim}_H(\mu_f)<1$;
 \item if $\omega$ is of unstable type, then $\operatorname{dim}_H(\mu_f)=0$.
 \end{enumerate}
\end{theorem}

Moreover, in the central-stable case, we have an explicit formula for the Hausdorff dimension of the invariant measure.

\begin{theorem}\label{thm: main3}
 Let $f\in \operatorname{Aff}(T,\omega)$ be an AIET of hyperbolic periodic type, semi-conjugated to a self-similar IET $T=(\pi,\lambda)$, with the vector of logarithms of slopes $\omega$. Let $\mu_f$ be the invariant measure of $f$.
 If the log-slope vector $\omega$ is of central-stable type, then
 \[
 \operatorname{dim}_H(\mu_f)=\frac{\rho_T}{\mathcal{G}(T, \omega)}\in (0,1),
 \]
 where $\mathcal{G}(T, \omega)$ is given by \eqref{eq: HD_invariant_formula}.
\end{theorem}

If the log-slope vector $\omega$ of the AIET $f$ in Theorem \ref{thm: main3} is invariant under the self-similarity matrix of $T$, then the expression above can be viewed as a \emph{Ledrappier-Young formula}, namely, as the ratio between the entropy and the Lyapunov exponent, for an ergodic piecewise linear map associated to $f$, see Remark~\ref{rmk:LY} for some details.

Let us point out that Theorems \ref{thm: main1} and \ref{thm: main3}, together with the recent rigidity results by the first and last authors \cite{berk_rigidity_2024} which provide conditions for a GIET to be smoothly conjugated to an AIET with a log-slope vector of central-stable type, also describe the dimensional properties of certain non-affine GIETs. Let us stress the fact that, although not explicitly stated in \cite{berk_rigidity_2024}, the results therein apply to GIETs semi-conjugated to self-similar IETs (see \cite[Proposition 2.3]{berk_rigidity_2024}).

We conjecture that, unlike in the periodic case, for AIETs semi-conjugated to a \emph{typical} IET (that is, coming from a full measure set of parameters), the Hausdorff dimension of the invariant measure, even for log-slope vectors of central-stable type, is 0. Roughly speaking, this should come from the fact that the underlying renormalization process is not a stationary Markov chain anymore. Hence, from this point of view, the IETs of periodic type seem to be the most interesting.

The next point of interest in our paper is the regularity of the (semi-)conjugacy $h$ between $f$ and $T$. If $\omega$ is a vector of stable type, then we denote by $\alpha(\omega)$  the modulus of the logarithm of the maximal eigenvalue, whose associated eigenvector appears in the basis decomposition of $\omega$. Define also the \emph{supremal H\"older exponent} $\mathfrak{H}(h)$ in the following way
\[
\mathfrak{H}(h):=\sup\{\alpha\in\R_{\ge 0}\mid \text{$D^{\lfloor\alpha\rfloor}f$ is $\{\alpha\}$-H\"older} \}.
\]
We have the following result on the regularity of the (semi-)conjugacy $h$.

\begin{theorem}
\label{thm: main2}
 Let $f\in \operatorname{Aff}(T,\omega)$ be an AIET of hyperbolic periodic type, semi-conjugated to a self-similar IET $T=(\pi,\lambda)$, with the vector of logarithms of slopes $\omega$. Let $h:I\to I$ be the semi-conjugacy between $f$ and $T$, i.e., $T\circ h=h\circ f$. Let $\mu_f$ be the invariant measure of $f$. Then the supremal H\"older exponent $\mathfrak{H}(h)$ of $h$ satisfies:
 \begin{enumerate}
 \item \label{numb:Thm 1.3 (1)}$\mathfrak{H}(h)= +\infty$, if $\omega$ is of stable type and $\alpha(\omega)=\rho_T$;
 \item \label{numb:Thm 1.3 (2)}$\mathfrak{H}(h)= 1+\frac{\alpha(\omega)}{\rho_T}$, if $\omega$ is of stable type and $\alpha(\omega)<\rho_T$;
 \item \label{numb:Thm 1.3 (3)}$0< \mathfrak{H}(h)<\operatorname{dim}_H(\mu_f)<1$, if $\omega$ is of central-stable type;
 \item \label{numb:Thm 1.3 (4)}$\mathfrak{H} (h)=0$, i.e. $h$ is not H\"older, if $\omega$ is of unstable type.
 \end{enumerate}
\end{theorem}

\begin{remark}\label{rem:h}
 In the case where the log-slope vector is of central-stable type, in Section~\ref{sec: regulh} we provide an effective formula for $\mathfrak{H}(h)=\frac{\rho_T}{\zeta^f}$ that allows us to actually compute it. The quantity $\zeta^f$ is determined as follows. We can treat the renormalization map defining the self-similarity of the IET $T$ as a shift of finite type to which we associate its graph $\mathfrak{G}_T$. This is a finite directed graph. On the graph $\mathfrak{G}_T$ (on its edges), we consider a function $\vartheta_-$, whose values depend on the AIET $f$, or equivalently on the log-slope vector $\omega$. Then, for any finite path in the graph, we can consider the mean value of the function $\vartheta_-$ along that path. The quantity $\zeta^f$ is the largest mean value for all elementary (with no repeated vertices) loops on $\mathfrak{G}_T$. Since the set of elementary loops is finite, the quantity $\zeta^f$ can be effectively computed.
\end{remark}

If the log-slope vector is of stable or central-stable type, $h$ is known to be invertible, i.e.\ it is really a conjugacy (see \cite[Theorem 1]{cobo_piece-wise_2002} and \cite[Theorem 1]{trujillo_affine_2024}). Then one can ask about the regularity of $h^{-1}$. To formulate the answer to this problem, let us recall first that for a given measurable system $(X,\mathcal B, S)$ a probability measure $\nu$ on $X$ is called \emph{$\phi$-conformal} with a \emph{potential} $\phi:X\to\R$ if the measures $\nu$ and $S_*\nu$ are equivalent and the Radon-Nikodym derivative $\frac{d(S^{-1})_*\nu}{d\nu}$ equals $e^\phi$. In our specific case, we consider conformal measures for IETs, with piece-wise constant potentials, constant over exchanged intervals. Such potentials can be naturally identified with the vectors in $\R^{\A}$ and thus, for every $\omega\in \R^{\A}$, we denote it by $\phi_\omega$. The existence and further properties of conformal measures are discussed in Section~\ref{subs: conformal}. If $\omega$ is of stable type, then the Hausdorff dimension of the unique $\phi_{\omega}$-conformal measure is 1. This follows from Corollary \ref{cor: bded_density}. Moreover, we prove the following result on the Hausdorff dimension of such measures in our setting, when the vector of potentials is of central-stable type.

\begin{theorem}
\label{thm: main5}
 Let $T=(\pi,\lambda)$ be a self-similar IET of hyperbolic periodic type, with a primitive self-similarity matrix $M$ and let $\nu_\omega$ be a $\phi_\omega$-conformal measure for $T$, where $\omega$ is of central-stable type. Then
 \[
 \operatorname{dim}_H(\nu_\omega)= \frac{\mathcal{H}(T, \omega)}{\rho_T} \in (0,1),
 \]
 where $\mathcal{H}(T, \omega)$ is given by \eqref{eq: HD_conformal_formula}.
\end{theorem}
{In Section \ref{sec: connections}, we show how the values of the Hausdorff dimensions given by Theorems \ref{thm: main2} and \ref{thm: main5} can be related for a given vector $\omega$ of central-stable type, after introducing an additional rescaling parameter. More precisely, if $\omega$ is of central-stable type and $f_t\in\operatorname{Aff}(T,t\omega)$, for every $t\in\R$, we show these numbers are implicitly related in the following way:
	\begin{equation}\label{eq:relation}
	\frac{d}{dt}\left(\frac{1}{t\,\operatorname{dim}_H(\mu_{f_t})} \right)=-\frac{{\operatorname{dim}_H(\nu_{t\omega})}}{t^2}.
	\end{equation}
As an important consequence of \eqref{eq:relation}, we get that
\begin{equation}\label{eq: conto0}
\lim_{t\to\pm\infty}\operatorname{dim}_H(\mu_{f_t})=0,
\end{equation}
and the rate of convergence to zero is of the order $1/|t|$. Both formulas are obtained by studying the relation of the Hausdorff dimensions with the properties of the matrix that governs an appropriate renormalization process.
In the same section we also show that both maps $t\mapsto \operatorname{dim}_H(\mu_{f_t})$ and $t \mapsto\operatorname{dim}_H(\nu_{t\omega})$ are analytic, as well as increasing on $(-\infty,0]$ and decreasing on $[0,\infty)$. It remains an open question whether \eqref{eq: conto0} holds also for the Hausdorff dimension of the conformal measure.

It is worth mentioning that at the moment our methods do not allow us to compute the Hausdorff dimension of the conformal measure, when $\omega$ is of unstable type. Firstly, the work of Bressaud, Hubert, and Maass \cite{bressaud_persistence_2010} in a certain subfamily of self-similar IETs suggests that there might be more than one $\phi_{\omega}$-conformal measure in this case. Moreover, in \cite{bressaud_persistence_2010} the authors provide an example of an orbit that supports the sum of Dirac measures, which is a $\phi_{\omega}$-conformal measure. Thus, we conjecture that for all possible $\phi_{\omega}$-conformal measures, with $\omega$ being of unstable type, their Hausdorff dimension is 0.
Finally, we are able to formulate a result on the regularity of $h^{-1}$.

\begin{theorem}
\label{thm: main4}
 Let $f\in \operatorname{Aff}(T,\omega)$ be an AIET of hyperbolic periodic type, semi-conjugated to a self-similar IET $T=(\pi,\lambda)$, with the vector of logarithms of slopes $\omega$ of either central-stable or stable type. Let $h:I\to I$ be the conjugacy between $f$ and $T$, i.e., $T\circ h=h\circ f$. Let $\nu_\omega$ be the $\phi_{\omega}$ conformal measure w.r.t.\ $T$. Then the supremal H\"older exponent $\mathfrak{H}(h^{-1})$ satisfies
 \begin{enumerate}
 \item {$\mathfrak{H}(h^{-1})= +\infty$, if $\omega$ is of stable type and $\alpha(\omega)=\rho_T$};
 \item {$\mathfrak{H}(h^{-1})= 1+\frac{\alpha(\omega)}{\rho_T}$, if $\omega$ is of stable type and $\alpha(\omega)<\rho_T$};
 \item $0< \mathfrak{H}(h^{-1})<\operatorname{dim}_H(\nu_\omega)<1$, if $\omega$ is of central-stable type.
 \end{enumerate}
\end{theorem}		

{
\begin{remark}\label{rem:invh}
In the case where the log-slope vector is of central-stable type, in Section~\ref{sec:regofinverse} we also provide an effective formula for $\mathfrak{H}(h^{-1})=\frac{\xi^f}{\rho_T}$ that also allows us to compute it. The quantity $\xi^f$ is determined similarly to $\zeta^f$, which we already mentioned in Remark~\ref{rem:h}.
This time the quantity $\xi^f$ is the smallest mean value of $\vartheta_-$ (mentioned in Remark~\ref{rem:h}) for all elementary loops on the directed graph $\mathfrak{G}_T$. Since the set of elementary loops is finite, the quantity $\xi^f$ can also be effectively computed.
\end{remark}
}			

\subsection{The structure and outline of the article}
Since the structure of the article is not linear, to facilitate navigation, we present some relationships between its parts below.

In Section \ref{sec: preliminaries}, we introduce the main notions and facts used throughout the paper. These include the Hausdorff dimension of a measure, maps of the interval, and the associated renormalization schemes. Already in this part of the paper, we show how to reduce the case when the considered vector of logarithms of slopes or the vector of potential is of central-stable type, to the case when it is of purely central type. It is especially crucial in the proof of Theorem \ref{thm: main5}, where the fact that the vector of potential is purely central (invariant) is required to define the renormalization as in Section \ref{sec: renormalization}, an important stepping stone in the proof.

In Section~\ref{sec: information content}, we show in a general setting how to obtain the Hausdorff dimension of a measure from the knowledge about the information content of the elements of the associated dynamical partition. In the large part of this paper, we strive to get the convergence of the information content required to apply the results of this section.

In Section \ref{sec: proofs}, we provide the proofs of Theorems \ref{thm: main1}, \ref{thm: main3} and \ref{thm: main5}. All of the proofs are given under the assumption that we have the proper convergence of information content (see Propositions \ref{prop: centrallimit}, \ref{prop: unstablelimit}, and \ref{prop: centrallimit_conformal}) and use the tools prepared in Section~\ref{sec: information content}. These convergences are obtained later in the paper.

Next, in Section \ref{sec: connections}, we introduce an additional linear parameter $t$ to the vector of the logarithms of slopes, as well as to the vector of the potential, to not only study the asymptotics of the Hausdorff dimensions with $t\to\pm\infty$, but also to provide a relation between the Hausdorff dimension of the invariant measure of the AIET with the Hausdorff dimension of the conformal measure of the IET, given by the same vector.

In order to apply the technical results that appear in the remainder of the paper, which often require compactness of the considered space, we first need to study the Cantor model of IET, which is done in Section \ref{sec: Cantor}. In particular, we relate the AIET with the Cantor model of the IET to which it is semi-conjugate and describe how the measures and the renormalization are being transferred. Considering the Cantor model is extremely important for at least two reasons. First, it allows us to show the existence and uniqueness of the conformal measure when the potential vector is of centrally stable type. Second, it is necessary to prove the vanishing of the Hausdorff dimension of the invariant measure for the AIET when its log-slope vector is of unstable type, and control over the lengths of the intervals from the dynamical partitions is more difficult due to the presence of wandering intervals.

In Section~ \ref{sec: skewproducts} we introduce a formalism that allows us to define suspensions over maps with a conformal measure and to lift the renormalization map to the level of suspension. The suspended renormalization is the main object whose properties are used to derive the main results of the paper concerning the Hausdorff dimension of measures.

In Section~\ref{sec: dynpar}, we give precise formulas for the information content of dynamical partitions (or estimations from below) in connection with the suspended renormalization defined in the previous section.

In Section~\ref{sec: perfectlyscaled}, we obtain, in a general framework, the limits of information content when the measure under consideration behaves ``uniformly'' under renormalization, see condition \eqref{eq:condlambda} with the rescaling exponent $\lambda$. In the first part of this section (Section~\ref{sec:perfectzero}), we consider the case when the measure is actually invariant under renormalization, that is, when $\lambda=0$. We obtain the convergence of the information content as well as bounds which are crucial in the proof of Theorems \ref{thm: main3} and \ref{thm: main5}. In the second part (Section~\ref{sec:lambda pos}), we consider the case where the measure under consideration is ``expanded'' by the renormalization, that is, when $\lambda>0$ in \eqref{eq:condlambda}. This is important later in the next section, which we now describe.

Section~\ref{sec:imperfectlyscaled} is devoted to the study of the growth of the information content, when condition \eqref{eq:condlambda} fails, that is, the way that the considered measure behaves under the renormalization is not identical on each element of the partition. In the context of the main results of this paper, this corresponds to log-slope vectors that are not eigenvectors of the renormalization matrix.
Results obtained here regarding the growth of the information content are helpful in showing that the Hausdorff dimension of the studied measure is equal to $0$ in this case.

In the previous sections, we assumed that the suspended renormalization map is ergodic. In Section \ref{sec:Markov}, we explain why this is true. To be more precise, we show that the renormalization process defined in Section~\ref{sec: renormalization} is actually isomorphic to a stationary Markov chain. The formalism introduced in this section will also be used later in the paper, in the part on the regularity of the semi-conjugacy. There, we treat the renormalizing map mainly from the point of view of topological dynamics, that is, we treat it as a shift of finite type and use the graph associated with it.

In Sections~\ref{sec: HDproof} and \ref{sec: HDproof_conformal} we apply the abstract results of Sections \ref{sec: perfectlyscaled} and \ref{sec:imperfectlyscaled} to obtain the required convergences in our situation (for IETs and AIETs), or in other words, to prove Propositions~\ref{prop: centrallimit},~\ref{prop: unstablelimit} and \ref{prop: centrallimit_conformal}. Since Sections~\ref{sec: perfectlyscaled} and \ref{sec:imperfectlyscaled} are written in a very general framework, to make it easier for the reader, we provide a list of objects that will be substituted in these sections in our concrete cases.

Sections~\ref{sec: reg conj} and \ref{sec:regofinverse} are devoted to determining the supremal regularity of semi-conjugacy and the inverses of conjugacy (if any), that is, to proving Theorems~\ref{thm: main2} and \ref{thm: main4}. Roughly speaking, they base on showing that while the typical local regularity is described by the results concerning the Hausdorff dimension related to the mean value of the information content, the regularity of the conjugacy or its inverse is decided by extremes of the information content, which are constructed in these sections. Still, the results obtained there largely use the notions obtained in previous sections, as well as the directed graph associated with the suspended renormalization mapping.

Finally, in Appendix~\ref{sec: appA}, we deal with stable vectors and relations of AIETs and measure defined by vectors whose difference is stable. In particular, the results of Appendix~\ref{sec: appA} are being used to get the reduction from the central-stable case to purely central, mentioned in the description of Section~\ref{sec: preliminaries}.

}

\section{Preliminaries}\label{sec: preliminaries}
We now recall basic notions and facts about the objects considered in this article. %

\subsection{Hausdorff dimension of a measure}
Let $\mu$ be probability measure on $\mathbb{R}$ with the Borel $\sigma$-algebra $\mathcal B$. We define the \emph{Hausdorff dimension} of $\mu$ as
\begin{equation*}
\dim_H(\mu) := \inf\{\dim_H(E) \mid E\in\mathcal B,\ \mu(E)=1 \},
\end{equation*}
where $\dim_H(E)$ denotes the classical Hausdorff dimension of a set $E \in \mathcal B$.

Recall that if $\dim_{H}(E)<1$, then $Leb(E)=0$. In particular, if $\mu$ is a non-trivial absolutely continuous measure with respect to $Leb$, then $\dim_H(\mu)=1$.

It is not difficult to show that the Hausdorff dimension of a measure is preserved by sufficiently regular maps. %

\begin{lemma}
Let $\mu$ and $\nu$ be two probability measures on $\R$ and let $h:\R\to\R$ be a bi-Lipschitz isomorphism such that $\nu=h_*\mu$. Then
\[
\dim_H(\mu)=\dim_H(\nu).
\]
\end{lemma}

One of the main tools to compute the Hausdorff dimension of measures on $\R$ is the classical Frostman's Lemma (see, e.g., \cite[p. 156]{przytycki_hausdorff_1989}), which we recall here.

\begin{theorem}[Frostman's Lemma]
\label{prop: frostman_Lemma}
Suppose that $\mu$ is a probability Borel measure on $\R$ and that for $\mu$-a.e.\ $x\in\R$, we have
\[ \delta_1 \leq \liminf_{\epsilon \rightarrow 0 } \dfrac{\log \mu(x-\epsilon, x+\epsilon)}{\log \epsilon} \leq \delta_2.\]
Then \[ \delta_1 \leq \dim_H(\mu) \leq \delta_2. \]
\end{theorem}

Let us mention that for several of the systems considered in this article, we will be able to see the lower limit in the lemma above as an actual limit whose exact value we can compute (see, e.g., Theorem~\ref{thm:HD_formula} for an effective criterion relying on Frostman's lemma). %

\subsection{Interval maps.}\label{subs: IETs}
Fix a natural number \( d \geq 2 \), and let \( \mathcal{A} \) be an alphabet of cardinality \( d \). We denote the set of vectors with all positive entries by \( \mathbb{R}_+^{\mathcal{A}} \). Let \( I \subset \mathbb{R} \) be an interval of the form \( [0, a) \) for some \( a \in \mathbb{R} \). Unless stated otherwise, all intervals considered in this work are supposed to be left open and right closed.
Denote by \( \mathcal{B}(I) \) the Borel \(\sigma\)-algebra on \( I \). Consider a partition of \( I \) into subintervals
$(I_\alpha)_{\alpha \in \mathcal{A}}.$

An \emph{interval exchange transformation} (IET) is a bijection \( T: I \to I \) that is a translation when restricted to \( I_\alpha \). In other words, there exists \( \delta \in \mathbb{R}^{\mathcal{A}} \) such that
\[
T(x) = x + \delta_\alpha \quad \text{for every } x \in I_\alpha.
\]
To describe an IET, one typically provides a pair of bijections \( \pi_{\varepsilon}:\mathcal{A}\to \{1,\ldots,d\} \), for \( \varepsilon \in \{b,t\} \), called a \emph{permutation} \( \pi = (\pi_t, \pi_b) \), as well as the \emph{length vector} \( \lambda \in \mathbb{R}_+^{\mathcal{A}} \).
For \( v \in \mathbb{R}^{\mathcal{A}} \), denote \( |v| = \sum_{\alpha \in \mathcal{A}} |v_\alpha| \).
Consider the partition of \( I = [0,|\lambda|) \) given by
\[
I_\alpha = \Big[ \sum_{\pi_t(\beta)<\pi_t(\alpha)} \lambda_\beta, \sum_{\pi_t(\beta)\leq\pi_t(\alpha)} \lambda_\beta \Big),
\]
in particular \(Leb(I_\alpha) = \lambda_\alpha \). 			In this case, we identify \( T = (\pi,\lambda) \).
Denoting by $\Omega_\pi$ the matrix
\begin{equation}
\label{eq:exchange_matrix}
(\Omega_\pi)_{\alpha, \beta} = \left\{ \begin{array}{cl} +1 & \text{if } \pi_b(\alpha) > \pi_b(\beta) \text{ and } \pi_t(\alpha) < \pi_t(\beta), \\
-1 & \text{if } \pi_b(\alpha) < \pi_b(\beta) \text{ and } \pi_t(\alpha) > \pi_t(\beta), \\
0 & \text{in other cases,}
\end{array}\right.
\end{equation}
we have that
\[ \delta = \Omega_\pi \lambda.\]
Let us denote the Lebesgue measure on the interval by \( Leb \).
It is clear from the definition that this measure is invariant under the IET.

We restrict ourselves to considering only \emph{irreducible} permutations, which means that
\[
\pi_t^{-1}(\{1,\ldots, k\}) \neq \pi_b^{-1}(\{1,\ldots, k\}) \quad \text{for all } 1 \leq k < d.
\]
We denote the set of irreducible permutations by \( \mathcal{S}^{\mathcal{A}} \). Denote the left endpoint of an interval \( J \subset I \) by \( \partial J \).
From now on, we will always assume that the IET satisfies the Keane condition:
\[
\forall_{n \geq 1} \, \forall_{\alpha, \beta \in \mathcal{A}} \quad \left( T^n(\partial I_{\alpha}) = \partial I_{\beta} \implies \partial I_\beta = 0 \right),
\]
which holds for all \( \pi \in \mathcal{S}^{\mathcal{A}} \) and $Leb$-almost all \( \lambda \in \mathbb{R}_+^{\mathcal{A}} \).
It was proved by Keane (see \cite{keane_interval_1975}) that this condition implies minimality.

Interval exchange transformations appear naturally when considering first return maps to a Poincaré section of locally Hamiltonian flows or translation flows. If one considers different classes of surface flows, such as perturbations of locally Hamiltonian flows, one obtains, as a first return map, a piecewise smooth bijection of an interval, with a finite number of discontinuities and a non-negative derivative. Such maps are called \emph{generalized interval exchange transformations} (GIETs). It is worth mentioning that when the underlying surface is a torus, then the first return map to a properly chosen Poincaré section is a circle diffeomorphism. Hence, the GIETs can be seen as generalizations of those.

The central object of this article, the \emph{affine interval exchange transformations}, or AIETs, for short, can be considered as an intermediate step between IETs and GIETs and correspond to the set of piecewise affine GIETs, linear on exchanged intervals. In contrast to the GIET case, AIETs can be parameterized using a finite-dimensional set of parameters. Namely, each AIET $f:I\to I$, which exchanges intervals $(I_{\alpha})_{\alpha\in\A}$, is given by a triple of parameters:
\begin{itemize}
\item a permutation $\pi=(\pi_t,\pi_b)$ governing the order of exchanged intervals before and after the action of $f$;
\item a length vector $\lambda\in \R_+^{\A}$, where $\sum_{\alpha\in\A}\lambda_{\alpha}=|I|$ and $\lambda_{\alpha}=|I_{\alpha}|$ for every $\alpha\in\A$, and
\item a \emph{vector of logarithm of slopes} (or shortly a \emph{log-slope vector}) $\omega\in\R^{\A}$, with $\langle\omega,\lambda\rangle=0$, where for every $\alpha\in\A$ we have $\omega_{\alpha}=\log Df|_{I_{\alpha}}$.
\end{itemize}
Hence we identify $f= (\pi,\lambda,\omega)$.

It is easy to see that the Lebesgue measure $Leb$ is invariant for every IET $T$ and it is quasi-invariant for every AIET $f$. Later, we present the relations between AIETs and IETs, namely the existence of (semi-)conjugacy between them. In particular, we discuss the relations between the invariant measures of $T$ and $f$.

\subsection{Rauzy-Veech induction and dynamical partitions for interval maps}\label{subs: RVinduction}
For \( \epsilon \in \{b,t\} \), let \( \alpha(\epsilon) = \pi_{\epsilon}^{-1}(d) \), i.e., the symbol of the rightmost subinterval before the exchange (\( \epsilon = t \)) and after the exchange (\( \epsilon = b \)).
If \( \lambda_{\alpha(b)} \neq \lambda_{\alpha(t)} \), then we say that $T= (\pi,\lambda)$ is of \emph{top type} if \( \lambda_{\alpha(b)} < \lambda_{\alpha(t)} \), and it is of \emph{bottom type} otherwise. If $(\pi,\lambda)$ is of top (resp. bottom) type, then we refer to $\alpha(t)$ (resp. $\alpha(b)$) as the \emph{winner} and to $\alpha(b)$ (resp. $\alpha(t)$) as the \emph{loser}.

Denote $I^{(0)}=I$ and let \( I^{(1)} \subset I \) be the subinterval defined as
\[
I^{(1)} = \begin{cases}
I \setminus T(I_{\alpha(b)}), & \text{if $(\pi,\lambda)$ is of top type}, \\
I \setminus I_{\alpha(t)}, & \text{if $(\pi,\lambda)$ is of bottom type},
\end{cases}
\]
and consider the first return map of \( T \) to \( I^{(1)} \). We refer to this assignment as the \emph{Rauzy-Veech induction} of \( T \) (see \cite{rauzy_echanges_1979, veech_gauss_1982}). We denote it as
\[
\mathcal{RV}(T): I^{(1)} \to I^{(1)}.
\]
It is easy to see that the result is an IET with the same number of subintervals. If $\lambda_{\alpha(b)} = \lambda_{\alpha(t)}$, then the Rauzy-Veech induction on $T$ is not defined. %

Keane's condition asserts that we can iterate this algorithm infinitely many times (see \cite{keane_interval_1975}). If $\mathcal{RV}^n(T)$ is well defined for every $n\in\N$, then we say that $T$ is \emph{infinitely renormalizable}. Under this assumption on $T$, there is a well-defined nested sequence of intervals $(I^{(n)})_{n\in\mathbb{N}}$
and of IETs $(\mathcal{RV})^n(T): I^{(n)} \to I^{(n)}$ with corresponding partitions $(I^{(n)}_{\alpha})_{\alpha\in\mathcal{A}}$. In terms of parameters, we write $(\mathcal{RV})^n(T) = (\pi^{(n)}, \lambda^{(n)})$, which, by definition, is the first return of $T$ to subinterval $I^{(n)} \subset I $. It is easy to check that the first return time of $T$ to $I^{(n)}$ is constant on each $I^{(n)}_{\alpha}$. Let us denote this return time by $q^{(n)}_\alpha$, so that $(\mathcal{RV})^n(T)_{|_{I^{(n)}_\alpha}} = T^{q^{(n)}_{\alpha}}.$

On the set $\mathcal S^\A$, we have a natural partition into minimal classes which are invariant under the induced action of $\mathcal{RV}$. We refer to them as the \emph{Rauzy classes}. For each Rauzy class $C\subset \mathcal S^{\A}$, we can construct a directed graph $\mathcal C$ in the following way. We take as vertices the elements of $C$ and then we add an arrow from $\pi$ to $\widetilde\pi$ if there exists $\lambda\in\R^{\A}$ for which $\mathcal{RV}(\pi,\lambda)=(\widetilde\pi,\widetilde\lambda)$ for some $\widetilde\lambda\in\R^{\A}$. Note that, for every $\pi\in C$, there are exactly two outgoing and two incoming arrows, given by the two possible types of $(\pi,\lambda)$.

If $T=(\pi,\lambda)$ is infinitely renormalizable, then we may assign to $T$ a path $\gamma:=\gamma(T)$ in $\mathcal {C}$ comprising of the arrows connecting the permutations $\pi=\pi^{(0)},\pi^{(1)}, \pi^{(2)}\ldots$ consecutively. We say that $\gamma$ is a \emph{combinatorial rotation number} of $T$. This naming is coherent with the classical notion of rotation number for circle homeomorphisms, where a rotation number can be interpreted as a path in the graph with one vertex and two arrows. A path $\gamma$ in $\mathcal C$ is \emph{$\infty$-complete} if every symbol $\alpha\in\A$ is a winner infinitely many times. One can show that a path $\gamma$ is $\infty$-complete iff there exists at least one IET $T$, for which $\gamma$ is a combinatorial rotation number (see Proposition 7.9 in \cite{yoccoz_interval_2010}).

Given an IET $T = (\pi, \lambda)$ satisfying Keane's condition, we associate a sequence of \emph{dynamical partitions} by
\begin{equation}\label{def:Q}
\mathcal{Q}:= \mathcal{Q}_0 = ( I_\alpha)_{\alpha \in \A}, \qquad \mathcal{Q}_n := \{ T^i(I_\alpha^{(n)}) \mid 0 \leq i < q_\alpha^{(n)}, \, \alpha \in \A \}, \quad n \geq 1.
\end{equation}

Given $x \in [0, 1)$ and $n \geq 1$, we denote by $\mathcal{Q}_n(x)$ (resp. $\mathcal{Q}(x)$) the unique element in $\mathcal{Q}_n$ (resp. $\mathcal{Q}$) containing $x$. {Note that for every $n\in \N$, the set of endpoints of the partition $\mathcal Q_n$ is contained in the bi-infinite orbits of the endpoints of the exchanged intervals, i.e. $\{T^{i}\partial I_{\alpha}\mid\alpha\in \A,\ i\in\Z\}$.}

When necessary, we will make the dependence on $T$ in the above notations explicit by writing the induction intervals as $I^{(n)}(T)$ (resp. the exchanged intervals as $I^{(n)}_\alpha(T)$) and the dynamical partitions as $\mathcal{Q}^T_n$ (resp. the elements containing a given point $x$ by $\mathcal{Q}^T_n(x)$).

By replacing in the above consideration $T$ by a GIET $f$, we obtain a Rauzy-Veech algorithm $\mathcal{RV}$ for GIETs. Analogously to the case of IET, provided that $f$ is infinitely renormalizable, we denote by $\gamma(f)$ the combinatorial rotation number of $f$ and by $(\mathcal{Q}^f_n)_{n\in\N}$ the resulting sequence of dynamical partitions.
If additionally $f=(\pi,\lambda,\omega)$ is an AIET, then $\mathcal{RV}(f)$ is also an AIET of the same number of intervals. We denote $\mathcal{RV}^n(f)=(\pi^{(n)},\lambda^{(n)},\omega^{(n)})$ for $n\in\N$.

\subsection{Cocycles}\label{subs: cocycles}

Given a vector $\omega \in \mathbb{R}^\mathcal{A}$ and an AIET $f$ with exchanged intervals $(I_\alpha(f))_{\alpha \in \mathcal{A}}$, we denote by $\phi_\omega^f: I \to \mathbb{R}$ the piecewise constant function defined by $\phi_\omega^f(x) = \omega_\alpha$, for any $x \in I_\alpha(f)$ and any $\alpha \in \mathcal{A}$. If $f=T$ is an IET and there is no risk of confusion, we will denote $\phi_\omega^T$ simply by $\phi_\omega$.

For $n \in \mathbb{N}$, we consider a sequence of vectors $(\omega^{(n)})_{n \in \mathbb{N}} \subset \mathbb{R}^{\mathcal{A}}$ defined by
\[
\omega^{(n)}_{\alpha} = S_{q^{(n)}_{\alpha}} \phi_\omega \big|_{I^{(n)}_{\alpha}},
\]
where $S_q\phi(x)$ is the Birkhoff sum $S_q\phi(x)=\sum_{0\leq k<q}\phi(T^kx)$.

This is well-defined since no discontinuity point lies in the interior of $T^i(I^{(n)}_{\alpha})$ for all $\alpha \in \mathcal{A}$ and $0\leq i< q^{(n)}_{\alpha}$.

Let us also define an invertible matrix $M(\pi, \lambda, \omega)$ using a formula that depends on the type of $(\pi, \lambda)$. If $(\pi, \lambda)$ is of top type, then

\begin{equation*}
(M(\pi,\lambda,\omega))_{\alpha,\beta}=\left\{
\begin{array}{cl}
1, & \text{ if }\alpha=\beta,\\
e^{\omega_{\alpha(b)}}, & \text{ if } \alpha= \alpha(b), \beta=\alpha(t),\\
0, & \text{ otherwise.}
\end{array}
\right.
\end{equation*}
On the other hand, if $(\pi,\lambda)$ is of bottom type, we set
\begin{equation*}(M(\pi,\lambda,\omega))_{\alpha,\beta}=\left\{
\begin{array}{cl}
1, & \text{ if }\alpha=\beta,\alpha\neq \alpha(t) \text{ or } \alpha=\alpha(t),\beta =\alpha(b),\\
e^{\omega_{\alpha(b)}}, &\text{ if } \alpha=\beta=\alpha(t),\\
0, & \text{otherwise}.
\end{array}\right.
\end{equation*}
For any $n\in \mathbb{N}$, we define product
\[M^{(n)}_{\pi,\lambda,\omega} = M(\pi^{(n-1)},\lambda^{(n-1)},\omega^{(n-1)})\cdot \ldots \cdot M(\pi^{(1)},\lambda^{(1)},v^{(1)}) \cdot M(\pi^{(0)},\lambda^{(0)},\omega^{(0)}),\]
and observe that \begin{equation}\label{formula}
\left(M^{(n)}_{\pi,\lambda,\omega}\right)_{\alpha,\beta} = \sum_{\substack{0\leq k< q^{(n)}_{\alpha}\\ T^k(I^{(n)}_{\alpha}) \subset I_{\beta}}} e^{S_k{\phi_\omega|_{I^{(n)}_{\alpha}}}}.
\end{equation}
We use convention that $M(\pi, \lambda) = M(\pi,\lambda,0)$ and $M^{(n)}_{\pi,\lambda} = M^{(n)}_{\pi,\lambda,0}$, where $0$ denotes the zero vector in $\mathbb{R}^{\mathcal{A}}$, and refer to the map $(\pi, \lambda) \mapsto M(\pi, \lambda)$ as the \emph{Rauzy-Veech cocycle}.  Notice that,
\begin{equation}\label{RVmatrix meaning}\left({M}^{(n)}_{\pi,\lambda}\right)_{\alpha,\beta} = \# \{0\leq k < q^{(n)}_\alpha \ | \ T^k(I_{\alpha}^{(n)}) \subset I_{\beta} \}.
\end{equation}
  It is well known that if $T$ is infinitely renormalizable, then the above matrix has eventually strictly positive entries, see Proposition 7.12 in \cite{yoccoz_interval_2010}. Moreover, by the construction of the sequence $(M^{(n)}_{\pi,\lambda,\omega})_{n\in\N}$, for every $n\in\N$ we have that
\begin{equation}\label{eq: allMpositive}
M^{(n)}_{\pi,\lambda,\omega}\text{ is positive (primitive)}\Leftrightarrow
M^{(n)}_{\pi,\lambda}\text{ is positive (primitive)}.
\end{equation}
Also note that by \eqref{formula} for any $\omega \in \mathbb{R}^{\mathcal{A}}$, we have
\[\left({M}^{(n)}_{\pi,\lambda, \omega}\right)_{\alpha,\beta} > 0 \quad\text{if and only if}\quad \left({M}^{(n)}_{\pi,\lambda}\right)_{\alpha,\beta} > 0.\]
Moreover, observe that
\begin{equation}\label{eq: heigthcocycle}
M^{(n)}_{\pi,\lambda} \omega = \omega^{(n)},
\end{equation}
and the matrices $M^{(n)}_{\pi,\lambda,\omega}$ satisfies the following cocycle property
\begin{equation}\label{eq: cocycleMv}
M^{(n_1+n_2)}_{\pi,\lambda,\omega}=M^{(n_2)}_{\pi^{(n_1)},\lambda^{(n_1)},\omega^{(n_1)}}\cdot M^{(n_1)}_{\pi,\lambda,\omega}.
\end{equation}

Note that the definition of the sequence $(M^{(n)}_{\pi,\lambda,\omega})$ does not depend directly on $T$, but rather on its combinatorial rotation number $\gamma=\gamma(T)$. It is noteworthy that for every $n\in\N$ it holds that
\[
\lambda^{(n)}M^{(n)}_{\pi,\lambda}=\lambda\quad\text{and}\quad q^{(n)}=M^{(n)}_{\pi,\lambda}\bar{1},
\]
where $\bar{1}$ is a vector of 1's. Moreover, if $f=(\pi,\widetilde\lambda,\omega)$ is an AIET with combinatorial rotation number $\gamma(f)=\gamma(T)$, then
\begin{equation}\label{eq: renolambda}
\widetilde\lambda^{(n)}\cdot M^{(n)}_{\pi,\lambda,\omega}=\widetilde\lambda\quad\text{and}\quad\omega^{(n)}=M^{(n)}_{\pi,\lambda}\cdot\omega.
\end{equation}

\subsection{Semiconjugacy and dynamical partitions for $\infty$-complete GIETs}
\label{sc:semiconjugacy}
As already mentioned, a path $\gamma$ in a Rauzy graph $\mathcal C$ is $\infty$-complete if and only if there exists an IET $T$, whose combinatorial rotation number is $\gamma$. Moreover,
this correspondence is one-to-one when restricted to uniquely ergodic IETs (see Corollary 5 and Proposition 6 in \cite{yoccoz_echanges_2005}). By \cite[Proposition 7]{yoccoz_echanges_2005}, any infinitely renormalizable GIET whose combinatorial rotation number is $\infty$-complete is semi-conjugated (via a non-decreasing surjective map) to an IET with the same combinatorial rotation number.

Let us point out that this semi-conjugacy is a conjugacy if and only if the GIET is minimal. Moreover, in this case, the conjugacy is unique. Indeed, this follows from the fact that any infinitely renormalizable IET is minimal and that the conjugacy sends the endpoints of the partitions obtained by the Rauzy-Veech induction of one map (which form a dense set in the interval) to the endpoints associated with the other.

On the other hand, in contrast to IETs, not every GIET is minimal. In fact, it is known that large classes of infinitely renormalizable GIETs are not minimal and admit so-called \emph{wandering intervals}, i.e., an interval that is disjoint from all its backward and forward iterates (see, e.g., \cite{bressaud_persistence_2010,marmi_affine_2010}). On the other hand, for AIETs, this semi-conjugacy is known to be a conjugacy in many contexts. For example, see e.g., \cite{cobo_piece-wise_2002, bressaud_persistence_2010}, we have that if $T$ is self-similar
\begin{equation}\label{eq: isconjugated}
	\begin{split}
	&\text{$f\in\operatorname{Aff}(T,\omega)$ is conjugated to $T$ if the log-slope vector $\omega$}\\
	&\text{ belongs to the associated central-stable space of $M(T)$.}
	\end{split}
	\end{equation}
The same result also holds for almost every IET (see \cite{trujillo_affine_2024}). Let us mention that in the latter case, central-stable refers to the Oseledets splitting associated with the Kontsevich-Zorich cocycle (see, e.g., \cite{zorich_finite_1996}).

In the following, we will say that an infinitely renormalizable GIET is $\infty$-complete if its rotation number is $\infty$-complete. Moreover, in the case of AIETs, we denote by $\operatorname{Aff}(T,\omega)$ the family of affine interval exchange transformations semi-conjugated to an IET $T= (\pi,\lambda)$, with the log-slope vector $\omega$. It follows from \cite[Proposition 2.3]{marmi_affine_2010} that $\operatorname{Aff}(T,\omega) \neq \emptyset$ if and only if $\langle \lambda, \omega \rangle = 0$.

Let $f$ be an $\infty$-complete GIET. Suppose that $T$ is the only uniquely ergodic IET to which $f$ is semi-conjugated. Denote by $h$ the non-increasing surjective semi-conjugacy satisfying $h \circ f = T \circ h$ described above. Notice that, in this case, $f$ is also uniquely ergodic and its unique invariant probability measure $\mu_f$ satisfies
\begin{equation}
\label{eq:projected_invariant_measure}
Leb = h_\ast \mu_f.
\end{equation}

Since $h$ is non-increasing, the preimage by $h$ of any point is either a point or an interval. Moreover, since $T$ is minimal, any wandering interval of $f$ is mapped by $h$ to a point. Let
\[
W_f:= \left\{ x \in [0, 1) \mid \exists \epsilon > 0 \text{ s.t. } h((\max\{0, x - \epsilon\}, \min\{1, x - \epsilon)) \text{ is a point} \right\}.
\]
Notice that $W_f$ is an open set in $[0, 1)$ whose connected components are the wandering intervals of $f$. Denote by $\textup{End}(W_f)$ the set of endpoints of $W_f$ and let $W_f^+ = W_f \cup \textup{End}(W_f)$. We have
\begin{equation}
\label{eq:inv_measure_wandering}
\mu_f(W_f^+) = 0.
\end{equation}
For any $x \in [0, 1)$, let us denote
\[w(x) = h^{-1}(h(x)),\]
that is, $w(x)$ is either $\{x\}$ if $x \notin W_f^+$ or is the unique connected component $W_f^+$ containing $x$.

By \cite[Proposition 7]{yoccoz_echanges_2005}, for any $n \geq 0$, the semi-conjugacy $h$ maps the endpoints of $\mathcal{Q}_n^f$ to the endpoints of $\mathcal{Q}_n^T$ while respecting their order. Hence, the dynamical partitions of $T$ and $f$ satisfy
\begin{equation}
\label{eq:semiconjugacy_partition}
h^{-1}(\mathcal{Q}_n^T(h(x))) = \left(w(l_{n, x}^f) \cup \mathcal{Q}_n^f(x)\right) \setminus w(r_{n, x}^f), \quad \text{ for $x \in [0, 1)$ and $n \geq 0$,}
\end{equation}
where $\mathcal{Q}_n^f(x) = [l_{n, x}^f, r_{n, x}^f).$ In particular,
\begin{equation} \label{eq:conjugacy_partition}
h^{-1}(\mathcal{Q}_n^T(h(x))) = \mathcal{Q}_n^f(x), \qquad \text{ if } l_{n, x}^f, r_{n, x}^f \notin W_f^+.\end{equation}
Furthermore, since the $f$-invariant measure $\mu_f$ is non-atomic and gives zero weight to the set of wandering intervals (in particular, satisfies \eqref{eq:inv_measure_wandering}), by \eqref{eq:projected_invariant_measure} and \eqref{eq:semiconjugacy_partition}, we have
\begin{equation}
\label{eq:invariant_lebesgue_partition}
Leb(\mathcal{Q}_n^T(h(x))) = \mu_f( \mathcal{Q}_n^f(x)), \qquad \text{ for $x \in [0, 1)$ and $n \geq 0$.}
\end{equation}

\subsection{Self-similar IETs and periodic type GIETs}
As already mentioned in the introduction, in this article, we narrow down our interest to a very specific class of combinatorial rotation numbers, which we are going to introduce in this section.

First, let us recall that an IET $T=(\pi,\lambda)$ is said to be \emph{self-similar} if its combinatorial rotation number $\gamma(\pi,\lambda)$ is an $\infty$-complete periodic path in the corresponding Rauzy graph. It is worth mentioning that in some literature, self-similar IETs are those with pre-periodic combinatorial rotation numbers, rather than periodic. If $T$ is self-similar, then there exists $N(\pi,\lambda)=N\in\N$ such that for every $k\in\N$ we have
\[
M_{\pi,\lambda}^{(N)}=M_{\pi^{(kN)},\lambda^{(kN)}}^{(N)}.
\]
We denote $M=M(\pi,\lambda)=M(T):=M_{\pi,\lambda}^{(N)}$ and we call the matrix $M$ a \emph{self-similarity matrix} of $T$. Since $\gamma$ is $\infty$-complete, the matrix $M$ is primitive, i.e., it is positive, up to taking some power. Hence, without the loss of generality, we may assume that the matrix $M$ is positive. In particular, the IET $T$ is uniquely ergodic (see \cite{veech_interval_1978}). Moreover, if $\rho_T$ is the logarithm of the Perron-Frobenius eigenvalue of $M$, the length vector $\lambda$ satisfies $\lambda M=e^{\rho_T} \lambda,$
i.e.,\ $\lambda$ is a left Perron-Frobenius eigenvector of $M$. In particular, $\lambda^{(kN)} = e^{-k\rho_T} \lambda,$ for any $k \geq 0$.

If a GIET (in particular, AIET) $f$ has a periodic $\infty$-complete combinatorial rotation number, then we say that $f$ is of \emph{periodic type}.

\begin{lemma}
\label{lem:uniformlimit}
Let $f$ be a periodic type GIET with period $N$ and let $\mu_f$ be its unique invariant probability measure. Let $T = (\pi, \lambda)$ be the unique self-similar IET to which $f$ is semi-conjugated, and let us denote by $\rho_T$ the logarithm of the Perron-Frobenius eigenvalue of the associated self-similarity matrix. Then

\begin{equation*}
\label{eq:uniformlimitGIET}
\lim_{k\to\infty}-\frac{1}{k}\log\mu_f\big(\mathcal Q^f_{kN}(x)\big) = \rho_T, \qquad \text{uniformly in $x \in [0,1)$.}
\end{equation*}
\end{lemma}
\begin{proof}
By \eqref{eq:invariant_lebesgue_partition}, it suffices to show that
\begin{equation}
\label{eq:uniformlimitIET}
\lim_{k\to\infty}-\frac{1}{k}\log Leb \big(\mathcal Q^T_{kN}(x)\big) = \rho_T, \qquad \text{uniformly in $x \in [0,1)$.}
\end{equation}
Fix $k \geq 1$. Recalling that $\lambda^{(kN)} = e^{-k\rho_T} \lambda$ and since $T$ is an IET, we have
\[ e^{-k\rho_T}\min_{\alpha \in \A} \lambda_\alpha = \min_{\alpha \in \A} \lambda^{(kN)}_\alpha\leq Leb \big(\mathcal Q^T_{kN}(x)\big) \leq \max_{\alpha \in \A} \lambda^{(kN)}_\alpha = e^{-k\rho_T}\max_{\alpha \in \A} \lambda_\alpha,\]
for any $x \in [0, 1).$ Equation \eqref{eq:uniformlimitIET} now follows easily from the inequality above.
\end{proof}
\subsection{Hyperbolic periodic type}

We will now define precisely the family of interval maps to which our results will apply, namely, those of hyperbolic periodic type.

Recall that every IET $T$ can be seen as a first return map of a translation flow on some translation surface of genus $g$, and the flow has exactly $\kappa$ singularities. Then, under the assumption that the flow does not have saddle connections (which translates into $T$ satisfying Keane's condition), $T$ is an IET of $d=2g+\kappa-1$ intervals.

In this case, we say that $T$ is of \emph{hyperbolic periodic type}, or \emph{hyperbolically self-similar}, if it is a self-similar IET and the associated self-similarity matrix $M$ has exactly $g$ distinct real eigenvalues of modulus larger than $1$, $g$ distinct real eigenvalues of modulus smaller than $1$, and $\kappa-1$ unit eigenvectors.

If a GIET (in particular AIET) $f$ has a combinatorial rotation number equal to $\gamma(T)$, with $T$ being hyperbolically self-similar, then we say that $f$ is of \emph{hyperbolic periodic type}.

Note that if $T$ is hyperbolically self-similar, then $\R^{\A}$ has a basis made from eigenvectors of $M(T)$: $g$ expanding, $g$ contracting and $\kappa-1$ fixed.

Let $\omega\in\R^{\A}$ be such that $\langle\lambda,\omega\rangle=0$. We say that $\omega$ is of
\begin{itemize}
\item \emph{unstable type}, if in the basis decomposition it has at least one expanding (non-Perron-Frobenius) eigenvector,
\item \emph{central-stable type}, if in the basis decomposition it has at least one fixed but no expanding eigenvectors.
\item \emph{stable type}, if in the basis decomposition it has only contracting eigenvectors.
\end{itemize}

\subsection{Conformal measures} Given an invertible Borel map $T:X\to X$ on a topological space $X$, we consider \emph{quasi-invariant} measures, i.e., probability Borel measures $\nu$ on $X$ such that $T_*\nu\sim\nu$. For a given Borel function $\phi:X\to \R$, we say that a quasi-invariant probability measure $\nu$ on $X$ is \emph{$\phi$-conformal} for $T$, if
\[\frac{d(T^{-1})_*\nu}{d\nu}=e^{\phi}.\]
In particular, if $\phi=0$, then $\nu$ is $T$-invariant. We refer to $\phi$ as the \emph{potential} of $\nu$.

The notion of ergodicity naturally extends into the family of quasi-invariant measures, namely, as is the case for the invariant measures, a conformal measure $\nu$ is ergodic if for any measurable set $A$, we have that $\nu(A\triangle T^{-1}(A))=0$ implies $\nu(A)=0$ or $\nu(A)=1$. It is easy to see that for a fixed potential $\phi$, the $\phi$-conformal measures form a simplex, just as in the case of invariant measures. Analogously to the invariant case, one can show that the ergodic measures are the extreme points of this simplex.

We have the following folklore result, which justifies the existence of conformal measures, for homeomorphisms on a compact space and continuous potentials.

\begin{proposition}
\label{prop:conformal_measures_general}
Assume that $T:X\to X$ is a uniquely ergodic homeomorphism of a compact space $X$, and denote its unique invariant probability by $\mu$. Let $\phi$ be a continuous potential satisfying $\int_{X} \phi\mbox{ } d\mu = 0$. Then there exists a $ \phi$-conformal measure.
\end{proposition}
\begin{proof}
Let us define an operator $\Phi$ acting on the space $\mathcal{P}(X)$ of probability measures on $X$
by the formula
\begin{equation*}
\Phi(\nu) = \frac{e^{-\phi}dT^{-1}_{*}\nu}{\int_{X} e^{-\phi}\ dT^{-1}_{*}\nu}.
\end{equation*}
As $T$ and $\phi$ are continuous, $\Phi$ is continuous in the weak-* topology on $\mathcal{P}(X)$. By Schauder-Tychonoff fixed point theorem, we obtain a measure $\mu\in \mathcal{P}(X)$ satisfying
\begin{equation}\label{eq: almostconformal}
\nu = \frac{e^{-\phi} T^{-1}_{*}{\nu}}{\int_{X} e^{-\phi}\ dT^{-1}_{*}\nu},
\end{equation}
which after taking $\Delta: = \log (\int_{X} e^{-\phi}\ dT^{-1}_{*}\nu)$, is equivalent to
\begin{equation*}
\frac{dT_*^{-1}\nu}{d\nu} = e^{\phi+\Delta}.
\end{equation*}By induction, for every $n \in \mathbb{N}$
\begin{equation}\label{formula_in_proposition}
\frac{dT_*^{-n}\nu}{d\nu} = e^{S_n(\phi+\Delta)}.
\end{equation}
Since $T$ is uniquely ergodic and $\phi$ has integral $0$, $\frac{1}{n}S_n(\phi +\Delta) $ converges uniformly to the constant $\Delta$.
Consequently, for any $\varepsilon > 0 $, we have
\begin{equation*}
e^{n(\Delta-\varepsilon)} < \int_{X} e^{S_n(\phi+\Delta)} d\mu < e^{n(\Delta+\varepsilon)},
\end{equation*}
for large enough $n$.
By \eqref{formula_in_proposition}, the middle term is equal to $1$, which gives $
\Delta = 0 $. Thus, by \eqref{eq: almostconformal}, the measure $\nu$ is $\phi$-conformal.
\end{proof}

\subsection{Conformal measures for IETs and central-stable vectors}\label{subs: conformal}
Notice that we cannot directly apply Proposition~\ref{prop:conformal_measures_general} to guarantee the existence of conformal measures in the IET setting, since these transformations are neither continuous nor defined on a compact space.

However, given a self-similar IET $T$ of hyperbolic type on $[0, 1)$ with exchanged intervals $(I_\alpha)_{\alpha \in \A}$ and any $\omega \in \R^\A$ of central-stable type, we can construct a $\phi_\omega^T$-conformal measure $\nu_\omega$, where $\phi_\omega^T:I\to \R$ is the function defined in Section~\ref{subs: cocycles}, as follows.

Recall that in this case $\textup{Aff}(T, \omega) \neq \emptyset$ and that any $f \in \textup{Aff}(T, \omega)$ is homeomorphically conjugated to $T$ (see \eqref{eq: isconjugated}). %
Let $h$ be a homeomorphism conjugating $f$ and $T$ satisfying $T \circ h = h \circ f$. Then
\begin{equation}
\label{eq:projected_Lebesgue_measure 0}
\nu_\omega := h_*Leb,
\end{equation}
is a $\phi_\omega^T$-conformal measure of $T$.
Indeed, by \eqref{eq:semiconjugacy_partition} {and the absence of wandering intervals}, if $(I_\alpha(f))_{\alpha \in \A}$ denote the intervals exchanged by $f$, then $h(I_\alpha(f)) = I_\alpha$, for any $\alpha \in \A$, that is, $h$ sends the intervals exchanged by $f$ to the intervals exchanged by $T$. In fact, we have
\begin{equation}\label{eq: intervals f T}
h(I^{(n)}_\alpha(f)) = I^{(n)}_\alpha\quad\text{for all}\quad \alpha\in\mathcal A,\quad n\geq 1.
\end{equation}
Therefore, $\phi_\omega^f = \phi^T_\omega \circ h$. As $T^{-1} \circ h = h \circ f^{-1}$, this gives
\begin{align}\label{eq: nu omega}
T^{-1}_*(h_*Leb) = h_*(f^{-1}_*Leb) = h_*(e^{\phi_\omega^f} Leb ) = h_*(e^{\phi_\omega \circ h} Leb ) = e^{\phi_\omega} h_\ast Leb.
\end{align}
As we shall see later in Proposition~\ref{prop:conformal_measures}, this $\phi_\omega^T$-conformal measure is in fact unique.

\subsection{Reducing central-stable to central.}
Throughout this paper, when we consider either log-slope vectors or vectors of potential, which are of central-stable type, we often want to reduce the study to the case of purely central (invariant) vectors.
In this section, we prove a lemma, which allows that.

Let $T=(\pi,\lambda)$ be a hyperbolically self-similar IET of period $N$ with a self-similarity matrix $M$ and let $\mathcal P_T^{(n)}:=\mathcal Q_{n\cdot N}^{T}$, for every $n\in\N$. Let also for every $x\in[0,1)$, $P_T^{(n)}(x)$ denote the unique element of $\mathcal P_T^{(n)}$, which contains $x$.

Let $\omega\in\R^{\A}$ be a vector of central-stable type and consider $f\in \textup{Aff}(T,\omega)$. Then there exists a homeomorphism $h:[0,1)\to [0,1)$ which conjugates $f$ and $T$. Let $\mathcal P_f^{(n)}:=\mathcal Q_{n\cdot N}^{f}$, for every $n\in\N$. Let also for every $x\in[0,1)$, $P_f^{(n)}(x)$ denote the unique element of $\mathcal P_f^{(n)}$, which contains $x$. Then $h(P_f^{(n)}(x))=P_T^{(n)}(h(x))$.

We have the following result.

\begin{lemma}\label{lem:from_cs_to_c}
Let $T=(\pi,\lambda)$ be a hyperbolically self-similar IET of period $N$. Let $\omega\in\R^{\A}$ be a vector of central-stable type. Consider the unique $\omega$-conformal measure $\nu_\omega$. Let $\omega=\omega_c+\omega_s$ be the decomposition of $\omega$ into the central (invariant) and stable vector, respectively, and let $\nu_{\omega_c}$ be the unique $\omega_c$-conformal measure.
Then,
\begin{equation}\label{eq: Rnuvc}
R_*(\nu_{\omega_c}|_{R^{-1}I})=e^{-\rho_{\nu_{\omega_c}}}\nu_{\omega_c},
\end{equation}
where $R:[0,e^{-\rho_T})\to[0,1)$ is the linear rescaling given by $R(x)=e^{\rho_T}x$, $\rho_{\nu_{\omega_c}}$ is the logarithm of the Perron-Frobenius eigenvalue of the matrix $M^{(N)}_{\pi,\lambda,\omega_c}$, and the vector $(\nu_{\omega_c}(I_\alpha))_{\alpha\in\mathcal A}$ is a left Perron-Frobenius eigenvector of $M^{(N)}_{\pi,\lambda,\omega_c}$.

Moreover,
for every $x\in [0,1)$, we have
\begin{equation}\label{eq:partitiopassintocentral}
\lim_{n\to\infty}\frac{\log \nu_\omega(P_T^{(n)}{(x)})}{\log \nu_{\omega_c}(P_T^{(n)}{(x)})}=1,
\end{equation}
and there exists a constant $C>1$ such that for every $n\geq 0$ and $x\in[0,1)$, we have
\begin{equation}\label{eq:baseshaveboundedquotients}
\frac{\nu_\omega(P_T^{(n)}{(x)})}{\nu_\omega(P_T^{(n+1)}{(x)})}\leq C.
\end{equation}
If $f\in\textup{Aff}(T,\omega)$, then for every $n\geq 0$ and $x\in[0,1)$, we have
\begin{equation}\label{eq:baseshaveboundedquotients_affine}
\frac{Leb(P_f^{(n)}{(x)})}{Leb(P_f^{(n+1)}{(x)})}\leq C.
\end{equation}
\end{lemma}
\begin{proof}
Equality \eqref{eq:partitiopassintocentral} is a direct consequence of Corollary~\ref{cor: bded_density}. Indeed, it follows from the existence of $D>1$ such that
\begin{equation}\label{eq: lipschitz_relation}
 D^{-1}\le \frac{\nu_\omega(J)}{\nu_{\omega_c}(J)}\le D\qquad\text{for every}\quad J\in\mathcal P^{(n)}_T,
\end{equation}
and from the continuity of both measures proved by Proposition~\ref{prop:conformal_measures}.

To show \eqref{eq:baseshaveboundedquotients}, in view of \eqref{eq: lipschitz_relation}, it is enough to prove that there exists $C>1$ such that for every $n\geq 0$ and $x\in [0,1)$, we have
\begin{equation}\label{eq:baseshaveboundedquotients_2}
\frac{\nu_{\omega_c}(P_T^{(n)}{(x)})}{\nu_{\omega_c}(P_T^{(n+1)}{(x)})}\leq C.
\end{equation}
As $P_T^{(n)}{(x)}=T^j I^{(nN)}_\alpha$ for some $\alpha\in\mathcal A$ and $0\leq j<q^{(nN)}_\alpha$, and $P_T^{((n+1)N)}{(x)}\subset P_T^{(n)}{(x)}$, we have $P_T^{(n+1)}{(x)}=T^j(\mathcal{RV}^{nN}(T))^i(I^{((n+1)N)}_\beta)$
for some $\beta\in\mathcal A$ and $0\leq j<q^{(N)}_\beta$. As $\nu_{\omega_c}$ is a $\phi^T_{\omega_c}$-conformal measure, we have
\begin{equation}\label{eq: Pn+1 to n}
\frac{\nu_{\omega_c}(P_T^{(n)}{(x)})}{\nu_{\omega_c}(P_T^{(n+1)}{(x)})}=\frac{\nu_{\omega_c}(T^j I^{(nN)}_\alpha)}{\nu_{\omega_c}(T^j(\mathcal{RV}^{nN}(T))^iI^{((n+1)N)}_\beta)}=\frac{\nu_{\omega_c}(I^{(nN)}_\alpha)}{\nu_{\omega_c}(\mathcal{RV}^{nN}(T)^iI^{((n+1)N)}_\beta)}.
\end{equation}
Let us consider the linear rescaling $R:[0,e^{-\rho_T})\to[0,1)$ given by $R(x)=e^{\rho_T}x$. By the self-similarity of $T$, the map $R^n:I^{(nN)}\to I^{(0)}$ is a conjugaty between $\mathcal{RV}^{nN}(T)$ and $T$.

As the vector $\omega_c$ is invariant for $M^n$, the measure $\nu_{\omega_c}|_{R^{-n}I}$ restricted to $R^{-n}I=[0,e^{-n\rho_T})$ is a conformal measure for the induced IET $\mathcal{RV}^{nN}(T)$, where the potential is determined by the vector $\omega_c^{(nN)}=\omega_c M^n=\omega_c$. It follows that the image $R^n_*(\nu_{\omega_c}|_{R^{-n}I})$ is a $\phi^T_{\omega_c}$-conformal measure. By the uniqueness of the probability $\phi^T_{\omega_c}$-conformal measure (see Proposition~\ref{prop:conformal_measures}), we have
\begin{equation}\label{eq: trannuvc}
R^n_*(\nu_{\omega_c}|_{R^{-n}I})=\nu_{\omega_c}({R^{-n}I})\nu_{\omega_c}.
\end{equation}
By the definition of the matrices $M^{(n)}_{\pi,\lambda,\omega_c}$, the $\phi^T_{\omega_c}$-conformal measure $\nu_{\omega_c}$ is such that, for any $n\geq 1$, we have
\begin{equation}\label{eq: reno nu}
(\nu_{\omega_c}(I_\alpha))_{\alpha\in\mathcal A}=(\nu_{\omega_c}(I^{(nN)}_\alpha))_{\alpha\in\mathcal A}\cdot M^{(nN)}_{\pi,\lambda,\omega_c}.
\end{equation}
Note that this formula is true for any vector $\omega\in\R^{\mathcal A}$ and follows from the same arguments as \eqref{eq: renolambda}.

As $(\pi,\lambda)$ is self-similar and $\omega_c$ is a fixed vector for its self-similarity matrix $M$, by \eqref{eq: heigthcocycle} and \eqref{eq: cocycleMv}, we have
\begin{equation*}\label{eq: MnN}
M^{(nN)}_{\pi,\lambda,\omega_c}=(M^{(N)}_{\pi,\lambda,\omega_c})^n\quad\text{for any}\quad n\geq 1.
\end{equation*}
In view of \eqref{eq: trannuvc} and \eqref{eq: reno nu}, this gives
\begin{equation*}\label{eq: reno nu center}
(\nu_{\omega_c}(I_\alpha))_{\alpha\in\mathcal A}=(\nu_{\omega_c}(R^{-n}I^{(N)}_\alpha))_{\alpha\in\mathcal A}\cdot (M^{(N)}_{\pi,\lambda,\omega_c})^n=\nu_{\omega_c}(R^{-n}I)(\nu_{\omega_c}(I_\alpha))_{\alpha\in\mathcal A}\cdot (M^{(N)}_{\pi,\lambda,\omega_c})^n.
\end{equation*}
It follows that $(\nu_{\omega_c}(I_\alpha))_{\alpha\in\mathcal A}$ is a left Perron-Frobenius eigenvector for $M^{(N)}_{\pi,\lambda,\omega_c}$ and $\nu_{\omega_c}(R^{-n}I)=e^{-n\rho_{\nu_{\omega_c}}}$. In particular, by \eqref{eq: trannuvc}, we have \eqref{eq: Rnuvc}.
Since $R^n\circ\mathcal{RV}^{nN}(T)=T\circ R^n$, by \eqref{eq: Rnuvc}, we obtain that
\begin{align*}
\nu_{\omega_c}(I^{(nN)}_\alpha)&=\nu_{\omega_c}(R^{-n}I_\alpha)=e^{-n\rho_{\nu_{\omega_c}}}\nu_{\omega_c}(I_\alpha),\\
\nu_{\omega_c}(\mathcal{RV}^{nN}(T)^iI^{((n+1)N)}_\beta)&=\nu_{\omega_c}(\mathcal{RV}^{nN}(T)^iR^{-n}I^{(N)}_\beta)=\nu_{\omega_c}(R^{-n}T^iI^{(N)}_\beta)\\
&=e^{-n\rho_{\nu_{\omega_c}}}\nu_{\omega_c}(T^iI^{(N)}_\beta).
\end{align*}
Together with \eqref{eq: Pn+1 to n}, this gives
\[\frac{\nu_{\omega_c}(P_T^{(n)}{(x)})}{\nu_{\omega_c}(P_T^{(n+1)}{(x)})}=\frac{\nu_{\omega_c}(I_\alpha)}{\nu_{\omega_c}(T^iI^{(N)}_\beta)}.\]
It follows that \eqref{eq:baseshaveboundedquotients_2} holds with
\[C:=\max\big\{\nu_{\omega_c}(T^iI^{(N)}_\beta)^{-1}\mid \beta\in\mathcal A,\, 0\leq j<q^{(N)}_\beta\big\}.\]

Finally, note that since $\omega$ is of central-stable type, by \eqref{eq: isconjugated}, the AIET $f\in \textup{Aff}(T,\omega)$ is conjugated to $T$ via a homeomorphism $h$ with $h(P_f^{(n)}(x))=P_T^{(n)}(h(x))$.
Thus \eqref{eq:baseshaveboundedquotients_affine} follows directly from \eqref{eq:baseshaveboundedquotients} and \eqref{eq:projected_Lebesgue_measure 0}.
This finishes the proof.
\end{proof}

\section{From information content to Hausdorff dimension}\label{sec: information content}
Let $(\mathcal{P}^{(n)})_{n\geq 0}$ be a sequence of finite partitions of $[0,1)$ into right open subintervals such that $\mathcal{P}^{(n+1)}$ is finer then $\mathcal{P}^{(n)}$, for every $n \geq 0$. For every $x\in[0,1)$ and $n\geq 0$ denote by $P^{(n)}(x)$ the unique element of $\mathcal{P}^{(n)}$ containing $x$. Denote by $\partial \Delta$ the endpoints of an interval $\Delta \subseteq [0, 1)$, and by $E^{(n)}$ the set of endpoints of the partition $\mathcal{P}^{(n)}$, for any $n \geq 0$.

Suppose that $\mu,\nu$ and $m$ are three probability Borel measures on $[0,1)$ such that
\begin{enumerate}[(a)]
\item \label{cond:mu}$-\lim_{n \to \infty} \frac{\log \mu(P^{(n)}(x))}{n} = a \geq 0$, for $m$-a.e. $x \in [0, 1)$;
\item \label{cond:leb} $-\lim_{n \to \infty} \frac{\log \nu(P^{(n)}(x))}{n} = b > 0$, for $m$-a.e. $x \in [0, 1)$;
\item \label{cond:exp_refining} there exists $ C>1$ such that for any $n\geq 0$, any $\Delta \in \mathcal{P}^{(n)}$ and any $\Delta_1, \Delta_2 \in \mathcal{P}^{(n + 1)}$ satisfying $\Delta_1, \Delta_2 \subseteq \Delta$,
\[ \max\left\{ \frac{\mu(\Delta_1)}{\mu(\Delta_2)},
\frac{\mu(\Delta)}{\mu(\Delta_1)}, \frac{\nu(\Delta_1)}{\nu(\Delta_2)},
\frac{\nu(\Delta)}{\nu(\Delta_1)}, \frac{m(\Delta)}{m(\Delta_1)} \right\} \leq C; \]

\item \label{cond:ae_generating} $\bigcap_{n\geq 1}P^{(n)}(x)=\{x\},$ for $m$-a.e. $x\in [0,1).$
\end{enumerate}

Moreover, let us assume that the nested sequence of partitions satisfies the following.

\begin{enumerate}[(a)]
\setcounter{enumi}{4}
\item\label{cond:three_elements} For any $n \geq 0$, every interval in $\mathcal{P}^{(n)}$ is the union of at least $3$ distinct elements of $\mathcal{P}^{(n+1)}$,
\end{enumerate}

\begin{lemma}
\label{lem:adj_comparable}
Fix $1 < \delta < C$. For $m$-a.e. $x \in [0, 1)$ there exists $N_x \in \N$ such that for any $ n \geq N_x$ and any $J \in \mathcal{P}^{(n)}$ adjacent to $P^{(n)}(x)$ we have
\begin{equation}
 \label{eq:comparability}
 \delta^{-n} \leq \frac{\mu(P^{(n)}(x))}{\mu(J)}, \frac{\nu(P^{(n)}(x))}{\nu(J)} \leq \delta^n.
\end{equation}
\end{lemma}

\begin{proof}
We write the proof only for $\mu$ since the same proof holds for $\nu$. Let $n > k \geq 0$. Notice that if $\Delta, \Delta' \in \mathcal{P}^{(n)}$ are adjacent and $\partial \Delta \cap E^{(n - 1)} = \emptyset$, then $\Delta, \Delta'$ are contained in the same element of the partition $\mathcal{P}^{(n - 1)}$, and by Condition \eqref{cond:exp_refining},
\[C^{-1} \leq \frac{\mu(\Delta')}{\mu(\Delta)} \leq C.\]
Moreover, since each element in $\mathcal{P}^{(n - 1)}$ is a union of at least $3$ distinct elements of $\mathcal{P}^{(n)}$, by Condition \eqref{cond:exp_refining} it follows that,
\[m\left(\bigcup \left\{ \Delta \in \mathcal{P}^{(n)} \mid \partial \Delta \cap E^{(n - 1)} \neq \emptyset \right\} \right) \leq 1 - C^{-1}.\]
Iterating this argument, if $\Delta, \Delta' \in \mathcal{P}^{(n)}$ are adjacent and $\partial \Delta \cap E^{(n - k)} = \emptyset$, then
\begin{equation}\label{eq: partitionratio}
 C^{-k} \leq \frac{\mu(\Delta')}{\mu(\Delta)} \leq C^k,
\end{equation}
and
\[m\left(\bigcup \left\{ \Delta \in \mathcal{P}^{(n)} \mid \partial \Delta \cap E^{(n - k)} \neq \emptyset \right\} \right) \leq (1 - C^{-1})^k.\]
Let
\[ A_n:= \bigcup \left\{ \Delta \in \mathcal{P}^{(n)} \mid \partial \Delta \cap E^{(n - r_n)} \neq \emptyset \right\},\text{ where } r_n:= \left\lfloor \frac{n \log \delta}{\log C} \right\rfloor.\]
Since $m(A_n) \leq (1 - C^{-1})^{\left\lfloor \tfrac{n\log \delta}{\log C} \right\rfloor}$, we have $\sum_{n \geq 1} m(A_n) < +\infty$. By the Borel-Cantelli lemma, for $m$-a.e.\ $x\in[0,1)$ there exists $N_x\in \N$ such that for any $n\geq N_x$ we have $\partial P^{(n)}(x)\cap E^{(n - r_n)} = \emptyset$. Then for any $J \in \mathcal{P}^{(n)}$ adjacent to $P^{(n)}(x)$, by taking $\Delta=J$, $\Delta'=P^{(n)}(x)$ and $k=r_n$ in \eqref{eq: partitionratio}, we have
\[\delta^{-n}\leq C^{-r_n} \leq \frac{\mu(P^{(n)}(x))}{\mu(J)} \leq C^{r_n}\leq \delta^n,\]
which completes the proof.
\end{proof}

\begin{theorem}
\label{thm:HD_formula}
For $m$-a.e.\ $x\in [0,1)$, we have
\begin{equation}\label{eq:a/b}
 \lim_{\epsilon\to 0}\frac{\log \mu(x-\epsilon,x+\epsilon)}{\log \nu(x-\epsilon,x+\epsilon)}=\frac{a}{b}.
\end{equation}
In particular, if $m = \mu$ and $\nu = \Leb$, then $\dim_H(\mu) = \tfrac{a}{b}$.
\end{theorem}
\begin{proof}
Fix $1 < \delta < C$. Let $x \in (0, 1)$ satisfying Condition \eqref{cond:ae_generating} and such that the conclusions of Lemma~\ref{lem:adj_comparable} hold. Notice that these are full $m$-measure conditions, so it suffices to show that \eqref{eq:a/b} holds for such a point.

Let $0 < \epsilon < \min\{x,1-x\}$ and define
\[n_\epsilon:= \max \left\{ n \geq 0 \,\big|\, |(x - \epsilon, x + \epsilon)\cap E^{(n)}|\leq 1 \right\}.\]
Notice that, in view of Condition \eqref{cond:ae_generating}, $n_\epsilon$ is finite.
Then, there exist $J_\epsilon \in \mathcal{P}^{(n_\epsilon)}$ adjacent to $P^{(n_\epsilon)}(x)$ and $\Delta_\epsilon \in \mathcal{P}^{(n_\epsilon+1)}$ (equal to or adjacent to $P^{(n_\epsilon+1)}(x)$) such that
\[\Delta_\epsilon \subseteq (x - \epsilon, x + \epsilon) \subseteq J_\epsilon \cup P^{(n_\epsilon)}(x),\] and either
\[\Delta_\epsilon \subseteq J_\epsilon\quad \text{ or } \quad \Delta_\epsilon \subseteq P^{(n_\epsilon)}(x).\]
Since $n_\epsilon \to \infty$ as $\epsilon \to 0$, there exists $\epsilon_x>0$ such that $n_\epsilon\geq N_x$ for any $0<\epsilon<\epsilon_x$.
By \eqref{eq:comparability}, for any $0<\epsilon<\epsilon_x$,
\begin{align*}
 \delta^{-n_\epsilon-1} \nu(P^{(n_\epsilon+1)}(x)) &\leq \nu(\Delta_\epsilon) \leq \nu(x- \epsilon, x + \epsilon)\\
 & \leq \nu(J_\epsilon) + \nu(P^{(n_\epsilon)}(x)) \leq 2\delta^{n_\epsilon} \nu(P^{(n_\epsilon)}(x)), \end{align*}
and, similarly,
\begin{align*}
 \delta^{-n_\epsilon-1} \mu(P^{(n_\epsilon+1)}(x)) &\leq \mu(\Delta_\epsilon) \leq \mu(x- \epsilon, x + \epsilon) \\
 &\leq \mu(J_\epsilon) + \mu(P^{(n_\epsilon)}(x)) \leq 2\delta^{n_\epsilon} \mu(P^{(n_\epsilon)}(x)).
\end{align*}
Therefore
\begin{gather*}
 \frac{-\log \mu(P^{(n_\epsilon)}(x)) - n_\epsilon\log\delta-\log 2}{-\log \nu(P^{(n_\epsilon+1)}(x)) + (n_\epsilon + 1)\log\delta}
 \leq
 \frac{\log \mu(x - \epsilon, x + \epsilon)}{\log \nu(x - \epsilon, x + \epsilon)} \\
 \leq \frac{ -\log \mu(P^{(n_\epsilon+1)}(x)) + (n_\epsilon + 1)\log\delta}{-\log \nu(P^{(n_\epsilon)}(x)) - n_\epsilon \log\delta-\log 2}.
\end{gather*}
Recalling that $n_\epsilon \to \infty$ as $\epsilon \to 0$, it follows from \eqref{cond:mu} and \eqref{cond:leb} that
\[ \frac{a - \log \delta}{b + \log\delta} \leq \liminf_{\epsilon \to 0} \frac{\log \mu(x - \epsilon, x + \epsilon)}{\log \nu(x - \epsilon, x + \epsilon)} \leq
\limsup_{\epsilon \to 0} \frac{\log \mu(x - \epsilon, x + \epsilon)}{\log \nu(x - \epsilon, x + \epsilon)} \leq \frac{a + \log\delta}{b - \log\delta}.\]
Since $1<\delta<C$ can be taken arbitrarily close to $1$, we obtain \eqref{eq:a/b}.

Finally, if $m = \mu$ and $\nu = \Leb$, the second assertion of the theorem follows by Frostman's Lemma (Proposition~\ref{prop: frostman_Lemma}).
\end{proof}

\begin{remark}
The conditions on $\mu$ (resp. $\nu$) in Theorem~\ref{thm:HD_formula} could be replaced by the limit in Condition \eqref{cond:mu} (resp. \eqref{cond:leb}) holding uniformly for $x \in [0, 1)$.
\end{remark}

\section{Proofs of Theorems~\ref{thm: main1},~\ref{thm: main3} and~\ref{thm: main5} }\label{sec: proofs}

Throughout this section, we will use the following notations.

Let $T=(\pi,\lambda)$ be a hyperbolically self-similar IET of period $N$. Let $M$ be the self-similarity matrix of $T$ and $\rho_T$ be the logarithm of its Perron-Frobenius eigenvalue (i.e., $\lambda M=e^{\rho_T}\lambda$). Let $\theta \in \R^{\mathcal A}$ be the unique right Perron-Frobenius eigenvector of $M$ (i.e., $M\theta = e^{\rho_T} \theta$) satisfying $\langle \lambda,\theta\rangle=1$.

Given $\omega \in \R^{\mathcal A}$, let $M(\omega):= M^{(N)}_{\pi, \lambda,\omega}$ (see Section~\ref{subs: cocycles} for a definition of this matrix) and let us denote by $\omega_c$ (resp. $\omega_s$) its projection to the subspace of $M$-invariant vectors (resp. subspace of stable type vectors).

For any $n \geq 0$ and any $x \in [0, 1)$, let $\mathcal P^{(n)}_T:=\mathcal Q_{n\cdot N}^T$ and let $P^{(n)}_T(x)$ be the unique element of $\mathcal P^{(n)}_T$ that contains $x$. Given $\alpha \in \A$ and $0 \leq i < q_\alpha^{(N)}$, we denote by $\beta(\alpha, i)$ the unique symbol in $\A$ satisfying $T^i\big(I^{(N)}_\alpha\big) \subseteq I_\beta$.

For any $\omega \in \R^\A$, let
\begin{gather}
\label{eq: HD_invariant_formula}
\mathcal{G}(T, \omega) = \rho_c - \sum_{\alpha \in \A} \sum_{0 \leq i < q_\alpha^{(N)}} \theta_{\beta(\alpha, i)} e^{-\rho_T}\lambda_{\alpha} \left( S_i \phi_{\omega_c}\big|_{I_\alpha^{(N)}} \right),\\
\label{eq: HD_conformal_formula}
\mathcal{H}(T, \omega) = \rho_c -\sum_{\alpha \in \A} \sum_{0 \leq i < q_\alpha^{(N)}} \theta^c_{\beta(\alpha, i)} e^{-\rho_c}\ell^c_{\alpha} e^{S_i \phi_{\omega_c}\big|_{I_\alpha^{(N)}}} \left( S_i \phi_{\omega_c}\big|_{I_\alpha^{(N)}} \right),
\end{gather}
where $\ell^c \in \R^\A_+$ (resp. $\theta^c \in \R^\A_+$) is the unique left (resp.\ right) Perron-Frobenius eigenvector of $M(\omega_c)$ satisfying $| \ell^c |= 1$ (resp. $\langle \ell^c, \theta^c \rangle = 1$).

Fix $\omega \in \R^\A$ and assume $f\in\operatorname{Aff}(T,\omega)$ is an AIET of hyperbolic periodic type semi-conjugated to $T$. Let $\mu_f$ be the unique invariant probability measure of $f$ and let $h$ be the semi-conjugacy between $f$ and $T$, as introduced in Section~\ref{sc:semiconjugacy}. For any $n \geq 0$ and any $x \in [0, 1)$, let $\mathcal P^{(n)}_f:=\mathcal Q_{n\cdot N}^f$ and let $P^{(n)}_f(x)$ be the unique element of $\mathcal P^{(n)}_f$ that contains $x$.

Then we have the following facts about the partitions $\mathcal P^{(n)}_f$ and $\mathcal P^{(n)}_T$.

\begin{proposition}\label{prop: centrallimit}
Assume that $\omega$ is of central-stable type. %
Then
\begin{equation}
\lim_{n\to\infty}-\frac{1}{n}\log \Leb\big(P^{(n)}_f(x)\big) = \mathcal{G}(T, \omega)> \rho_T\qquad\text{for $\mu_f$-a.e }x\in[0,1),
\end{equation}
where $\mathcal{G}(T, \omega)$ is given by \eqref{eq: HD_invariant_formula}.

\end{proposition}

\begin{proposition}\label{prop: unstablelimit}
Assume that $\omega$ is of unstable type. Then
\begin{equation}
\lim_{n\to\infty}-\frac{1}{n}\log \Leb\big(P^{(n)}_f(x)\big)=+\infty\qquad\text{for $\mu_f$-a.e }x\in[0,1).
\end{equation}
\end{proposition}

\begin{proposition}\label{prop: centrallimit_conformal}
Assume that $\omega$ is of central-stable type. %
Denote by $\nu_\omega$ the unique $\phi^T_{\omega}$-conformal measure. Then
\begin{equation}
\lim_{n\to\infty}-\frac{1}{n}\log \nu_\omega\big(P^{(n)}_T(x)\big) = \mathcal{H}(T, \omega)< \rho_T\qquad\text{for $\nu_\omega$-a.e }x\in[0,1),
\end{equation}
where $\mathcal{H}(T, \omega)$ is given by \eqref{eq: HD_conformal_formula}.
\end{proposition}

We postpone the proofs of the above propositions to later sections (see Section~\ref{sec: HDproof}~and~\ref{sec: HDproof_conformal}), as they are the most technically involving results. We now show how we can deduce Theorems~\ref{thm: main1},~\ref{thm: main3} and~\ref{thm: main5} from these propositions.

\begin{proof}[Proof of Theorem~\ref{thm: main3}.]
In the following, we will assume that $\omega$ is of central-stable type. Theorem~\ref{thm: main3} will follow from Theorem~\ref{thm:HD_formula}. For this reason, we need to verify the assumptions of this theorem, namely \eqref{cond:mu}--\eqref{cond:three_elements}, for $\mathcal{P}^{(n)}=\mathcal{P}_f^{(n)}$, $\mu = \mu_f = m$, $\nu = Leb$, $a = \rho_T$ and $b = \mathcal{G}(T, \omega)$.

Conditions \eqref{cond:mu} and \eqref{cond:leb} follow directly from Lemma~\ref{lem:uniformlimit} and Proposition~\ref{prop: centrallimit}. Thus it remains to verify that Conditions \eqref{cond:exp_refining}, \eqref{cond:ae_generating}, and \eqref{cond:three_elements} hold.

We check now that \eqref{cond:exp_refining} holds. Recall that since $\omega$ is of central-stable type, then $f$ is conjugate to $T$ (see \eqref{eq: isconjugated}) via a homeomorphism $h:[0, 1) \to [0, 1)$ satisfying $h \circ f = T \circ h$. Since $\mu_f$ is an $f$-invariant measure and $\mu_f=(h^{-1})_*Leb$, for every $n\in\N$ it holds that
\[
\min_{J\in \mathcal P_f^{(n)}}\mu_f(J)=e^{-n\rho_T}\min_{\alpha\in\A}|I_{\alpha}|,\quad \max_{J\in \mathcal P_f^{(n)}}\mu_f(J)= e^{-n\rho_T}\max_{\alpha\in\A}|I_{\alpha}|,
\]
which gives the desired conditions for $\mu=m=\mu_f$. The inequalities in Condition \eqref{cond:exp_refining} for $Leb$, on the other hand, follow directly from \eqref{eq:baseshaveboundedquotients_affine} in Lemma~\ref{lem:from_cs_to_c}.
Thus, the condition \eqref{cond:exp_refining} holds.

Condition \eqref{cond:ae_generating} follows directly from the fact that $P_f^{(n)}(x)=h^{-1}(P_T^{(n)}(hx))$ and
\[
\max\{|J|\mid J\in\mathcal P_T^n \}\le e^{-n\rho_T}\to 0\quad\text{ as }\quad n\to\infty.
\]

Note that since $M$ is positive, every interval in $\mathcal P^{(n)}_f$ contains at least 3 elements of partition $\mathcal P^{(n+1)}_f$, thus Condition \eqref{cond:three_elements} is satisfied.

Therefore, we can apply Theorem~\ref{thm:HD_formula} and obtain that
$
\dim_H(\mu_f)=\frac{\rho_T}{\mathcal{G}(T, \omega)}.
$
Notice that by Proposition~\ref{prop: centrallimit}, we have $\dim_H(\mu_f) \in (0, 1)$. This finishes the proof of Theorem~\ref{thm: main3}.
\end{proof}

{
\begin{remark}
\label{rmk:LY}
If the log-slope $\omega$ is invariant under the self-similarity matrix of $T$ then the expression in Theorem~\ref{thm: main3} takes the form of a \emph{Ledrappier-Young formula}, namely, a ratio between the entropy and the Lyapunov exponent, not for the AIET $f$ itself, but for an associated piecewise linear (many-to-one) map $F : [0,1) \to [0,1)$ which is ergodic with respect to a probability measure $\mu_F$ equivalent to $\mu_f$, whose density is bounded away from zero and infinity (thus having the same Hausdorff dimension as $\mu_f$), and whose Lyapunov exponent $\lambda(F, \mu_F):= \int_0^1 \log |DF(x)|\, d\mu_F(x)$ satisfies $\lambda(F, \mu_F) = \mathcal{G}(T, \omega) > 0$. In this setting, the Ledrappier-Young-type formula for piecewise smooth interval maps proved by Hofbauer and Raith \cite{hofbauer_hausdorff_1992} implies $\dim(\mu_F) = \frac{h(F, \mu_F)}{\lambda(F, \mu_F)}$, where $h(F, \mu_F)$ denotes the metric entropy. %

We point out that although the proofs of Theorem~\ref{thm: main3} and Proposition~\ref{prop: centrallimit} do not rely on the results in \cite{hofbauer_hausdorff_1992}, our approach can be used to define (a posteriori) the transformation $F$ described above, satisfying the hypothesis in \cite{hofbauer_hausdorff_1992} for the dimension formula to hold. Indeed, replacing $(T, \mu)$ by $(f, \mu_f)$ in Section \ref{sec: renormalization}, letting $A = I^{(N)}(f)$ and $R: A \to [0, 1)$ denote the linear rescaling (where $N$ is the period of $f$) and taking $\theta$ as the piecewise constant function equal to $\theta_\alpha$ on the $\alpha$-th interval exchanged by $T$, where $(\theta_\alpha)_{\alpha\in\mathcal A}$ is the unique right Perron-Frobenious eigenvector of $M$ satisfying $\langle \lambda, \theta \rangle = 1$, then $F$ appears as the first coordinate of the map in \eqref{def:Rmu} and $d\mu_F = \theta d \mu_f$ is the pushforward to the first coordinate of the measure $\mu_f^\theta$ defined in Section \ref{sec: renormalization}.

Finally, we stress that the framework developed in this work allows us to treat not only the central-stable scenario but also the unstable one, where the results in \cite{hofbauer_hausdorff_1992} do not apply, as well as to consider conformal measures.
\end{remark}
}

\begin{proof}[Proof of Theorem~\ref{thm: main1}.]
If $\omega$ is of stable type, the results follow directly from Proposition~\ref{prop:stableconj}. On the other hand, if $\omega$ is of central-stable type, the result is a direct consequence of Theorem~\ref{thm: main3}.

Assume now that $\omega$ is of unstable type. By Proposition~\ref{prop: unstablelimit}, for $\mu_f$-a.e. $x\in [0,1)$ we have
\begin{equation}\label{eq: unstableproof_1}
\lim_{n\to\infty}-\frac{1}{n}\log Leb(P^{(n)}_f(x))=\infty.
\end{equation}
Moreover, by Lemma~\ref{lem:uniformlimit}, we also have that for every $x\in[0,1)$ the following convergence holds
\begin{equation}\label{eq: unstableproof_2}
\lim_{n\to\infty}-\frac{1}{n}\log\mu_f (P^{(n)}_f(x))=\rho_T.
\end{equation}

In view of Theorem~\ref{prop: frostman_Lemma}, it is enough to find for $\mu_f$-a.e. $x\in I$, a sequence $(\epsilon_n)_{n\in\N}$, which converges to $0$ and such that
\begin{equation}\label{eq: enough for 0HD}
\lim_{n\to\infty} \frac{\log \mu (x-\epsilon_n,x+\epsilon_n)}{\log \epsilon_n}=0.
\end{equation}

Let $x\in [0,1)$ be such that \eqref{eq: unstableproof_1} is satisfied and let $\epsilon_{n}:=Leb(P^{(n)}_f(x))$. In particular, by \eqref{eq: unstableproof_1}, we have $\epsilon_n\to 0$. Then, since $P^{(n)}_f(x)\subset (x-\epsilon_n,x+\epsilon_n)$, we obtain
\[
\frac{\log \mu_f (x-\epsilon_n,x+\epsilon_n)}{\log \epsilon_n}\leq
\frac{\log \mu_f(P^{(n)}_f(x))}{\log \epsilon_n}=\frac{\log \mu_f(P^{(n)}_f(x))}{n}\cdot
\frac{n}{\log Leb (P^{(n)}_f(x))}\to 0,
\]
where the convergence follows from \eqref{eq: unstableproof_1} and \eqref{eq: unstableproof_2}. This proves \eqref{eq: enough for 0HD} and completes the proof of the theorem.
\end{proof}
\begin{proof}[Proof of Theorem~\ref{thm: main5}.]
In the following, we will assume that $\omega$ is of central-stable type. As in the proof of Theorem~\ref{thm: main3}, we want to use Theorem~\ref{thm:HD_formula}. For this reason, we need to verify the assumptions of this theorem, namely \eqref{cond:mu}--\eqref{cond:three_elements}, for $\mathcal P^{(n)}=\mathcal P^{(n)}_T$, $\mu=\nu_\omega = m$, $\nu = Leb$, $a = \mathcal{H}(T, \omega)$ and $b = \rho_T$.

Conditions \eqref{cond:mu} and \eqref{cond:leb} follow directly from Proposition~\ref{prop: centrallimit_conformal} and Lemma~\ref{lem:uniformlimit}. Thus it remains to verify that Conditions \eqref{cond:exp_refining}, \eqref{cond:ae_generating}, and \eqref{cond:three_elements} hold.

We check now that \eqref{cond:exp_refining} holds. The inequalities for Lebesgue measure $\nu$ follow easily from the fact that the length vector $\lambda$ is a Perron-Frobenius eigenvector of $M$. Namely
\[
\min_{J\in \mathcal P_T^{(n)}}\nu(J)=e^{-n\rho_T}\min_{\alpha\in\A}|I_{\alpha}|,\quad \max_{J\in \mathcal P_T^{(n)}}\nu(J)= e^{-n\rho_T}\max_{\alpha\in\A}|I_{\alpha}|,
\]
for every $n\in\N$. On the other hand, the inequalities for $\mu=m=\nu_\omega$ hold by \eqref{eq:baseshaveboundedquotients} in Lemma~\ref{lem:from_cs_to_c}.

Condition \eqref{cond:ae_generating} follows directly from the fact that
\[
\max\{|J|\mid J\in\mathcal P_T^{(n)} \}\le e^{-n\rho_T}\to 0\quad\text{ as }\quad n\to\infty.
\]

Note that since $M$ is positive, every interval in $\mathcal P^{(n)}_T$ contains at least 3 elements of partition $\mathcal P^{(n+1)}$, thus Condition \eqref{cond:three_elements} is satisfied.

Therefore, we can apply Theorem~\ref{thm:HD_formula} and obtain that
$
\dim_H(\nu_\omega)=\frac{\mathcal{H}(T, \omega)}{\rho_T}.
$
Notice that, by Proposition~\ref{prop: centrallimit_conformal}, we have $\dim_H(\nu_\omega) \in (0, 1)$. This finishes the proof of Theorem~\ref{thm: main5}.
\end{proof}

\section{Implicit connections between Hausdorff dimensions and the Perron-Frobenius eigenvalue, and their applications}\label{sec: connections}
In this section, we reveal some implicit connections between the Hausdorff dimension of the invariant measure $\mu_f$ for $f\in \textup{Aff}(T,\omega)$ and the $\phi^T_\omega$-conformal measure $\nu_\omega$ when $T$ is a hyperbolically self-similar IET. Since the stable and unstable cases are not interesting from this point of view, we focus only on the case when the log-slope vector $\omega$ is of central-stable type. As Theorem~\ref{thm: main3}~and~~\ref{thm: main5} show, both Hausdorff dimensions are in fact determined by the central part of the log-slope vector $\omega$. For this reason, it is enough to assume that $\omega$ is an invariant vector of the self-similarity matrix $M$.
 	
In order to determine implicit relationships between dimensions, we introduce an additional scaling real parameter $t$ which, following the language of thermodynamics, can be treated as the inverse of temperature. More precisely, for every $t\in\R$ we deal with the $f_t\in\textup{Aff}(T,t\omega)$ invariant measure $\mu_t$ and the $\phi^T_{t\omega}$-conformal measure $\nu_t$. Their Hausdorff dimension are given by
\begin{equation}\label{eq: both dim}
\dim_H(\mu_t)=\frac{\rho_T}{\mathcal{G}(T,t\omega)}\quad\text{and}\quad\dim_H(\nu_t)=\frac{\mathcal{H}(T,t\omega)}{\rho_T}.
\end{equation}
As we will see in a moment, they are directly determined by the matrices $M(t\omega)=M^{(N)}_{\pi,\lambda,t\omega}$. Consider the function $\rho:\R\to\R$ such that $\rho(t)$ is
the logarithm of the Perron-Frobenius eigenvalue. The function $\rho$ fully determines both dimensions and acts as an intermediary that allows us to find implicit connections between them.

\begin{lemma}\label{lem: rho analytic and complex}
The map $\rho$ is real analytic and convex.
\end{lemma}

\begin{proof}
Define $F:\mathbb{R}^2 \rightarrow \mathbb{R}$ by the formula
\[F(t,\lambda)=\det(e^{\lambda} I- M(t\omega)).\]
By definition, for every $t\in \mathbb{R}$, we have $F(t,\rho(t)) = 0$. Moreover, by simplicity of Perron-Frobenius eigenvalue, we have $F(t,\lambda) = (e^{\lambda}-e^{\rho(t)})Q(t,\lambda)$, where $Q(t,\rho(t))\neq 0$. In view of \eqref{formula}, for $\alpha,\beta\in\mathcal A$, we have
\begin{equation}\label{formula1}
M_{\alpha,\beta}(t\omega) = \sum_{\substack{0\leq k< q^{(N)}_{\alpha} \\ \beta(\alpha,k)=\beta}} e^{t\cdot S_k{\phi_\omega|_{I^{(N)}_{\alpha}}}},
\end{equation}
so $F$ is analytic. We have $\frac{\partial}{\partial\lambda}F(t,\rho(t))=e^{\rho(t)}Q(t,\rho(t)) \neq 0$, and thus the analyticity of $\rho$ follows from analytic implicit function theorem. Since the family of log-convex maps is closed under taking the sum (see Lemma in \cite{Kingman}), by \eqref{formula1}, the maps
 $t \mapsto M_{\alpha,\beta} (t\omega)$ are log-convex. Then, by Kingman's Theorem (see Theorem in \cite{Kingman}), $\rho$ is convex.
\end{proof}
The following result gives an explicit formula for the Hausdorff dimension of the invariant measure for AIETs and for the conformal measure for IETs in the central case, in terms of the Perron-Frobenius eigenvalue of the associated matrix, treated as a function in $t$. This allows us to describe the dependence of one of the dimensions with respect to the other.
\begin{theorem}
For every $t\in\R$, we have
\[\mathcal{G}(T, t\omega)=\rho(t)-t\rho'(0)\quad\text{and}\quad\mathcal{H}(T, t\omega)=\rho(t)-t\rho'(t).\]
In particular, we have
\begin{equation}\label{dim_formulas}
\dim_H(\mu_t)=\frac{\rho(0)}{\rho(t)-\rho'(0)t} \quad \text{and}\quad
	\dim_H(\nu_t)=\frac{\rho(t)-\rho'(t)t}{\rho(0)}.
\end{equation}
Moreover,
\begin{equation}\label{dim_formula}
\frac{d}{dt}\Big(\frac{1}{t\dim_H(\mu_t)}\Big)=-\frac{\dim_H(\nu_t)}{t^2}.
\end{equation}
\end{theorem}

\begin{proof}
For every $t\in\R$, let $\ell(t),\theta(t)\in\R^{\mathcal A}_{>0}$ be the left and the right Perron-Frobenius eigenvectors such that $|\ell(t)|=1$ and $\langle\ell(t),\theta(t)\rangle=1$. Then $e^{\rho(t)}=\ell(t){M}(t\omega)\theta(t)$. Since the map $t \mapsto \rho(t)$ is analytic, it is not difficult to show (using the implicit function theorem) that the maps $t \mapsto \ell(t)$ and $t \mapsto \theta(t)$ are also analytic.

Notice that
\begin{equation}\label{differentiation}
	\tfrac{d}{dt}e^{\rho(t)}=\tfrac{d}{dt} \big(\ell(t) M(t\omega)\theta(t)\big) = \ell(t) \tfrac{d}{dt}{M}(t\omega)\theta(t).
\end{equation}
Indeed, using Leibniz's product rule twice, we obtain
\begin{align*}
\tfrac{d}{dt} \big(\ell(t) M(t\omega)\theta(t)\big)&
={\ell}'(t) M(t\omega)\theta(t) + \ell(t)\tfrac{d}{dt}{M}(t\omega)\theta(t)+ \ell(t)M(t\omega) {\theta'(t)}\\
&=\ell(t)\tfrac{d}{dt}{M}(t\omega)\theta(t)+e^{\rho(t)}(\ell'(t)\theta(t)+l(t)\theta'(t))\\
&=\ell(t)\tfrac{d}{dt}{M}(t\omega)\theta(t)+e^{\rho(t)}\tfrac{d}{dt}\langle\ell(t),\theta(t)\rangle=\ell(t)\tfrac{d}{dt}{M}(t\omega)\theta(t).
\end{align*}
Observe that \eqref{eq: HD_invariant_formula} can be rewritten using \eqref{formula1} and \eqref{differentiation} in the following way:

\begin{align*}
		\mathcal{G}(T, t\omega) &= \rho(t) - e^{-\rho(0)}\sum_{\alpha \in \A}\sum_{\beta\in\mathcal{A}}\ell_{\alpha}(0)\theta_{\beta}(0)\sum_{\substack{0 \leq i < q_\alpha^{(N)} \\ \beta(\alpha,i)=\beta}} \left( tS_i \phi_{\omega}\big|_{I_\alpha^{(N)}} \right) \\
&=
		\rho(t) - te^{-\rho(0)}\sum_{\alpha \in \A}\sum_{\beta\in\mathcal{A}}\ell_{\alpha}(0)\theta_{\beta}(0)\tfrac{d}{dt}M_{\alpha,\beta}(t\omega)|_{t=0} \\
&=
\rho(t) - te^{-\rho(0)}\ell(0)\Big(\tfrac{d}{dt}M(t\omega)|_{t=0} \Big)\theta(0) \\
&
		= \rho(t) - te^{-\rho(0)}\tfrac{d}{dt}e^{\rho(t)}|_{t=0} = \rho(t)-t\rho'(0).
\end{align*}
Similarity, for \eqref{eq: HD_conformal_formula}, we have
\begin{align*}
		\mathcal{H}(T, t\omega) &= \rho(t) - e^{-\rho(t)}\sum_{\alpha \in \A}\sum_{\beta\in\mathcal{A}}\ell_{\alpha}(t)\theta_{\beta}(t)\sum_{\substack{0 \leq i < q_\alpha^{(N)} \\ \beta(\alpha,i)=\beta}} e^{tS_i \phi_{\omega}\big|_{I_\alpha^{(N)}}} \left( tS_i \phi_{\omega}\big|_{I_\alpha^{(N)}} \right)
		\\ &=
		\rho(t) - te^{-\rho(t)}\sum_{\alpha \in \A}\sum_{\beta\in\mathcal{A}}\ell_{\alpha}(t)\theta_{\beta}(t)\tfrac{d}{dt} M_{\alpha,\beta}(t\omega) \\
&=
\rho(t) - te^{-\rho(t)}\tfrac{d}{dt} e^{\rho(t)}
		=\rho(t)-t\rho'(t).
	\end{align*}
In summary, in view of \eqref{eq: both dim}, we have
\begin{equation*}
	\dim_H(\mu_t)=\frac{\rho(0)}{\rho(t)-\rho'(0)t} \quad \mbox{ and }\quad
	\dim_H(\nu_t)=\frac{\rho(t)-\rho'(t)t}{\rho(0)}.
\end{equation*}
It follows that
\begin{align*}
\frac{d}{dt}\Big(\frac{1}{t\dim_H(\mu_t)}\Big)=\frac{d}{dt}\Big(\frac{\rho(t)}{t\rho(0)}-\frac{\rho'(0)}{\rho(0)}\Big)
=-\frac{\rho(t)-\rho'(t)t}{t^2\rho(0)}=-\frac{\dim_H(\nu_t)}{t^2}.
\end{align*}
\end{proof}

Finally, we will use the relations proven so far to show the asymptotic behavior of Hausdorff dimensions when the parameter $t$ goes to infinity, that is, when the temperature goes to zero.

\begin{theorem}
Both maps $t\mapsto \dim_H(\mu_t)$ and $t\mapsto \dim_H(\nu_t)$ are analytic, and increasing on $(-\infty, 0]$ and decreasing on $[0,\infty)$. Moreover,
\[\lim_{t\to\pm \infty}\dim_H(\mu_t)=0.\]
More precisely, for $t\geq 1$, we have
\begin{gather*}
\frac{\dim_H(\mu_1)}{t}\leq\dim_H(\mu_t)\leq\frac{1}{(\frac{1}{\dim_H(\mu_1)}-1)t+1},\\
\frac{\dim_H(\mu_{-1})}{t}\leq\dim_H(\mu_{-t})\leq\frac{1}{(\frac{1}{\dim_H(\mu_{-1})}-1)t+1}.
\end{gather*}
\end{theorem}

\begin{proof}
In view of \eqref{dim_formulas}, the analyticity of dimensions follows directly from the analyticity of $\rho$.
To prove monotonicity, consider the following derivative:
\begin{equation*}
	\frac{d}{dt}\dim_H(\nu_t)=\frac{-\rho''(t)t}{\rho(0)}.
\end{equation*}
Since, in view of Lemma \ref{lem: rho analytic and complex}, $\rho''(t)\geq 0$, it follows that $\dim_H(\nu_t) $ increases on $(-\infty, 0]$ and decreases on $[0,\infty)$. Note that
\begin{equation*}
	\frac{d}{dt} \dim_H(\mu_t)= \frac{-{(\dim_H(\mu_t))}^2}{\rho(0)}(\rho'(t)-\rho'(0)).
\end{equation*}
Since $\rho''(t)\geq 0$, we have $\rho'(t)-\rho'(0)\geq 0$ for $t\geq 0$ and $\rho'(t)-\rho'(0)\leq 0$ for $t\leq 0$. It follows that $\dim_H(\mu_t) $ increases on $(-\infty, 0]$ and decreases on $[0,\infty)$.

Note that all monotonicities are in fact strict. Otherwise, there would be an interval on which the dimension function is constant. Due to analyticity, this implies that the entire function is constant, which contradicts the fact that
\[\dim_H(\mu_t)<1=\dim_H(\mu_0)\quad\text{and}\quad\dim_H(\nu_t)<1=\dim_H(\nu_0)\quad\text{for}\quad t\neq 0.\]

In view of \eqref{dim_formula}, we have
\begin{equation*}
0\geq \frac{d}{dt}\Big(\frac{1}{t\dim_H(\mu_t)}\Big)\geq-\frac{1}{t^2}.
\end{equation*}
By the left inequality, we have
\[\frac{1}{t\dim_H(\mu_t)}\leq \frac{1}{\dim_H(\mu_1)}\quad\text{for}\quad t\geq 1.\]
This gives $\frac{\dim_H(\mu_1)}{t}\leq\dim_H(\mu_t)$. After integrating the right inequality over the interval $[1,t]$, we get
\[\frac{1}{t\dim_H(\mu_t)}-\frac{1}{\dim_H(\mu_1)}\geq \frac{1}{t}-1\quad\text{for}\quad t\geq 1.\]
It follows that $\dim_H(\mu_t)\leq\frac{1}{(\frac{1}{\dim_H(\mu_1)}-1)t+1}$.
The proofs of the negative versions of both inequalities that determine the rate of convergence to zero proceed in the same way as their positive versions.
\end{proof}

\section{Continuous model for IETs}\label{sec: Cantor}

\label{sc:cantor_model}
It is possible to extend any IET $T: I \to I$ satisfying Keane's condition to a continuous homeomorphism of a Cantor space. Here, we follow closely the construction given in \cite[\S 2.1]{marmi_cohomological_2005}. We will not provide proofs for the facts stated here but, instead, refer to \cite{marmi_cohomological_2005} (see also \cite[\S 10]{yoccoz_echanges_2005}) for additional details.

For the sake of simplicity, let us assume $I = [0, 1)$. Let $D_T^+$ denote the union of the (bi-infinite) orbits of the left endpoints of the intervals exchanged by $T$ and denote $D_T = D_T^+ \setminus \{0\}$. Let $\delta$ be a finite measure on $D_T$ giving positive weight to every point in $D_T$. Define increasing maps $i^+, i^-: I \to \R,$
\[ i^-(x) = x+\delta([0, x)), \qquad i^+(x) = x+\delta([0, x]), \qquad \text{ for any } x \in I. \]
Notice that $i^-$ and $i^+$ coincide on $I \setminus D_T$, $i^-(x) < i^+(x)$ for any $x \in D_T$, and the {inverses of both $i^+$ and $i^-$ are well defined and continuous on the respective images}. Moreover,
\begin{equation}
\label{eq:i_pm}
i^-(x) = \lim_{t \to x^-} i^+(x), \qquad \text{ for any } x \in (0, 1).
\end{equation}
Let $K = \overline{i^+}(I)$ and $r = \lim_{x \to 1^-} i^\pm(x).$ Notice that $K = \overline{i^-}(I) = i^+(I) \cup i^-(I) \cup \{r\}.$ Since $T$ is minimal, $K$ is a Cantor space given by
\[K = [0, r] \setminus \bigcup_{x \in D_T} (i^-(x), i^+(x)).\]
A unique continuous extension of $T$ to $K$ exists. More precisely, there exists a unique homeomorphism $\widehat{T}: K \to K$ satisfying
\begin{equation}
\label{eq:cantor_conjugacy}
\widehat{T} \circ i^+ = i^+ \circ T.
\end{equation}
Furthermore, for any function $\phi: I \to \R$ that is continuous when restricted to the intervals exchanged by $T$ and has finite one-sided limits at the endpoints (in fact, it is enough to ask for $\phi$ to be continuous on $I \setminus D_T$ with finite one-sided limits at points from $D_T$) the function $\phi \circ (i^+)^{-1}: i^+(I) \subsetneq K \to \R$ can be uniquely extended to a continuous function on $\widehat \phi :K \to \R$ satisfying
\[ \widehat{\phi} \circ i^+ = \phi.\]

{Finally, note that the set $i^{-}(D_T)\cup i^{+}(D_T)$ is a finite union of orbits via $\widehat T$. Thus if $T$ is uniquely ergodic, with $Leb$ being the only invariant measure, then $\widehat{Leb}:=(i^+)_*Leb$ is the only invariant probability measure of $\widehat T$.}

\subsection{Rauzy-Veech induction on the continuous model of an IET} While the classical Rauzy-Veech induction provides a picture which gives a good ``geometric'' intuition regarding the proof of the main results of this article, to be fully correct, for the reasons described in Subsection~\ref{subs: AIET and conformal}, we need to consider the renormalization on the continuous model of an IET. For this purpose, we formally define the extension of the Rauzy-Veech induction on the Cantor model of an IET.

Fix $T=(\pi,\lambda)$, an IET, as well as $n\in\N$ and consider the dynamical partition $\mathcal Q_n$ defined by \eqref{def:Q}. Let
\[
\widehat{\mathcal Q}_{n}:=\{\overline{i^{+}(J)}\mid J\in \mathcal Q_n\}.
\]
We now check that $\widehat{\mathcal Q}_n$ is indeed a dynamical partition of $K$ into Rokhlin towers, with respect to $\widehat T$.
First, we check that $\widehat{\mathcal Q}_n$ is a partition of $K$. We have
\[
\bigcup_{\widehat J\in \widehat{\mathcal Q}_n}\widehat J=\bigcup_{J\in {\mathcal Q}_n} \overline{i^{+}(J)}=\overline{i^+(I)}=K.
\]
Moreover, let $J_1=[a_1,b_1)$ and $J_2=[a_2,b_2)$, with $b_1\le a_2$, be two distinct elements of $\mathcal Q_n$. Then $a_1,b_1,a_2,b_2\in D_T^+$ and
\[
\max \overline{i^+(J_1)}=i^-(b_1)<i^+(a_2)=\min \overline{i^+(J_2)}.
\]
Thus, any two distinct elements of $\widehat {\mathcal Q}_n$ are disjoint.

We now show that $\widehat{\mathcal Q}_n$ is a finite union of Rokhlin towers and, in fact, it has the same tower structure as $\mathcal Q_n$. More precisely, we show that for every $\alpha\in\A$, the set $\bigsqcup_{i=0}^{q_{\alpha}^{(n)}-1} \widehat T^i \widehat I_{\alpha}^{(n)}$ is a Rokhlin tower, where $\widehat I_{\alpha}^{(n)}:=\overline{ i^+(I_{\alpha}^{(n)})}$. Since we have already shown that $\widehat{\mathcal Q}_n$ is a disjoint partition, it is enough to show that for every $0\leq i< q_{\alpha}^{(n)}$, we have
\begin{equation}\label{eq: imageisimage}
\widehat T^i(\overline{ i^+(I_{\alpha}^{(n)})})=\overline{ i^+(T^{i}I_{\alpha}^{(n)})}.
\end{equation}
Indeed, $\widehat y\in \overline{ i^+(I_{\alpha}^{(n)})}$ iff there exists a sequence $(\widehat x_n)_{n\in\N}$ in $i^+(I_{\alpha}^{(n)})$ such that $\widehat x_n\to \widehat y$ as $n\to\infty$ and $\widehat x_n=i^+(x_n)$ for some $x_n\in I_{\alpha}^{(n)}$. We need to show that $\widehat T^i(y)\in \overline{ i^+(T^{i}I_{\alpha}^{(n)})}$. By the continuity of $\widehat T$, \eqref{eq:cantor_conjugacy} and monotonicity of $i^+$, we have
\[
\widehat T^i(y)=\lim_{n\to\infty }\widehat T^{i}(\widehat x_n)= \lim_{n\to\infty} \widehat T^{i}(i^+ (x_n))=
\lim_{n\to\infty}i^+(T^i(x_n))\in \overline{ i^+(T^{i}I_{\alpha}^{(n)})}.
\]

We showed that for an arbitrary IET that satisfies the Keane condition, we can easily transfer the dynamical partition obtained by the Rauzy-Veech induction onto its continuous model. Moreover, note that we may extend the definition of the Rauzy-Veech induction itself by considering the first return map $\widehat{\mathcal{RV}}({\widehat T})$ of $\widehat T$ to $\widehat I^{(1)}:=\overline{i^+(I^{(1)})}$. Since $I^{(1)}$ is a finite union of elements of $\mathcal Q_1$, it is a well-defined map. Analogously, by induction, we get the same properties for $\widehat{\mathcal{RV}}^n({\widehat T})$ on $\widehat{I}^{(n)}:=\overline{i^+(I^{(n)})}$ for all $n\geq 1$.

{Notice that
\begin{equation}\label{diampart}
\lim_{n\to\infty}\max_{J\in\widehat{\mathcal Q}_n}\operatorname{diam}(J)=0.
\end{equation}
Indeed, if $J=\overline{i^+(T^kI^{(n)}_\alpha)}$, then
\[\operatorname{diam}(J)=|I^{(n)}_\alpha|+\delta(\operatorname{Int}(T^kI^{(n)}_\alpha)).\]
As $|I^{(n)}|\to 0$, it is enough to show that
\[\lim_{n\to\infty}\delta\Big(\bigcup_{\alpha\in\mathcal A}\bigcup_{0\leq k<q_\alpha^{(n)}}\operatorname{Int}(T^kI^{(n)}_\alpha)\Big)=0.\]
For every $\alpha\in\mathcal{A}$, let $u_\alpha$ be the left end of $I_\alpha$. As $|I^{(n)}|\to 0$, there exists a non-decreasing sequence $\{k_n\}_{n\in\N}$ of natural numbers such that $k_n\to\infty$ as $n\to\infty$ and
for any natural $n$, we have
\[T^iu_\alpha\notin \operatorname{Int}I^{(n)}\quad\text{for all}\quad \alpha\in\mathcal A\quad\text{and}
\quad -k_n\leq i\leq k_n. \]
It follows that all points $T^iu_\alpha$, for $\alpha\in\mathcal A$ and $-k_n\leq i\leq k_n$, are left ends of the intervals from the partition $\mathcal Q_n$. As $k_n\to\infty$, we have
\[\delta\Big(\bigcup_{\alpha\in\mathcal A}\bigcup_{0\leq k<q_\alpha^{(n)}}\operatorname{Int}(T^kI^{(n)}_\alpha)\Big)\leq \sum_{\alpha\in\mathcal A}\sum_{|k|>k_n}\delta(T^ku_\alpha)\to 0,\]
which gives \eqref{diampart}.}

In this article we are considering the \emph{self-similar} IETs, i.e. IETs $(\pi,\lambda)$, for which there exists natural $N$ such that $\pi^{(N)}=\pi$ and $\lambda^{(N)}=e^{-\rho} \cdot \lambda$, for some $\rho=\rho_T>0$. Then, if we denote by $R_{\rho}:[0,e^{-\rho})\to [0,1)$, the \emph{rescaling} given by $R_{\rho}(y)= e^{\rho}\cdot y$, we have
\begin{equation}\label{eq: linearrenormalization_interval}
R_{\rho}(I_{\alpha}^{(N)})=I_{\alpha},
\end{equation}
and
\begin{equation}\label{eq: rescaledcommutes_interval}
R_{\rho}\circ (\mathcal {RV})^{N}(T)=T\circ R_{\rho}.
\end{equation}

We now extend $R_{\rho}$ to the Cantor model of $T$.
Define the rescaling $\widehat R_{\rho}:\overline{i^+(I^{(N)})}\to K$ in the following way. For every $x\in I^{(N)}$, we put
\begin{equation}\label{def: hat R}
\widehat R_{\rho}(i^+(x)):=i^+(R_{\rho} (x))=i^+(e^\rho\cdot x)=e^\rho\cdot x+\delta([0,e^\rho\cdot x]).
\end{equation}
Note that $\widehat R_{\rho}$ is continuous on $i^+(I^{(N)})$.
We now show that the definition of $\widehat R_{\rho}$ can be continuously extended to $\overline{i^+(I^{(N)})}$. For this purpose, take any $\widehat y\in \overline{i^+(I^{(N)})}\setminus i^+(I^{(N)})$ and let $(\widehat x_n)_{n\in\N}$ ($\widehat{x}_n=i_+(x_n)$ for $x_n\in I^{(N)}$) be an increasing sequence of elements of the set $i^+(I^{(N)})$ converging to $\widehat y$.
We put
\[
\widehat R_{\rho}(\widehat y):=\lim_{n\to\infty}\widehat R_{\rho}(\widehat x_n).
\]
Note that the above definition does not depend on the choice of the sequence $(x_n)_{n\in\N}$. Indeed, since $\widehat y\notin i^+(I^{(N)})$, we have $\widehat y=i^-(y)=y+\delta([0,y))$ for some $y\in D^+_T$ and thus we have $x_n\le y$ for every $n\in\N$. In particular $\widehat x_n\to \widehat y$ implies
\[
\lim_{n\to\infty}(x_n+\delta([0,x_n]))=y+\delta([0,y)),
\]
which in turn implies that $\lim_{n\to\infty}x_n=y$. However, then we have that
\[
\widehat R_{\rho}(\widehat y)=\lim_{n\to\infty}\widehat R_{\rho}(\widehat x_n)
= \lim_{n\to\infty}(e^\rho x_n+\delta([0,e^\rho x_n]))=e^\rho y+\delta ([0,e^\rho y))=i^{-}(e^\rho y),
\]
and the limit does depend on the choice of the sequence $(x_n)_{n\in\N}$. Hence $\widehat R_{\rho}$ is well defined, continuous, and, by the definition, its restriction to $i^+(I^{(N)})$ is conjugated to $R_{\rho}$ via $i^+$.

Finally, note that the dynamical properties of the $R_{\rho}$ are inherited by $\widehat R_{\rho}$. Indeed, by continuity of $\widehat R_\rho$ and \eqref{eq: linearrenormalization_interval}, we get
\[
\widehat R_{\rho}(\overline{i^+(I_{\alpha}^{(N)})})=\overline{i^+(I_{\alpha})}.
\]
On the other hand, again by the continuity of $\widehat R_{\rho}$ as well as by the continuity of $\widehat T$ and by \eqref{eq: rescaledcommutes_interval}, we get
\[
\widehat R_{\rho}\circ (\widehat{\mathcal {RV}})^{N}(\widehat T)=\widehat T\circ \widehat R_{\rho}.
\]

The conclusion of this subsection is that we can freely use the notion of both the Rauzy-Veech induction as well as the self-similarity on the continuous model of the Rauzy-Veech induction. This approach is going to be useful in the proof of Theorem~\ref{thm: main1}, especially in the case when the log-slope vector is of unstable type. In the remainder of the paper, we drop the $\widehat{\cdot}$ notation as long as there is no risk of confusion. It is worth mentioning that unless the condition described in Subsection~\ref{subs: AIET and conformal} holds, it is not necessary to pass to the continuous model. However, since we do not want to omit any of the cases, we advise the reader to have in mind that for full correctness of the remaining computations, it is needed to assume that we work on the continuous model.

\subsection{Relating AIETs and the Cantor model for IETs}
\label{subs: AIET and conformal}
Following the same arguments as in Section~\ref{subs: conformal}, given $\omega \in \R^\A$ and an AIET $f \in \textup{Aff}(T, \omega)$ conjugated to an IET $T = (\pi, \lambda)$, we can construct a $\phi_\omega^T$-conformal measure $\nu_\omega$ of $T$ as follows. If $h$ is a conjugacy between $f$ and $T$ satisfying $T \circ h = h \circ f$, then
\begin{equation*}
\label{eq:projected_Lebesgue_measure}
\nu_\omega := h_*Leb,
\end{equation*}
is a $\phi_\omega^T$-conformal measure of $T$.

However, if the AIET is not conjugated to the IET, that is, if the AIET admits wandering intervals, then the measure $h_* Leb$ is not necessarily a $\phi_\omega^T$-conformal measure of $T$. Indeed, for the argument in Section~\ref{subs: conformal} to work, it suffices to verify that $\phi_\omega^f = \phi^T_\omega \circ h$, but this might not be the case if one of the discontinuity points of $f$ belongs to a wandering interval, as the function $\phi_\omega^f$ takes different values at the left and the right of the endpoint but the whole wandering interval is mapped to a single point by $h$.

Nevertheless, it is always possible to use the AIET to define a conformal measure for the Cantor model of the underlying IET. This fact will be crucial later in estimating the Hausdorff dimension of the unique invariant measure of an $\infty$-complete AIET (see Section~\ref{sec: HDproof}). To do this, we define a semi-conjugacy (no longer continuous) between $f$ and $\widehat T$ (in fact, to an appropriate restriction of $\widehat T$) that solves the issue with wandering intervals mentioned above.

Let $W_f^+$ be as in Section~\ref{sc:semiconjugacy}, that is, the union of wandering intervals of $f$ together with its endpoints. Let $D_f^+$ denote the union of the (bi-infinite) orbits of the left endpoints of the intervals exchanged by $f$ and denote $D_f = D_f^+ \setminus \{0\}$. Define $W_L$ to be the union of the connected components of $W_f^+ \setminus D_f$ such that its right endpoint belongs to $D_f$. In other words, $W_L$ is the ``left part'' of the maximal wandering intervals of $f$ that contain a point in $D_f$. Notice that the set $W_L$ is $f$-invariant and, since $f$ is semi-conjugate to $T$, each connected component of $W_f^+$ contains at most one point of $D_f$.

Define $\widehat h: I \to K$ as

\begin{equation}
\label{eq:semiconjugacy_cantor}
\widehat h(x) = \begin{cases} i^-(h(x)) \quad \text{ if } x \in W_L,\\
i^+(h(x)) \quad \text{ otherwise}.
\end{cases}
\end{equation}
Of course, $\widehat h$ is no longer continuous.
Note that
\begin{equation}
\label{eq:semiconjugacy_cantor_relation}
\widehat T \circ \widehat h(x) = \widehat h \circ f(x), \qquad \text{ for any } x \in I.
\end{equation}
Indeed, if $x \notin W_L$, then the equation above follows from \eqref{eq:cantor_conjugacy}. If $x \in W_L$, applying \eqref{eq:i_pm}, we obtain
\begin{align*}
\widehat T \circ \widehat h (x) & = \widehat T \circ i^- \circ h (x) = \lim_{t \to x^-} \widehat T \circ i^+ \circ h (t) = \lim_{t \to x^-} i^+ \circ T \circ h (t) \\ & = \lim_{t \to x^-} i^+ \circ h \circ f (t)
 = i^- \circ h \circ f (x) = \widehat h \circ f (x).
\end{align*}
Moreover,
\begin{equation}
\label{eq:potential_cantor}
\phi_\omega^f(x) = \widehat{\phi}_\omega \circ \widehat h(x), \qquad \text{ for any } x \in I,
\end{equation}
where $\widehat{\phi}_\omega:K\to\R$ denotes the unique continuous extension of $\phi^T_\omega \circ {(i^+)}^{-1} : i^+(I)\subseteq K \to \R$ to $K$. Indeed, notice that $\widehat h (W_L) \subseteq K \setminus i^+(I)$. Then, if $x \notin W_L$,
$$\phi_\omega^f(x) = \phi_\omega \circ h (x) = \phi_\omega \circ (i^{+})^{-1} \circ i^+ \circ h (x) = \widehat{\phi}_\omega \circ \widehat h (x).$$
On the other hand, if $x \in W_L$, applying \eqref{eq:i_pm},
\begin{align*}
\phi_\omega^f(x) & = \lim_{t \to x^-} \phi_\omega \circ h (t) = \lim_{t \to x^-} \phi_\omega \circ (i^{+})^{-1} \circ i^+ \circ h (t) = \widehat{\phi}_\omega \Big( \lim_{t \to x^-} i^+ \circ h (t) \Big) \\
& = \widehat{\phi}_\omega \circ i^- \circ h (x) = \widehat{\phi}_\omega \circ \widehat h(x).
\end{align*}
Hence, using the same arguments as in Section~\ref{subs: conformal} (see \eqref{eq: nu omega}), it follows from \eqref{eq:semiconjugacy_cantor_relation} and \eqref{eq:potential_cantor} that $\widehat \nu_\omega:= \widehat h_* Leb$ is a $\widehat{\phi}_\omega$-conformal measure for $\widehat T$. Note that in the case of the existence of wandering intervals of $f$, the measure $\widehat \nu_\omega$ cannot be continuous. Indeed, the images of wandering intervals give rise to atoms of the measure.

By definition, $\widehat{\phi}_\omega=\phi^T_\omega \circ {(i^+)}^{-1}$ is constant on $i^+(I_\alpha)$ and equal to $\omega_\alpha$. By continuity, $\widehat{\phi}_\omega$ is constant on the element $\widehat{I}_\alpha:=\overline{i^+(I_\alpha)}$ of the partition $\widehat{\mathcal Q}_0$, and equal to $\omega_\alpha$. In view of \eqref{eq:semiconjugacy_cantor_relation} and \eqref{eq:potential_cantor},
this gives
\[\widehat h^{-1} \widehat T^{-k}\widehat I_{\alpha}=f^{-k}I_\alpha(f)\quad\text{for all}\quad k\in\Z,\ \alpha\in\mathcal{A}.\]
Since $f$ and $\widehat T$ have the same combinatorial description of returns to the sets $I^{(n)}(f)$ and $\widehat I^{(n)}$, respectively, it follows that
\[\widehat h^{-1}\widehat{\mathcal Q}_n(\widehat h (x))=\mathcal{Q}^f_n(x)\quad\text{for any}\quad x\in I.\]
As $\widehat \nu_\omega:= \widehat h_* Leb$, this gives
\begin{equation}\label{eq: equalmeasuresonpartitions}
\widehat \nu_\omega(\widehat{\mathcal Q}_n(\widehat h (x)))=Leb(\mathcal{Q}^f_n(x))\quad\text{for all}\quad x\in I,\ n\in\N.
\end{equation}
We emphasize that the last equality is crucial in the proof of zero Hausdorff dimension of $f$-invariant measures when the log-slope vector is of unstable type.

\subsection{Existence and uniqueness of conformal measures in the central-stable case} Given a self-similar IET $T = (\pi, \lambda)$ and any $\omega \in \R^{\mathcal A}$ satisfying $\langle \lambda, \omega \rangle = 0$, by applying Proposition~\ref{prop:conformal_measures_general} to the continuous model $\widehat T$ on $K$ and the potential $\widehat\phi_\omega$, as defined in Section~\ref{subs: AIET and conformal}, we get that there exists a $\widehat\phi_\omega$-conformal measure for $\widehat T$.

Let us point out that the conformal measure above is not necessarily unique. Moreover, {its projection $(i^+)^{-1}_*\widehat \nu$} may not necessarily correspond to a $\phi_\omega$-conformal measure for $T$, %
since it may be supported on the set $i^-(D_T)$. We will now address these two issues under additional assumptions on $\omega$.

For any probability measure $\nu$ on $I$, we define a sequence $(\nu^{(n)})_{n \in \mathbb{N}}$ in $\mathbb{R}_+^{\mathcal{A}}$ by $\nu^{(n)}_\alpha = \nu(I^{(n)}_{\alpha})$, for $\alpha\in\mathcal A$. If $\nu $ is $\phi_\omega $-conformal for the IET $T$, then one can show that $\nu^{(n)} M^{(n)}_{\pi,\lambda, \omega} = \nu^{(0)}$ for every $n\in\N$. With this notion, we are able to prove the following result, which gives a unique ergodicity in the case when the potential is given by a vector of central-stable type.

\begin{proposition}
\label{prop:conformal_measures}
Let $T=(\pi,\lambda)$ be a hyperbolically self-similar IET with the self-similarity matrix $M$.
Let $\omega\in \R^{\A}$ be a vector of central-stable type.
Then there exists a unique $\widehat\phi_\omega$-conformal probability measure $\widehat\nu_\omega$.
In particular, $\widehat\nu_\omega$ is an ergodic measure for $\widehat T$.

Moreover, $\widehat\nu_\omega(i^{-}D_T)=0$ (in fact, $\widehat \nu_\omega$ is continuous) and thus, the measure $\nu_\omega:=(i^+)^{-1}_*\widehat\nu_\omega$ is a well-defined $\phi_\omega$-conformal measure for $T$. In particular, $\nu_\omega$ is the unique $\phi_\omega$-conformal measure for $T$ and is continuous.

\end{proposition}
\begin{proof}
Since $\omega$ is of central-stable type, it is, in particular, a linear combination of right-hand side contracting or unit eigenvectors of $M$. Since $\lambda$ is a left Perron-Frobenius eigenvector of $M$, we have that $\langle \omega,\lambda \rangle=0$. Then the existence of $\widehat\phi_\omega$-conformal measure $\widehat\nu_\omega$ follows directly from Proposition~\ref{prop:conformal_measures_general} applied to $\widehat T$ (which is uniquely ergodic, by unique ergodicity of $T$).

We now justify that $\widehat\nu_\omega$ is a unique $\widehat\phi_\omega$-conformal measure. Let $N$ be the period of $T$. Without loss of generality, assume that $M$ is positive. Since $\omega$ is of central-stable type, the entries
of matrices $M^{(N)}_{\pi,\lambda^{(kN)},\omega^{(kN)}}$, $k\in\N$, are uniformly bounded from below and above. Therefore, they act on the simplex $\Lambda^\A$ as contractions with a uniform contraction scale, with respect to the projective Hilbert metric.
Thus, the set $\bigcap_{k=1}^{\infty} \R_{+}^{\mathcal{A}}M^{(kN)}_{\pi,\lambda,\omega}$ is a line. Let $(\widehat\nu^{(n)})_{n \in \mathbb{N}}$ be a sequence in $\mathbb{R}_+^{\mathcal{A}}$ given by $\nu^{(n)}_\alpha = \nu(I^{(n)}_{\alpha})$, for $\alpha\in\mathcal A$. Since $\widehat\nu_\omega^{(kN)} M^{(kN)}_{\pi,\lambda, \omega} = \widehat\nu_\omega^{(0)}$, this line is spanned by the vector $\widehat\nu_\omega^{(0)}$.

To show the uniqueness of the $\widehat\phi_\omega$-conformal measure, it remains to argue why $\widehat\nu_\omega^{(0)}$ can be induced by at most one such measure.
This follows from the invertibility of $M^{(kN)}_{\pi,\lambda, \omega}$ and from the fact that, by the Dynkin Lemma, a measure is fully determined by the values it gives on the $\pi$-system $\bigcup_{k=0}^{\infty}\widehat{\mathcal Q}_{kN}$, which generates the Borel $\sigma$-algebra on $K$.

{Finally, if $\widehat\nu_\omega$ had an atom, then by the fact that $\omega$ is central-stable, its orbit would have infinite measure. Indeed, by Proposition~4.3 in \cite{trujillo_affine_2024}, for every $\widehat x_0\in K$ the exists $C>0$ and infinitely many times $(n_k)_{k\in\N}$ such that $|S_{n_k}\widehat{\phi}_\omega(\widehat x_0)|\leq C$ for all $k\in\N$. If $\widehat x_0$ is an atom of $\widehat\nu_\omega$, then
\[\widehat\nu_\omega(\{\widehat T^{n_k}\widehat x_0\})=e^{S_{n_k}\widehat{\phi}_\omega(\widehat x_0)}\widehat\nu_\omega(\{\widehat x_0\})\geq e^{-C}\widehat\nu_\omega(\{\widehat x_0\})>0\quad\text{for all}\quad k\in\N,\]
which contradicts the finiteness of $\widehat\nu_\omega$.
Hence $\widehat\nu_\omega$ is continuous and thus $\nu_\omega:=(i^+)^{-1}_*\widehat\nu_\omega$ defines an $\phi_\omega$-conformal measure for $T$. Moreover, $\nu_\omega$ is also continuous and unique.}
\end{proof}

\begin{remark}
In the above consideration, we can consider instead of an IET $T= (\pi,\lambda)$, an AIET $f\in\Aff(T,\omega)$, with $\omega$ being of central-stable type. Then $Leb$ is an $\phi^f_\omega$-conformal measure. Proceeding analogously as in the proof of Proposition~\ref{prop:conformal_measures}, we can prove that $Leb$ is the only $\phi^f_\omega$-conformal measure. In particular, it is ergodic.
\end{remark}

\section{Suspension over non-singular maps}\label{sec: skewproducts}
In this and the following sections, we establish abstract results on the Hausdorff dimension of measure for a more general class of systems than the ones listed in Theorems~\ref{thm: main1},~\ref{thm: main2},~\ref{thm: main3}, and~\ref{thm: main5}. For this reason, we highlight the notions, which later, in Section~\ref{sec: HDproof}, will be substituted by corresponding objects, according to the systems considered in the main results of this paper.

Let {$X$} be a compact metric space equipped with the $\sigma$-algebra $\mathcal{B}$ of Borel subsets. Let {$T:X\to X$} be a Borel invertible map and let {$\phi:X\to \R$} be a Borel map. Recall that a probability Borel measure {$\mu$} on $(X,\mathcal{B})$ is called {$\phi$-conformal} if the measures $T_*\mu$ and $\mu$ are equivalent and the Radon-Nikodym derivative
\[\frac{d(T^{-1})_*\mu}{d\mu}=e^{\phi}.\]
For any Borel map {$\theta:X\to\R$} denote by $T_\theta:X\times\R\to X\times \R$ the skew product defined by
\[T_\theta(x,r)=\left(T(x), e^{-\phi(x)}(r-\theta(x))\right).\]
\begin{lemma}
 The map $T_\theta:X\times\R\to X\times \R$ is Borel invertible and preserves the product measure $\mu\otimes Leb$.
\end{lemma}

\begin{proof}
 For any test function of the form $F(x,r)=f(x)\chi_{A}(r)$, using Fubini's theorem, we have
 \begin{align*}
 \int_{X \times \mathbb{R}} F \circ T_{\theta} \ d(\mu \otimes {Leb})
 &= \int_X f(T(x)) \Big( \int_{\mathbb{R}} \chi_{A}(e^{-\phi(x)}(r - \theta(x))) \, d{Leb} \Big) d\mu \\
 &= \int_X f(T(x)) \Big( \int_{\mathbb{R}} \chi_{e^{\phi(x)}A + \theta(x)}(r) \, d{Leb} \Big) d\mu \\
 &= \int_X f(T(x)) e^{\phi(x)} \Big( \int_{\mathbb{R}} \chi_{A}(r) \, d{Leb} \Big) d\mu \\
 &= \int_X f(T(x)) \Big( \int_{\mathbb{R}} \chi_{A}(r) \, d{Leb} \Big) d(T^{-1})_*\mu \\
 &= \int_X f(x) \Big( \int_{\mathbb{R}} \chi_{A}(r) \, d{Leb} \Big) d\mu = \int_{X \times \mathbb{R}} F \ d(\mu \otimes {Leb}).
 \end{align*}
\end{proof}

For every $n\in\Z$, let $S_n\phi$ and $S_n^\phi\theta$ be defined by:
\[
S_n\phi(x):=\left\{
\begin{array}{rl}
 \sum_{0\leq i<n}\phi(T^ix) & \text{if }n\geq 0,\\
 -\sum_{n\leq i<0}\phi(T^ix) & \text{if }n< 0,
\end{array}\right.\]
\[
S_n^\phi\theta(x):=\left\{
\begin{array}{rl}
 \sum_{0\leq i<n}e^{S_{i}\phi(x)}\theta(T^ix) & \text{if }n\geq 0,\\
 -\sum_{n\leq i<0}e^{S_{i}\phi(x)}\theta(T^ix) & \text{if }n< 0.
\end{array}\right.\]
Then, for every $n\in\Z$,
\begin{equation}\label{eq: trans meas n}
\frac{d(T^{-n})_*\mu}{d\mu}=e^{S_n\phi}\quad{and}\quad T^n_\theta(x,r)=\left(T^nx,e^{-S_{n}\phi(x)}\left(r-S_n^\phi\theta(x)\right)\right).
\end{equation}

Let us consider an equivalence relation $\sim_\theta$ on $X\times\R$ given by
\[(x,r)\sim_\theta (y,s)\Leftrightarrow (y,s)=T_\theta^n(x,r)\quad\text{for some}\quad n\in\Z,\]
and the corresponding quotient space (of orbits) $X^\theta:=(X\times\R)/\sim_\theta$. If $\theta>0$ then {we identify} $X^\theta$
with the suspension set
\[X^{\theta}=\{(x,s):x\in X,0\leq r<\theta(x)\}.\]
Indeed, for any $(x,r)\in X\times \R$ the sequence $(S_n^\phi\theta(x))_{n\in\Z}$ is strictly increasing. Therefore, there exists a unique $n\in \Z$ such that $S_n^\phi\theta(x)\leq r<S_{n+1}^\phi\theta(x)$. Then $T^n_\theta(x,r)\in X^\theta$ and $T^m_\theta(x,r)\notin X^\theta$ for $m\neq n$.

If additionally $\theta>0$ is $\mu$-integrable, then the restriction of $\mu\otimes Leb$ to $X^\theta$ is denoted by $\mu^\theta$. We will usually assume that $\int_X \theta\,d\mu=1$, so that the measure $\mu^\theta$ is probabilistic.

\subsection{Renormalization map}\label{sec: renormalization}
\label{sc:renormalization_map}
Let {$A\in\mathcal{B}$} be a subset such that $\bigcup_{n\in\N}T^{-n}A=\bigcup_{n\in\N}T^{n}A=X$. Then the maps {$n_A:X\to \Z_{\geq 0}$} and {$r_A:A\to \N$} given by
\begin{align*}
 n_A(x)&=\min\{n\geq 0\mid T^{-n}x\in A\}, \qquad\text{for } x\in X,\\
 r_A(x)&=\min\{n\in\N\mid T^nx\in A\}, \qquad\text{for } x\in A,
\end{align*}
are well defined everywhere and Borel.
Denote by $T_A:A\to A$ the induced map, i.e.
\[T_A(x):=T^{r_A(x)}(x), \qquad\text{for all } x\in A.\]
Moreover, suppose that there exists {a \emph{rescaling} of $T$, that is} a Borel invertible map {$R:A\to X$} such that
\begin{equation}\label{def:R}
 T\circ R=R\circ T_A.
\end{equation}

Suppose that $\mu$ is a $\phi$-conformal measure and $0<\mu(A)<1$. Denote by $\mu_A$ the conditional measure on $A$, i.e.\ $\mu_A(B)=\mu(B|A)=\mu(B\cap A)/\mu(A)$. We will also deal with the restriction $\mu|_A$ of $\mu$ to $A$, i.e.\ $\mu|_A(B)=\mu(B\cap A)$.

Let $\phi_A:A\to\R$ be the Borel map given by
\begin{equation}\label{def:phiA}
\phi_A(x)=S_{r_A(x)}\phi(x).
\end{equation}
In view of \eqref{eq: trans meas n}, we have $(T_A)_*\mu_A\sim \mu_A$ and
\begin{equation}\label{eq:render}
\frac{d(T^{-1}_A)_*\mu_A}{d\mu_A}=e^{\phi_A}.
\end{equation}
On the other hand,
\begin{equation*}\label{eq:render1}
\frac{d(T^{-1}_A)_*(R^{-1}_*\mu)}{dR^{-1}_*\mu}=\frac{dR^{-1}_*(T^{-1}_*\mu)}{dR^{-1}_*\mu}=e^{\phi\circ R}.
\end{equation*}
Due to the above identities, we have the following property.

\begin{lemma}\label{lem:renmeas}
 Suppose that $\mu$ is the unique $\phi$-conformal measure, or $\mu_A\sim R^{-1}_*\mu$ and $\mu$ is ergodic for $T$.
 If $\phi\circ R=\phi_A$ then $\mu_A= R^{-1}_*\mu$. On the other hand, if $\mu_A= R^{-1}_*\mu$, then $\phi\circ R=\phi_A$, $\mu_A$-a.e.
\end{lemma}

From now on assume that $\phi\circ R=\phi_A$ and $\mu_A= R^{-1}_*\mu$. Let
\[\rho_\mu:=-\log \mu(A)>0.\]
For any Borel $\theta:X\to\R$, denote by $\theta_A=\theta^\phi_A:A\to\R$ the map
\[\theta_A(x)=S^\phi_{r_A(x)}\theta(x).\]
Then $(T_A)_{\theta_A}:A\times\R\to A\times\R$ is the induced map on $A\times\R$ for the skew product $T_\theta$. It follows that
$(T_A)_{\theta_A}$ preserves the product measure $\mu_A\times Leb$.

Let us consider the Borel map $\xi_A:X\times \R\to A\times \R$ given by
\[\xi_A(x,r)=T_{\theta}^{-n_A(x)}(x,r).\]
This map is surjective and is not one-to-one, however
\[(x,r)\sim_\theta (y,s)\Leftrightarrow \xi_A(x,r)\sim_{\theta_A} \xi_A(y,s).\]
Indeed, if $(x,r)\sim_\theta (y,s)$ then $(y,s)=T^n_\theta(x,s)$, $\xi_A(x,r)=T_{\theta}^{-n_1}(x,r)\in A\times \R$ and $\xi_A(y,s)=T_{\theta}^{-n_2}(y,s)\in A\times \R$. Hence $\xi_A(y,s)=T_{\theta}^{n+n_1-n_2}\xi_A(x,r)$. As $\xi_A(x,r), \xi_A(y,s)\in A\times\R$,
we have $\xi_A(y,s)=(T_A)_{\theta_A}^{m}\xi_A(x,r)$ for some $m\in\Z$. The implication in the opposite direction is even more direct.

Therefore, the quotient map $\xi_A:X^\theta\to A^{\theta_A}$ is well defined and bijective. Suppose that $\theta>0$ and is $\mu$-integrable (with $\mu^\theta$ probabilistic).
Then $\theta_A>0$, and we can use the identification
\[X^\theta=\{(x,s):x\in X,0\leq r<\theta(x)\}\text{ and }A^{\theta_A}=\{(x,s):x\in A,0\leq r<\theta_A(x)\}.\]
Moreover, we have
\[\xi_A(x,r)=T_\theta^{-n_A(x)}(x,r)\in A^{\theta_A}\text{ for any }(x,r)\in X^\theta.\]
Indeed, if $n=n_A(x)\geq 0$ and $(y,s)=\xi_A(x,r)=T_\theta^{-n}(x,r)$, then $y=T^{-n}x\in A$ and $r=e^{-S_n\phi(y)}(s-S^{\phi}_{n}\theta(y))$.
As $0\leq r<\theta(x)$, it follows that
\[S^{\phi}_{n}\theta(y)\leq s<e^{-S_n\phi(y)}\theta(T^{n}y)+S^{\phi}_{n}\theta(y)=S^{\phi}_{n+1}\theta(y).\]
Since $0\leq n<r_A(y)$, it follows that
\[0\leq s\leq S^{\phi}_{n+1}\theta(y)\leq S^{\phi}_{r_A(y)}\theta(y)=\theta_A(y),\]
so $T_\theta^{-n_A(x)}(x,r)=(y,s)\in A^{\theta_A}$.

Since $\xi_A:X^\theta\to A^{\theta_A}$ is a bijection and is piecewise defined by iterations of the skew product, it preserves the product measure $\mu\otimes Leb$. It follows that
\begin{equation}\label{eq:tranmeas}
 (\xi_A)_*(\mu^\theta)=(\mu|_A)^{\theta_A}.
\end{equation}

Let us consider the Borel invertible map $S:A\times \R\to X\times \R$ given by
\[S(x,r)=(Rx,e^{-\rho_\mu}r).\]
As $\mu(A)=e^{-\rho_\mu}$ and $R_*(\mu|_A)=\mu(A) \mu$, we get $S_*(\mu|_A\times Leb)=\mu\times Leb$. Note that $S(A^{\theta_A})=X^{e^{-\rho_\mu}\theta_A\circ R^{-1}}$.

Finally, we define the renormalization map {$R_\mu:X^\theta\to X^{e^{-\rho_\mu}\theta_A\circ R^{-1}}$} given by the composition $R_\mu=S\circ\xi_A$.
Then $R_\mu$ is a Borel bijection and, by \eqref{eq:tranmeas}, we have $(R_\mu)_*(\mu^\theta)=\mu^{e^{-\rho_\mu}\theta_A\circ R^{-1}}$. From now on, we will assume that
\begin{equation}
{\theta_A\circ R^{-1}=e^{\rho_\mu}\theta.}
\end{equation}
Then $R_\mu:X^\theta\to X^\theta$ is a Borel automorphism preserving the measure $\mu^\theta$. {The map ${R_{\mu}}$ will play the role of a renormalization map that will allow us to define a stationary Markov chain, whose ergodicity is going to be crucial when computing the exact value of the Hausdorff dimension of the considered measures}.

For every $n\geq 0$, let $A_n:=\{x\in A: r_A(x)=n\}$. Then $(A_n)_{n\geq 0}$ is a Borel partition of $A$ such that
\[\{T^iA_n:n\geq 0, 0\leq i<n\}\]
is a Borel partition of $X$. Moreover, if $(x,r)\in X^\theta$ and $x\in T^iA_n$ for some $n\geq 0$ and $0\leq i<n$, then
\begin{align}\label{def:Rmu}
 \begin{aligned} R_{\mu}(x,r)&=S(T_\theta^{-i}(x,r))=\left(R(T^{-i}x),e^{-\rho_\mu-S_{-i}\phi(x)}\left(r-S_{-i}^\phi\theta(x)\right)\right)\\
&=\left(R(T^{-i}x),e^{-\rho_\mu}\left(e^{S_{i}\phi(T^{-i}x)}r+S_{i}^\phi\theta(T^{-i}x)\right)\right).
\end{aligned}
\end{align}
Denote by $\pi:X^\theta\to X$ the projection $\pi(x,r)=x$. Then
\begin{equation}\label{eq:piR}
 \pi(R_{\mu}(x,r))=R(T^{-i}x)\quad\text{if}\quad x\in T^iA_n\text{ with }0\leq i<n.
\end{equation}

\begin{remark}
 To obtain the main results of this article, we will substitute in Sections~\ref{sec: HDproof} and~\ref{sec: HDproof_conformal} for $T$ the hyperbolically self-similar IET (or rather its Cantor model), with either the Lebesgue or the conformal measure. Moreover, the sequence of partitions $\mathcal P^{(n)}$ will be the sequence of partitions $\mathcal Q_{n\cdot N}^T$, where $N$ is the period of $T$, while the vector $\theta$ is going to be a Perron-Frobenius vector of a proper matrix, depending on the considered case, so that $\mu^\theta$ is a probability measure preserved by the associated map $R_\mu$.

 Hence, while the following results apply in a more general setting, the reader may want to visualize the techniques and objects in these concrete cases. In Sections~\ref{sec: HDproof} and~\ref{sec: HDproof_conformal} we provide a complete list of the precise objects to which we will apply this formalism.
\end{remark}

\section{A sequence of dynamical partitions and their information content} \label{sec: dynpar}
Let {$\mathcal{P}=(P_\alpha)_{\alpha\in\mathcal{A}}$} be a finite Borel partition of $X$. For every $x\in X$ denote by $P(x)$ the unique atom of $\mathcal{P}$ containing $x$.
One can consider the associated partition $\mathcal{P}^\theta=(P^\theta_\alpha)_{\alpha\in\mathcal{A}}$ on $X^\theta$ for which $P^\theta_\alpha$ is the suspension over $P_\alpha$ and under $\theta$.

Suppose that for every $\alpha\in\mathcal{A}$ there exists {$q_\alpha\geq 0$} such that $R^{-1}(P_\alpha)\subset A_{q_\alpha}$. Then
\[\mathcal{P}^{(1)}:=\{T^i(R^{-1}P_\alpha)\mid \alpha\in\mathcal{A},0\leq i<q_\alpha\}\]
is a Borel partition of $X$. Assume that for every $\alpha\in \mathcal{A}$ and $0\leq i<q_\alpha$ we have $T^i(R^{-1}P_\alpha)\subset P_{\beta(\alpha,i)}$ for some $\beta(\alpha,i)\in\mathcal{A}$. Simply speaking, the partition $\mathcal{P}^{(1)}$ is finer than $\mathcal{P}$.

For every $n\geq 1$, let {$X^{(n)}=R^{-n}X$}. Then $X^{(1)}=A$. In view of \eqref{def:R}, the induced map $T_{X^{(n)}}:X^{(n)}\to X^{(n)}$ satisfies
\begin{equation}\label{eq:Rn}
 T\circ R^n=R^n\circ T_{X^{(n)}}.
\end{equation}
Then $(R^{-n}P_\alpha)_{\alpha\in\mathcal{A}}$ is a partition of $X^{(n)}$ such that for any $\alpha\in\mathcal{A}$ there exists $q^{(n)}_\alpha\geq 0$ such that $T_{X^{(n)}}=T^{q^{(n)}_\alpha}$ on $R^{-n}P_\alpha$.
It follows that the collection {$\mathcal{P}^{(n)}$} defined
\[\mathcal{P}^{(n)}:=\{T^i(R^{-n}P_\alpha)\mid \alpha\in\mathcal{A},0\leq i<q^{(n)}_\alpha\}\]
is a partition of $X$ finer than $\mathcal{P}^{(n-1)}$. For every $x\in X$, we denote by ${P}^{(n)}(x)$ the only element of $\mathcal{P}^{(n)}$ containing $x$. {Note that $\mathcal P^{(n)}$ can be seen as a collection of levels of all Rokhlin towers whose bases form the set $R^{-n}\mathcal P$.

 The main goal of this section is to compute the information content of the partition $P^{(n)}$ for every $n\in\N$, with respect to a given measure $\nu$. More precisely, we are going to give formulas (and inequalities) for the value $-\log(\nu(P^{(n)}(x)))$ for a given $x\in X$.}

Note that for any $x\in X$ and $0\leq k\leq n$, we have
\begin{equation}\label{eq:passP}
 P^{(n)}(R^{-k}x)=R^{-k}P^{(n-k)}(x),
\end{equation}
{that is, the assignment of an element of a partition commutes with the rescaling of the base, up to a change of the index.}
Indeed, suppose that $x\in T^j(R^{-(n-k)}P_\alpha)=P^{(n-k)}(x)$ for some $\alpha\in\mathcal{A}$ and $0\leq j<q^{(n-k)}_\alpha$.
Then $R^{-k}x\in R^{-k}T^j(R^{-(n-k)}P_\alpha)= T_{X^{(k)}}^j(R^{-n}P_\alpha)\in\mathcal{P}^{(n)}$. Hence
\[R^{-k}P^{(n-k)}(x)=R^{-k}T^j(R^{-(n-k)}P_\alpha)=P^{(n)}(R^{-k}x).\]
In view of \eqref{eq:piR}, $R^{n-k}(T^{-j}x)=\pi(R^{n-k}_{\mu}(x,r))\in P^{(k)}(\pi(R^{n-k}_{\mu}(x,r)))$. Therefore $x\in T^jR^{-(n-k)}P^{(k)}(\pi(R^{n-k}_{\mu}(x,r)))\in\mathcal{P}^{(n)}$.
It follows that
\begin{equation}\label{eq:passP1}
 P^{(n)}(x)= T^jR^{-(n-k)}P^{(k)}(\pi(R^{n-k}_{\mu}(x,r)))\text{ for some }0\leq j<\max_{\alpha\in\mathcal{A}}q_\alpha^{(n-k)}.
\end{equation}

Suppose that $\nu$ is a probability Borel measure such that $T_*\nu\sim \nu$ and the Radon-Nikodym derivative $\frac{d(T^{-1}_*\nu)}{d\nu}=e^{\psi}$ is constant on the atoms of the partition $\mathcal{P}$.
As $\mathcal{P}^{(1)}$ is finer that $\mathcal{P}$, the Radon-Nikodym derivative $\frac{d(T^{-j}_*\nu)}{d\nu}$ is constant on $R^{-1}P_\alpha$ for $0\leq j\leq q_\alpha$. Hence $\frac{d((T_A^{-1})_*(\nu|_A))}{d(\nu|_{A})}$ is constant on the atoms of $R^{-1}\mathcal{P}$. As $R\circ T_A=T\circ R$, this gives
\begin{equation}\label{eq: RD renorm}
\frac{d(T^{-1}_*(R_*(\nu|_A)))}{d(R_*(\nu|_{A}))}=\frac{d((T_A^{-1})_*(\nu|_A))}{d(\nu|_{A})}\circ R^{-1}
\end{equation}
is constant on the atoms of $\mathcal{P}$. The same argument shows that for any $n\geq 1$,
\[\frac{d(T^{-1}_*(R^n_*(\nu|_{X^{(n)}})))}{d(R^n_*(\nu|_{X^{(n)}}))}=\frac{d((T_{X^{(n)}}^{-1})_*(\nu|_{X^{(n)}}))}{d(\nu|_{X^{(n)}})}\circ R^{-n}\]
is constant on the atoms of $\mathcal{P}$. It follows that
\begin{equation}\label{eq:nupassj}
 \frac{d(T^{-j}_*(R^n_*(\nu|_{X^{(n)}})))}{d(R^n_*(\nu|_{X^{(n)}}))}\text{ is constant on }R^{-1}P_\alpha{\in\mathcal P^{(1)}}\text{ for }0\leq j\leq q_\alpha.
\end{equation}
Note that, by \eqref{eq:passP}, for any $(x,r)\in X^\theta$ and $n\geq 1$,
\begin{equation}\label{eq: lognu_first_estimate}
 \begin{split}
 \log & \ \nu ({P}^{(n)}(x))\\&=\sum_{0\leq i<n}\log \frac{\nu({P}^{(n)}(R^{-i}\pi(R^i_\mu(x,r))))}{\nu({P}^{(n)}(R^{-(i+1)}\pi(R^{i+1}_\mu(x,r))))}+\log \nu({P}^{(n)}(R^{-n}\pi(R^n_\mu(x,r))))\\
 & =\sum_{0\leq i<n}\log\frac{\nu(R^{-i}{P}^{(n-i)}(\pi(R^i_\mu(x,r))))}{\nu(R^{-i}{P}^{(n-i)}(R^{-1}\pi(R^{i+1}_\mu(x,r))))}+\log \nu(R^{-n}{P}(\pi(R^n_\mu(x,r))))\\
 & =\sum_{0\leq i<n}\log\frac{R^i_*(\nu|_{X^{(i)}})({P}^{(n-i)}(\pi(R^i_\mu(x,r))))}{R^i_*(\nu|_{X^{(i)}})({P}^{(n-i)}(R^{-1}\pi(R^{i+1}_\mu(x,r))))}\\
&\quad+\log R^n_*(\nu|_{X^{(n)}})({P}(\pi(R^n_\mu(x,r))))
 .\end{split}
\end{equation}

Let $(y,s)\in X^\theta$ and let $y=\pi(y,s)\in T^jR^{-1}P_\alpha\in \mathcal{P}^{(1)}$ for some $\alpha\in\mathcal A$ and $0\leq j=j(y,s)<q_\alpha$. By \eqref{eq:piR}, $T^{j}R^{-1}\pi(R_{\mu}(y,s))=y$. It follows that for any $k\geq 1$, we have
\begin{equation}\label{eq:relativeposition}
 {P}^{(k)}(y)=T^j{P}^{(k)}(R^{-1}\pi(R_{\mu}(y,s))),\quad\text{and}\quad {P}^{(k)}(R^{-1}\pi(R_{\mu}(y,s)))\subset R^{-1}P_\alpha.\end{equation}
{The above expression relates an arbitrary point $y\in X$ with an orbit of an element of the partition $\mathcal P^{(k)}$, which is included in some of the intervals from the partition $\mathcal P^{(1)}$, which in turn is contained in a rescaled domain $A$.}
In view of \eqref{eq:relativeposition}, \eqref{eq:nupassj}, and again \eqref{eq:passP}, {for every $i\in\N$},
we have
 \begin{align*}
 \log &\frac{R^i_*(\nu|_{X^{(i)}})({P}^{(k)}(\pi(y,s)))}{R^i_*(\nu|_{X^{(i)}})({P}^{(k)}(R^{-1}\pi(R_\mu(y,s))))}\\
 &= \log \frac{R^i_*(\nu|_{X^{(i)}})(T^j{P}^{(k)}(R^{-1}\pi(R_\mu(y,s))))}{R^i_*(\nu|_{X^{(i)}})({P}^{(k)}(R^{-1}\pi(R_\mu(y,s))))}\\
 &= \log \frac{R^i_*(\nu|_{X^{(i)}})(T^j{P}^{(1)}(R^{-1}\pi(R_\mu(y,s))))}{R^i_*(\nu|_{X^{(i)}})({P}^{(1)}(R^{-1}\pi(R_\mu(y,s))))}\\
 &= \log \frac{R^i_*(\nu|_{X^{(i)}})(T^jR^{-1}{P}(\pi(R_\mu(y,s))))}{R^i_*(\nu|_{X^{(i)}})(R^{-1}{P}(\pi(R_\mu(y,s))))}.
\end{align*}
By applying the above to $k=n-i$ and $(y,s)=R^{i}_{\mu}(x,r)$, where $0\le i<n$, we get from \eqref{eq: lognu_first_estimate} that
\begin{align}\label{eq:logP1}
 \begin{aligned}
 \log \nu({P}^{(n)}(x))
 & =\sum_{0\leq i<n} \log \frac{R^i_*(\nu|_{X^{(i)}})(T^{j(R^{i}_\mu(x,r))}R^{-1}{P}(\pi(R^{i+1}_\mu(x,r))))}{R^i_*(\nu|_{X^{(i)}})(R^{-1}{P}(\pi(R^{i+1}_\mu(x,r))))}\\
 &\quad+\log R^n_*(\nu|_{X^{(n)}})({P}(\pi(R^n_\mu(x,r)))).
 \end{aligned}
\end{align}
Moreover, by applying \eqref{eq:passP} with $n=k=1$ to the numerators in the above expression, as well as \eqref{eq:relativeposition} with $k=1$, we get
\begin{equation}\label{eq:logP2}
 \log \nu({P}^{(n)}(x))
 =\sum_{0\leq i<n} \log \frac{R^i_*(\nu|_{X^{(i)}})({P}^{(1)}(\pi(R^{i}_\mu(x,r))))}{R^i_*(\nu|_{X^{(i)}})({P}(\pi(R^{i}_\mu(x,r))))}+\log \nu({P}(\pi(x,r))).
\end{equation}

For a given $(x,r)\in X^\theta$, let $\alpha,\beta\in\mathcal{A}$ and $0\leq j<q_\alpha$ such that $x\in T^{j}R^{-1}P_\alpha\subset P_\beta$ ($\beta=\beta(\alpha,j)$).
Denote by $0\leq l(x,r)<q_\alpha$ the the minimal number such that $\beta(\alpha,l(x,r))=\beta$ and
\[\nu(T^{l(x,r)}R^{-1}P_\alpha)=\max\{\nu(T^{l}R^{-1}P_\alpha)\mid \ 0\leq l<q_\alpha,\ \beta(\alpha,l)=\beta\}.\]
In other words, we choose the level of a tower built over $R^{-1}P_{\alpha}$, which is the largest wrt. measure $\nu$, among all the levels of this tower that are included in the set $P_{\beta}$. Then,
\begin{align*}
 \log&\frac{R^{i}_*(\nu|_{X^{(i)}})({P}(\pi(x,r)))}{R^{i+1}_*(\nu|_{X^{(i+1)}})({P}(\pi(R_\mu(x,r))))}
 =\log\frac{R^{i}_*(\nu|_{X^{(i)}})({P}(\pi(x,r)))}{R^{i}_*(\nu|_{X^{(i)}})(R^{-1}{P}(\pi(R_\mu(x,r))))}\\
 &
 =\log\frac{R^{i}_*(\nu|_{X^{(i)}})(P_\beta)}{R^{i}_*(\nu|_{X^{(i)}})(R^{-1}P_\alpha)}\geq
 \log\sum_{\substack{0\leq l<q_\alpha\\
 \beta(\alpha,l)=\beta}}\frac{R^{i}_*(\nu|_{X^{(i)}})(T^{l}R^{-1}P_\alpha)}{R^{i}_*(\nu|_{X^{(i)}})(R^{-1}P_\alpha)}\\
 &\geq \log\frac{R^{i}_*(\nu|_{X^{(i)}})(T^{l(x,r)}R^{-1}P_\alpha)}{R^{i}_*(\nu|_{X^{(i)}})(R^{-1}P_\alpha)}
 =\log\frac{R^{i}_*(\nu|_{X^{(i)}})(T^{l(x,r)}R^{-1}{P}(\pi(R_\mu(x,r))))}{R^{i}_*(\nu|_{X^{(i)}})(R^{-1}{P}(\pi(R_\mu(x,r))))}.
\end{align*}
In view of \eqref{eq:logP1}, it follows that
\begin{align*}
 -\log \nu({P}^{(n)}(x))
 &=-\sum_{0\leq i<n}\log \frac{R^i_*(\nu|_{X^{(i)}})(T^{j(R_\mu^i(x,r))}R^{-1}{P}(\pi(R^{i+1}_\mu(x,r))))}{R^i_*(\nu|_{X^{(i)}})(R^{-1}{P}(\pi(R^{i+1}_\mu(x,r))))}\\
 &\quad+\sum_{0\leq i<n}\log \frac{R^i_*(\nu|_{X^{(i)}})({P}(\pi(R^i_\mu(x,r))))}{R^{i+1}_*(\nu|_{X^{(i+1)}})({P}(\pi(R^{i+1}_\mu(x,r))))}
 -\log\nu({P}(x))\\
 &\geq-\sum_{0\leq i<n}\log \frac{R^i_*(\nu|_{X^{(i)}})(T^{j(R_\mu^i(x,r))}R^{-1}{P}(\pi(R^{i+1}_\mu(x,r))))}{R^i_*(\nu|_{X^{(i)}})(R^{-1}{P}(\pi(R^{i+1}_\mu(x,r))))}\\
 &\quad+\sum_{0\leq i<n}\log \frac{R^i_*(\nu|_{X^{(i)}})(T^{l(R_\mu^i(x,r))}R^{-1}{P}(\pi(R^{i+1}_\mu(x,r))))}{R^i_*(\nu|_{X^{(i)}})(R^{-1}{P}(\pi(R^{i+1}_\mu(x,r))))}
 .
\end{align*}
Hence
\begin{equation}\label{eq:logP3}
 -\log \nu({P}^{(n)}(x))
 \geq\sum_{0\leq i<n}\log \frac{R^i_*(\nu|_{X^{(i)}})(T^{l(R_\mu^i(x,r))}R^{-1}{P}(\pi(R^{i+1}_\mu(x,r))))}{R^i_*(\nu|_{X^{(i)}})(T^{j(R_\mu^i(x,r))}R^{-1}{P}(\pi(R^{i+1}_\mu(x,r))))}.
\end{equation}
{It is worth to mention at this point that for IETs (and AIETs) the equalities \eqref{eq:logP1} and \eqref{eq:logP2} are going to be crucial while estimating the exact, non-trivial value of the Hausdorff dimension of the conformal (invariant) measure in the case, where the log-slope vector $\omega$ is going to be given by the eigenvector of the Rauzy matrix corresponding to the eigenvalue $1$. On the other hand, the inequality \eqref{eq:logP3} will be crucial when proving that if $\omega$ is of unstable type, then the Hausdorff dimension of the invariant measure is equal $0$.}

\section{Perfectly scaled renormalizations}\label{sec: perfectlyscaled}
\label{sc:perfectly_scaled}
{Let $\mu$ be a $\phi$-conformal measure on $X$ be as in Section~\ref{sec: renormalization}, in particular $\phi\circ R=\phi_A$, $R^{-1}_*(\mu|_A)= e^{-\rho_\mu}\mu$, and $\theta_A\circ R^{-1}=e^{\rho_\mu}\theta$. Then the renormalization map $R_{\mu}: (X^\theta,\mu^\theta)\to (X^\theta,\mu^\theta)$ is well defined. Let $\nu$ be a $\psi$-conformal measure on $X$ such that $\psi=\log\frac{d(T^{-1}_*\nu)}{d\nu}$ is constant on the atoms of the partition $\mathcal{P}$.
As we have already seen in Section~\ref{sec: dynpar}, the renormalized Radon-Nikodym derivative $\frac{d(T^{-1}_*(R_*(\nu|_A)))}{d(R_*(\nu|_{A}))}$ is also constant on the atoms of the partition $\mathcal{P}$.

In this section, assume that the logarithm of the Radon-Nikodym derivative $\psi$ is perfectly rescaled when passing through renormalization. More precisely, we assume that there exists $\lambda\in\R$ such that}
\begin{equation}\label{eq:condlambda}
 \log\frac{d(T^{-1}_*(R_*(\nu|_A)))}{d(R_*(\nu|_{A}))}=e^\lambda\log\frac{d(T^{-1}_*\nu)}{d\nu}.
\end{equation}
It follows that for any $n,j\geq 0$,
\begin{equation}\label{eq:lank}
 \log\frac{d(T^{-j}_*(R^n_*(\nu|_{X^{(n)}})))}{d(R^n_*(\nu|_{X^{(n)}}))}=e^{\lambda n}\log\frac{d(T^{-j}_*\nu)}{d\nu}.
\end{equation}
Finally note that, by \eqref{eq:render} and \eqref{eq: RD renorm}, condition \eqref{eq:condlambda} is equivalent to $\psi_A\circ R^{-1}=e^\lambda\psi.$

 \subsection{Case $\lambda=0$}\label{sec:perfectzero}
 Suppose that $\lambda=0$ and $\nu$ is the unique $\psi$-conformal measure. By Lemma~\ref{lem:renmeas}, $\nu(R^{-1}B)=e^{-\rho_\nu}\nu(B)$ for any Borel $B\subset X$, where $\rho_\nu:=-\log \nu(A)$. {Then $R^n_*(\nu|_{X^{(n)}})=e^{-n\rho_{\nu}}\nu$.}
 In view of \eqref{eq:logP2}, this gives
 \begin{equation}\label{eq:infla=0}
 -\log \nu({P}^{(n)}(x)) =-\sum_{0\leq i<n}\log\frac{\nu({P}^{(1)}(\pi(R^i_\mu(x,r))))}{\nu({P}(\pi(R^{i}_\mu(x,r))))}-\log \nu({P}(x)).
 \end{equation}

 \begin{theorem}
 \label{thm:perfect_zero}
 Suppose that $R_\mu:(X^\theta,\mu^\theta)\to (X^\theta,\mu^\theta)$ is ergodic and $\nu$ is the unique $\psi$-conformal measure which satisfies \eqref{eq:condlambda} with $\lambda=0$.
 Then, for $\mu$-a.e.\ $x\in X$,
 \begin{align}\label{eq:inf1}
 \begin{aligned}
 \lim_{n\to\infty}-\frac{1}{n}\log \nu(P^{(n)}(x))=H^u_\nu(R_\mu)&=-\sum_{\alpha\in \mathcal{A}}\sum_{0\leq i<q_\alpha}\log\frac{\nu(T^iR^{-1}P_\alpha)}{\nu(P_{\beta(\alpha,i)})}\theta_{\alpha,i}>0\\
 &=-\sum_{\alpha\in \mathcal{A}}\sum_{0\leq i<q_\alpha}\log\frac{\nu(T^iR^{-1}P_\alpha)}{\nu(P_{\alpha})}\theta_{\alpha,i},
 \end{aligned}
 \end{align}
 where $\theta_{\alpha,i}=\int_{T^iR^{-1}P_\alpha}\theta(x)d\mu(x)$.
 If additionally $\nu$ is $T$-invariant, then
 \begin{equation}\label{eq:rhonu}
 H^u_\nu(R_\mu)=\rho_\nu.
 \end{equation}
 \end{theorem}

 \begin{proof}
 By \eqref{eq:infla=0} and Birkhoff's ergodic theorem, for $\mu^\theta$-a.e.\ $(x,r)\in X^\theta$ the sequence $-\frac{1}{n}\log \nu(P^{(n)}(\pi(x,r)))$ tends to the integral
 \begin{align*}
 -\int_{X^\theta}\log\frac{\nu({P}^{(1)}(\pi(x,r)))}{\nu({P}(\pi(x,r)))}d\mu^\theta(x,r)
 &=-\int_X\theta(x)\log\frac{\nu({P}^{(1)}(x))}{\nu({P}(x))}d\mu(x)\\
 &=-\sum_{\alpha\in \mathcal{A}}\sum_{0\leq i<q_\alpha}\log\frac{\nu(T^iR^{-1}P_\alpha)}{\nu(P_{\beta(\alpha,i)})}\theta_{\alpha,i},
 \end{align*}
 As the measure $\mu^\theta$ is $R_\mu$-invariant, we also have
 \[-\int_{X^\theta}\log\frac{\nu({P}^{(1)}(\pi(x,r)))}{\nu({P}(\pi(x,r)))}d\mu^\theta(x,r)=
 -\int_{X^\theta}\log\frac{\nu({P}^{(1)}(\pi(x,r)))}{\nu({P}(\pi(R_{\mu}(x,r))))}d\mu^\theta(x,r).\]
 Recall that if $\pi(x,r)\in T^jR^{-1}P_\alpha$, then $\pi(R_\mu(x,r))\in P_\alpha$. It follows that
 \[{P}^{(1)}(\pi(x,r)))=T^jR^{-1}P_\alpha\quad\text{and}\quad {P}(\pi(R_{\mu}(x,r)))=P_\alpha.\]
 This gives the second line of \eqref{eq:inf1}.

 If additionally $\nu$ is $T$-invariant, then
 \[\frac{\nu({P}^{(1)}(\pi(x,r)))}{\nu({P}(\pi(R_{\mu}(x,r))))}=\frac{\nu(T^jR^{-1}P_\alpha)}{\nu({P}_\alpha)}
 =\frac{\nu(R^{-1}P_\alpha)}{\nu({P}_\alpha)}=e^{-\rho_\nu}.\]
 It follows that
 \begin{equation*}
 -\int_{X^\theta}\log\frac{\nu({P}^{(1)}(\pi(x,r)))}{\nu({P}(\pi(x,r)))}d\mu^\theta(x,r)=\rho_\nu,
 \end{equation*}
 which gives \eqref{eq:rhonu}
 \end{proof}

 From now on, we will assume that the map $\theta$ is constant on every $P_\beta$ and let
 $\theta_\beta$ be the value of $\theta$ on $P_\beta$ for $\beta\in\mathcal{A}$. Then $\theta_{\alpha,i}=\mu(T^iR^{-1}P_\alpha)\theta_{\beta(\alpha,i)}$.

 For every $\beta\in\mathcal{A}$, denote by $\overline{\mu}^\beta$ and $\overline{\nu}^\beta$ two probability distributions on the set
 \begin{equation}\label{eq: defSigma_beta}
 \Sigma_\beta=\{(\alpha,i)\in \mathcal{A}\times \Z_{\geq 0}\mid 0\leq i<q_\alpha,\,\beta(\alpha,i)=\beta\}
\end{equation}
 given by
 \[\overline{\mu}^\beta(\alpha,i)=\frac{\mu(T^iR^{-1}P_\alpha)}{\mu(P_\beta)}\quad\text{and}\quad
 \overline{\nu}^\beta(\alpha,i)=\frac{\nu(T^iR^{-1}P_\alpha)}{\nu(P_\beta)}\text{ for }(\alpha,i)\in\Sigma_\beta.\]

{
\begin{lemma}\label{lem: mu=nu}
Suppose additionally that $A\subset P_{\alpha_0}$ for some $\alpha_0\in\mathcal{A}$, $TP_\alpha\nsubseteq A$ for all $\alpha\in\mathcal{A}$, and $\operatorname{diam}\mathcal{P}_n\to 0$ as $n\to\infty$. If $\overline{\mu}^\beta(\alpha,i)=\overline{\nu}^\beta(\alpha,i)$ for all $\beta\in\mathcal A$ and $\alpha\in\Sigma_\beta$, then $\mu=\nu$.
\end{lemma}

\begin{proof}
By assumption,
\begin{equation}\label{eq: propor}
\frac{\mu(T^iR^{-1}P_\alpha)}{\nu(T^iR^{-1}P_\alpha)}=\frac{\mu(P_\beta)}{\nu(P_\beta)}\quad\text{if}\quad \beta(\alpha,i)=\beta.
\end{equation}
It follows that
\[\frac{\mu(R^{-1}P_\alpha)}{\nu(R^{-1}P_\alpha)}=\frac{\mu(P_{\alpha_0})}{\nu(P_{\alpha_0})}=c_{\alpha_0}\quad\text{for all}\quad\alpha\in\mathcal{A}.\]
As $R_*(\mu|_A)=e^{-\rho_\mu}\mu$ and $R_*(\nu|_A)=e^{-\rho_\nu}\nu$, this gives
\begin{equation}\label{eq: propor1}
\frac{\mu(P_\alpha)}{\nu(P_\alpha)}=c_{\alpha_0}e^{\rho_\mu-\rho_\nu}=e^{\rho_\mu-\rho_\nu}\frac{\mu(P_{\alpha_0})}{\nu(P_{\alpha_0})}\quad\text{for all}\quad\alpha\in\mathcal{A}.
\end{equation}
Since both measures are probabilistic, we have $\mu(P_\alpha)=\nu(P_\alpha)$ for all $\alpha\in\mathcal{A}$. In view of \eqref{eq: propor}, this gives
$\mu(T^iR^{-1}P_\alpha)=\nu(T^iR^{-1}P_\alpha)$ for all $\alpha\in\mathcal A$ and $0\leq i<q_\alpha$. It follows that the Radon-Nikodym derivatives for both measures are equal, i.e.\ $\phi=\psi$.
Moreover, by \eqref{eq: propor1}, we obtain $\rho_\mu=\rho_\nu$. Therefore,
\[\mu(R^{-n}P_\alpha)=e^{-n\rho_\mu}\mu(P_\alpha)=e^{-n\rho_\nu}\nu(P_\alpha)=\nu(R^{-n}P_\alpha)\quad\text{for all}\quad\alpha\in\mathcal A, \ n\geq 1.\]
Since the Radon-Nikodym derivatives for $\mu$ and $\nu$ are equal, it follows that
\[\mu(T^iR^{-n}P_\alpha)=\nu(T^iR^{-n}P_\alpha)\quad\text{for all}\quad\alpha\in\mathcal A\quad\text{and}\quad 0\leq i<q_\alpha^{(n)}.\]
Thus, $\mu$ and $\nu$ coincide on atoms of the partition $\mathcal P_n$ for any $n\geq 1$. Since the diameters of $\mathcal P_n$ tend to zero as $n\to\infty$, we obtain $\mu=\nu$.
\end{proof}}

{Let us recall that given two probability measures $P,Q $ on a finite set $\Omega$, we consider \emph{Kullback-Leibler divergence} (also called relative entropy) of $P$ with respect to $Q$ defined by the formula
 \begin{equation*}
 D(P\| Q) =\sum_{\omega \in \Omega} P(\omega)\log\frac{P(\omega)}{Q(\omega)}.
 \end{equation*}
 We recall the following Theorem, which in the literature is commonly called ``divergence inequality''.
 \begin{theorem} \label{thm:KL}
Let $P$ and $Q$ be two probability measures on a finite set $\Omega$. Then
 \begin{equation*}
 D(P\| Q) \geq 0,
 \end{equation*}
 with equality if and only if $P=Q$.
 \end{theorem}}

 The following two theorems allow us to estimate the right-hand side of \eqref{eq:inf1}, which in turn will allow us to deduce that the Hausdorff dimension of certain measures is strictly between $0$ and $1$.

 \begin{theorem}
 \label{thm:difference_conformal}
 Suppose that $\mu$ is $T$-invariant, $R_\mu:(X^\theta,\mu^\theta)\to (X^\theta,\mu^\theta)$ is ergodic and $\nu$ is a $\psi$-conformal measure on $X$ which satisfies \eqref{eq:condlambda} with $\lambda=0$. Then
 \begin{align}\label{eq:inf2}
 H^u_\nu(R_\mu)=\rho_\mu+\sum_{\beta\in\mathcal{A}}D(\overline{\mu}^\beta\|\overline{\nu}^\beta)\mu(P_\beta)\theta_\beta.
 \end{align}
 {Suppose additionally that $A\subset P_{\alpha_0}$ for some $\alpha_0\in\mathcal{A}$, $TP_\alpha\nsubseteq A$ for all $\alpha\in\mathcal{A}$, and $\operatorname{diam}\mathcal{P}_n\to 0$ as $n\to\infty$. Then $\mu\neq\nu$ implies $H^u_\nu(R_\mu)>\rho_{\mu}$.}
 \end{theorem}

 \begin{proof}
 Since $\mu$ is $T$-invariant, in view of \eqref{eq:rhonu} applied to $\nu=\mu$, we have
 \[\rho_\mu=H^u_\mu(R_\mu)=-\sum_{\alpha\in \mathcal{A}}\sum_{0\leq i<q_\alpha}\log\frac{\mu(T^iR^{-1}P_\alpha)}{\mu(P_{\beta(\alpha,i)})}\mu(T^iR^{-1}P_\alpha)\theta_{\beta(\alpha,i)}.\]
 By \eqref{eq:inf1}, we also have
 \[H^u_\nu(R_\mu)=-\sum_{\alpha\in \mathcal{A}}\sum_{0\leq i<q_\alpha}\log\frac{\nu(T^iR^{-1}P_\alpha)}{\nu(P_{\beta(\alpha,i)})}\mu(T^iR^{-1}P_\alpha)\theta_{\beta(\alpha,i)}.\]
 It follows that
 \begin{align*}
 H^u_\nu(R_\mu)&=\rho_\mu+\sum_{\alpha\in \mathcal{A}}\sum_{0\leq i<q_\alpha}\log\frac{\frac{\mu(T^iR^{-1}P_\alpha)}{\mu(P_{\beta(\alpha,i)})}}{\frac{\nu(T^iR^{-1}P_\alpha)}{\nu(P_{\beta(\alpha,i)})}}
 \frac{\mu(T^iR^{-1}P_\alpha)}{\mu(P_{\beta(\alpha,i)})}\mu(P_{\beta(\alpha,i)})\theta_{\beta(\alpha,i)}\\
 &=\rho_\mu+\sum_{\beta\in \mathcal{A}}\sum_{(\alpha,i)\in\Sigma_\beta}\log\frac{\frac{\mu(T^iR^{-1}P_\alpha)}{\mu(P_{\beta})}}{\frac{\nu(T^iR^{-1}P_\alpha)}{\nu(P_{\beta})}}
 \frac{\mu(T^iR^{-1}P_\alpha)}{\mu(P_{\beta})}\mu(P_{\beta})\theta_{\beta}\\
 &=\rho_\mu+\sum_{\beta\in\mathcal{A}}D(\overline{\mu}^\beta\|\overline{\nu}^\beta)\mu(P_\beta)\theta_\beta,
 \end{align*}
 which gives \eqref{eq:inf2}.

 The last assertion follows now directly from Theorem~\ref{thm:KL} and Lemma~\ref{lem: mu=nu}.
 \end{proof}

 \begin{theorem}
 \label{thm:difference_self-similar}
 Suppose that $R_\mu:(X^\theta,\mu^\theta)\to (X^\theta,\mu^\theta)$ is ergodic and $\varrho$ is $T$-invariant.
 Then
 \begin{align}\label{eq:infin3}
 H^u_\mu(R_\mu)=\rho_\varrho-\sum_{\beta\in\mathcal{A}}D(\overline{\mu}^\beta\|\overline{\varrho}^\beta)\mu(P_\beta)\theta_\beta.
 \end{align}	
 {Suppose additionally that $A\subset P_{\alpha_0}$ for some $\alpha_0\in\mathcal{A}$, $TP_\alpha\nsubseteq A$ for all $\alpha\in\mathcal{A}$, and $\operatorname{diam}\mathcal{P}_n\to 0$ as $n\to\infty$. Then $\mu\neq\varrho$ implies $H^u_\mu(R_\mu)<\rho_{\varrho}$.}
 \end{theorem}

 \begin{proof}
 Since $\varrho$ is $T$-invariant, it satisfies \eqref{eq:condlambda} with $\lambda=0$ and in view of \eqref{eq:rhonu}, we have
 \[\rho_\varrho=H^u_\varrho(R_\mu)=-\sum_{\alpha\in \mathcal{A}}\sum_{0\leq i<q_\alpha}\log\frac{\varrho(T^iR^{-1}P_\alpha)}{\varrho(P_{\beta(\alpha,i)})}\mu(T^iR^{-1}P_\alpha)\theta_{\beta(\alpha,i)}.\]
 By \eqref{eq:inf1} again, we have
 \[H^u_\mu(R_\mu)=-\sum_{\alpha\in \mathcal{A}}\sum_{0\leq i<q_\alpha}\log\frac{\mu(T^iR^{-1}P_\alpha)}{\mu(P_{\beta(\alpha,i)})}\mu(T^iR^{-1}P_\alpha)\theta_{\beta(\alpha,i)}.\]
 It follows that
 \begin{align*}
 H^u_\mu(R_\mu)&=\rho_\varrho-\sum_{\alpha\in \mathcal{A}}\sum_{0\leq i<q_\alpha}\log\frac{\frac{\mu(T^iR^{-1}P_\alpha)}{\mu(P_{\beta(\alpha,i)})}}{\frac{\varrho(T^iR^{-1}P_\alpha)}{\varrho(P_{\beta(\alpha,i)})}}
 \frac{\mu(T^iR^{-1}P_\alpha)}{\mu(P_{\beta(\alpha,i)})}\mu(P_{\beta(\alpha,i)})\theta_{\beta(\alpha,i)}\\
 &=\rho_\varrho-\sum_{\beta\in\mathcal{A}}D(\overline{\mu}^\beta\|\overline{\varrho}^\beta)\mu(P_\beta)\theta_\beta,
 \end{align*}
 which gives \eqref{eq:infin3}.

 The last assertion follows again directly from Theorem~\ref{thm:KL} and Lemma~\ref{lem: mu=nu}.
 \end{proof}

 \subsection{Case $\lambda>0$}\label{sec:lambda pos}
 If $\lambda>0$ then, in view of \eqref{eq:logP3} and \eqref{eq:lank},
 {\begin{align}\label{eq:infinf}
 -\log \nu({P}^{(n)}(x))
 \geq
 \sum_{0\leq i<n}e^{\lambda i}\log \frac{\nu(T^{l(R_\mu^i(x,r))}{R^{-1}}P(\pi(R^{i+1}_\mu(x,r))))}{\nu(T^{j(R_\mu^i(x,r))}{R^{-1}}P(\pi(R^{i+1}_\mu(x,r))))}.
 \end{align}}
 From now on for any Borel probability measure $\nu$ on $X$ we denote by $g^{\nu}_A:X^\theta\to\R_{\geq 0}$ the map given by
 \begin{equation}\label{def:gA}
 g^{\nu}_A(x,r)=\log \frac{\nu(T^{l(x,r)}R^{-1}{P}(\pi(R_\mu(x,r))))}{\nu(T^{j(x,r)}R^{-1}{P}(\pi(R_\mu(x,r))))}\geq 0.
 \end{equation}
{Recall that for given a $(x,r)\in X^\theta$ if $x\in T^jR^{-1}{P}(\pi(R_\mu(x,r))))=T^jR^{-1}P_\alpha\in \mathcal{P}^{(1)}$ for some $\alpha\in\mathcal A$ and $0\leq j<q_\alpha$, then $j(x,r):=j$.						
If additionally $T^{j}R^{-1}P_\alpha\subset P_\beta$ (i.e.\ $\beta=\beta(\alpha,j)$), then
$0\leq l(x,r)<q_\alpha$ is the minimal number such that $\beta(\alpha,l(x,r))=\beta$ and
\[\nu(T^{l(x,r)}R^{-1}P_\alpha)\geq \nu(T^{l}R^{-1}P_\alpha)\quad\text{for all}\quad 0\leq l<q_\alpha\text{ with } \beta(\alpha,l)=\beta.\]}				

 \begin{lemma}\label{lem:Xn}
 Suppose that $A\subset P_{\alpha_0}$ for some $\alpha_0\in\mathcal A$, $\operatorname{diam} {X}^{(n)}\to 0$.
 Let $\nu$ be a probability Borel measure such that $\frac{d(T^{-1})_*\nu}{d\nu}$ is constant on atoms of $\mathcal{P}$, \eqref{eq:condlambda} holds for some $\lambda\in\R$. If $g^\nu_{X^{(n)}}\equiv 0$ for every $n\geq 1$, then $\nu$ is $T$-invariant.
 \end{lemma}

 \begin{proof}
 Suppose that $g_A\equiv 0$. Assume that $x,T^kx\in P_\beta$ for some $k>0$ so that $T^ix\notin A$ for $1\leq i\leq k$. Then both $x,T^kx\in P_\beta$ belong the same tower over some $R^{-1}P_\alpha$, so there is $0\leq j<q_\alpha-k$ so that
 $x\in T^jR^{-1}P_\alpha$ and $T^kx\in T^{j+k}R^{-1}P_\alpha$. By definition, $l(x,0)=l(T^kx,0)=l$ and
 \[0=g_A(x,0)=\log\frac{\nu(T^lR^{-1}P_\alpha)}{\nu(T^jR^{-1}P_\alpha)}\quad\text{and}\quad 0=g_A(T^kx,0)=\log\frac{\nu(T^lR^{-1}P_\alpha)}{\nu(T^{j+k}R^{-1}P_\alpha)}.\]
 It follows that $\nu(T^{j+k}R^{-1}P_\alpha)=\nu(T^{j}R^{-1}P_\alpha)$ and $\frac{d(T^{-k}_*\nu)}{d\nu}=1$ on $T^{j}R^{-1}P_\alpha=P^{(1)}(x)$.

 Suppose that $g_{X^{(n)}}\equiv 0$ for all $n\geq 1$. As $\operatorname{diam} {X}^{(n)}\to 0$, the intersection $\bigcap_{n\geq 1}{X}^{(n)}$ contains at most one point denoted by $x_0$.
 Assume that $x,T^kx\in P_\beta$ for some $k>0$ and $T^ix\neq x_0$ for $1\leq i\leq k$. By assumption, there exists $n\geq 1$ such that $T^ix\notin X^{(n)}$ for $1\leq i\leq k$. As $g_{X^{(n)}}\equiv 0$, it follows that $\frac{d(T^{-k}_*\nu)}{d\nu}=1$ on $P^{(n)}(x)$ whenever
 $x,T^kx\in P_\beta$ for some $\beta\in\mathcal{A}$ and $k>0$, and $T^ix\neq x_0$ for $1\leq i\leq k$.

For a given $\alpha\in \mathcal A$ choose any $x\in R^{-1}P_\alpha=P^{(1)}(x)$, which is not an element of the backward semi-orbit of $x_0$.		
Let $k:=r_A(x)>1$ be the first return time to $A$.					
 As $x,T^kx\in A\subset P_{\alpha_0}$, we have $\frac{d(T_A^{-1})_*(\nu|_A)}{d(\nu|_A)}=\frac{d((T^{-k})_*\nu)}{d\nu}=1$ on $P^{(n)}(x)$ for some $n\geq 1$.
Since $\frac{d(T_A^{-1})_*(\nu|_A)}{d(\nu|_A)}$ is constant on the atoms of $R^{-1}\mathcal P$, we have $\log\frac{d(T_A^{-1})_*(\nu|_A)}{d(\nu|_A)}=0$ on $R^{-1}P_\alpha$.
Hence $\log\frac{d(T_A^{-1})_*(\nu|_A)}{d(\nu|_A)}=0$. In view of \eqref{eq: RD renorm} and \eqref{eq:condlambda}, we get the $T$-invariance of the measure $\nu$.
 \end{proof}

 \begin{lemma}\label{lem:infty}
 Let $T:(X,\mu)\to (X,\mu)$ be a measure-preserving ergodic map. Let $f:X\to \R_{\geq 0}$ be an integrable map with $\int_X fd\mu>0$. Let $(a_n)_{n\geq 0}$ be a sequence of positive numbers such that $a_n\to+\infty$ as $n\to \infty$.
 Then, for $\mu$-a.e.\ $x\in X$, we have
 \[\frac{1}{n}\sum_{0\leq i<n}a_if(T^ix)=+\infty.\]
 \end{lemma}

 \begin{theorem}\label{thm:infinte}
 Suppose that $A\subset P_{\alpha_0}$ for some $\alpha_0\in\mathcal A$, $\operatorname{diam} {X}^{(n)}\to 0$ and $R_\mu:(X^\theta,\mu^\theta)\to (X^\theta,\mu^\theta)$ is ergodic.
 Suppose that $\nu$ is probability Borel measure which is not $T$-invariant, and such that $\frac{d(T^{-1})_*\nu}{d\nu}$ is constant on atoms of $\mathcal{P}$ and \eqref{eq:condlambda} holds for some $\lambda>0$.
 Then, for $\mu$-a.e.\ $x\in X$,
 \begin{equation}\label{eq:inf3}
 \lim_{n\to\infty}-\frac{1}{n}\log \nu(P^{(n)}(x))=+\infty.
 \end{equation}
 \end{theorem}

 \begin{proof}
 In view of \eqref{eq:infinf},
 \[-\frac{1}{n}\log \nu({P}^{(n)}(x))\geq
 \frac{1}{n}\sum_{0\leq i<n}e^{\lambda i}g^\nu_A (R^i_\mu(x,r)).\]
 By definition, $g^\nu_A$ is non-negative, depends only on the first coordinate, and is constant on the atoms of the partition $\mathcal{P}^{(1)}$.
 If $g^\nu_A$ is non-zero, then its integral is positive and, by the ergodicity of $R_\mu$, \eqref{eq:inf3} follows from Lemma~\ref{lem:infty}. Moreover, if $g^\nu_{X^{(k)}}$ is non-zero for some $k\geq 2$ then, replacing $A$ by $X^{(k)}$ and $R$ by $R^k$, and using the same arguments, we get \eqref{eq:inf3} along an arithmetic progression, which gives the full version of \eqref{eq:inf3} immediately.									
 As $\nu$ is not $T$-invariant, by Lemma~\ref{lem:Xn}, $g^\nu_{X^{(k)}}$ is non-zero for some $k\geq 1$, which completes the proof.
 \end{proof}

{The above result provides a key argument to show that the Hausdorff dimension of the AIET invariant measure is zero when the log-slope vector is an expanding eigenvector of the self-similarity matrix.	
In the next section, we modify the arguments used so far so that they can also be applied to the general case when the log-slope vector is of unstable type.}

 \section{Imperfectly scaled renormalizations}\label{sec:imperfectlyscaled}
 In this section, we relax our assumption on the measure $\nu$ and we show a version of Theorem~\ref{thm:infinte} when $\lambda$ in \eqref{eq:condlambda} is no longer a constant but is a function constant on the atoms of the partition $\mathcal{P}$.
 Suppose again that $\nu$ is a probability Borel measure such that $T_*\nu\sim \nu$ and $\frac{d(T^{-1}_*\nu)}{d\nu}=e^{\psi}$ is constant on the atoms of the partition $\mathcal{P}$.
 As $\mathcal{P}^{(n)}$ is finer than $\mathcal{P}$, the Radon-Nikodym derivative $\frac{d((T_{X^{(n)}}^{-1})_*(\nu|_{X^{(n)}}))}{d(\nu|_{X^{(n)}})}$ is constant on the atoms of $R^{-n}\mathcal{P}$. As $R^n\circ T_{X^{(n)}}=T\circ R^n$, this gives that
 \[\frac{d((T^{-1})_*R^n_*(\nu|_{X^{(n)}}))}{d(R^n_*(\nu|_{X^{(n)}}))}=\frac{d((T_{X^{(n)}}^{-1})_*(\nu|_{X^{(n)}}))}{d(\nu|_{X^{(n)}})}\circ R^{-n}\]
 is constant on the atoms of $\mathcal{P}$. Therefore both $\frac{d(T^{-1}_*\nu)}{d\nu}$ and $\frac{d((T^{-1})_*R^n_*(\nu|_{X^{(n)}}))}{d(R^n_*\nu|_{X^{(n)}})}$ are constant on $P_\alpha$.
 Assume additionally {$\log\frac{d(T^{-1}_*\nu)}{d\nu}$ and $\log\frac{d((T^{-1})_*R^n_*(\nu|_{X^{(n)}}))}{d(R^n_*\nu|_{X^{(n)}})}$} have the same sign.
 Then there exists a map $\lambda_1^{(n)}:X\to\R$ which is constant on the atoms of $\mathcal{P}$ and
 \begin{equation}\label{eq:condlambda1}
 \log\frac{d(T^{-1}_*(R^n_*(\nu|_{X^{(n)}})))}{d(R^n_*(\nu|_{X^{(n)}}))}=e^{\lambda_1^{(n)}}\log\frac{d(T^{-1}_*\nu)}{d\nu}.
 \end{equation}
 As $\log\frac{d(T^{-1}_*(R^n_*(\nu|_{X^{(n)}})))}{d(R^n_*(\nu|_{X^{(n)}}))}$ for $n\geq 0$ is constant on the atoms of $\mathcal{P}$, we will identify this map with a vector in $\R^{\mathcal{A}}$.

 Let us consider any atom $T^jR^{-1}P_\alpha\in\mathcal{P}^{(1)}$ and any $k\in\Z$ such that $0\leq j+k<q_\alpha$. Then $\frac{d(T^{-k}_*\nu)}{d\nu}$ and $\frac{d((T^{-k})_*R^n_*(\nu|_{X^{(n)}}))}{d(R^n_*\nu|_{X^{(n)}})}$ are constant on $T^jR^{-1}P_\alpha$.
 Assume additionally that {their logs} have the same sign.
 Then there exists a partially defined map $\lambda_k^{(n)}$, defined on the union of some atoms of $\mathcal{P}^{(1)}$, which is constant on every such atom and
 \begin{equation}\label{eq:lank1}
 \log\frac{d(T^{-k}_*(R^n_*(\nu|_{X^{(n)}})))}{d(R^n_*(\nu|_{X^{(n)}}))}=e^{\lambda_k^{(n)}}\log\frac{d(T^{-k}_*\nu)}{d\nu}.
 \end{equation}
 In view of \eqref{eq:logP3},
 \begin{equation}\label{eq: Pnestimate}
 -\frac{1}{n}\log \nu({P}^{(n)}(x))\geq
 \frac{1}{n}\sum_{0\leq i<n}e^{\lambda^{(i)}(R^i_\mu(x,r))}g_A^{\nu} (R^i_\mu(x,r)),
 \end{equation}
 where $g_A^{\nu}:X^\theta\to \R_{\geq 0}$ is given by \eqref{def:gA} and $\lambda^{(i)}:X^\theta\to \R$ is given by
 \[\lambda^{(i)}(x,r):=\lambda^{(i)}_{l(x,r)-j(x,r)}(x,r).\]
 {Indeed,
 \[
 \begin{split}
 -&\log \nu({P}^{(n)}(x)) \geq\sum_{0\leq i<n}\log \frac{R^i_*(\nu|_{X^{(i)}})(T^{l(R_\mu^i(x,r))}R^{-1}{P}(\pi(R^{i+1}_\mu(x,r))))}{R^i_*(\nu|_{X^{(i)}})(T^{j(R_\mu^i(x,r))}R^{-1}{P}(\pi(R^{i+1}_\mu(x,r))))}\\
 &=\sum_{0\leq i<n}\log \frac{\big(T^{-(l(R_\mu^i(x,r))-j(R_\mu^i(x,r)))}_*R^i_*(\nu|_{X^{(i)}})\big)(T^{j(R_\mu^i(x,r))}R^{-1}{P}(\pi(R^{i+1}_\mu(x,r))))}{R^i_*(\nu|_{X^{(i)}})(T^{j(R_\mu^i(x,r))}R^{-1}{P}(\pi(R^{i+1}_\mu(x,r))))}\\
 &=\sum_{0\leq i<n}e^{\lambda^{(i)}(R_\mu^i(x,r))}\log \frac{\big(T^{-(l(R_\mu^i(x,r))-j(R_\mu^i(x,r)))}_*\nu\big)(T^{j(R_\mu^i(x,r))}R^{-1}{P}(\pi(R^{i+1}_\mu(x,r))))}{\nu(T^{j(R_\mu^i(x,r))}R^{-1}{P}(\pi(R^{i+1}_\mu(x,r))))}\\
 &=\sum_{0\leq i<n}e^{\lambda^{(i)}(R_\mu^i(x,r))}\log \frac{\nu(T^{l(R_\mu^i(x,r))}R^{-1}{P}(\pi(R^{i+1}_\mu(x,r))))}{\nu(T^{j(R_\mu^i(x,r))}R^{-1}{P}(\pi(R^{i+1}_\mu(x,r))))}\\
 &=\sum_{0\leq i<n}e^{\lambda^{(i)}(R_\mu^i(x,r))} g^{\nu}_A(R_\mu^i(x,r)).
 \end{split}
 \]}
 Note that $\lambda^{(i)}$ depends only on the first coordinate and is constant on atoms of the partition $\mathcal{P}^{(1)}$.

 \begin{proposition}\label{thm:inf3}
 Suppose that $R_\mu:(X^\theta,\mu^\theta)\to (X^\theta,\mu^\theta)$ is ergodic.
 Let $\nu$ be a Borel measure such that $T_*\nu\sim \nu$, the Radon-Nikodym derivative $\frac{d(T^{-1}_*\nu)}{d\nu}$ is constant on the atoms of the partition $\mathcal{P}$ and the condition \eqref{eq:lank1} holds. Assume that there exists a Borel set $B\subset X^\theta$ such that $\mu^\theta(B)>0$, $g^\nu_A$ is positive on $B$, and
 \[\lim_{n\to\infty}\inf\{e^{\lambda^{(n)}(x,r)}\mid (x,r)\in B\}=+\infty.\]
 Then, for $\mu$-a.e.\ $x\in X$,
 \begin{equation}\label{eq:inf31}
 \lim_{n\to\infty}-\frac{1}{n}\log \nu(P^{(n)}(x))=+\infty.
 \end{equation}
 \end{proposition}

 \begin{proof}
 By the ergodicity of $R_\mu$, \eqref{eq:inf31} follows directly from Lemma~\ref{lem:infty} and
 \[-\frac{1}{n}\log \nu({P}^{(n)}(x))\geq
 \frac{1}{n}\sum_{0\leq i<n}e^{\lambda^{(i)}(R^i_\mu(x,r))}g^\nu_A (R^i_\mu(x,r))\chi_B(R^i_\mu(x,r)),\]
 which is a consequence of \eqref{eq: Pnestimate}.
 \end{proof}

 Let $\nu$ be a Borel measure such that $T_*\nu\sim \nu$ and $\frac{d(T^{-1}_*\nu)}{d\nu}$ is constant on the atoms of the partition $\mathcal{P}$.
 Assume that there are vectors $h_1,h_2,\ldots,h_m\in \R^{\mathcal{A}}\setminus\{0\}$ ($h_i=(h_{i,\beta})_{\beta\in\mathcal{A}}$ for $1\leq i\leq m$) and exponents $\lambda_1>\lambda_2>\ldots>\lambda_m$ such that
 \begin{equation}\label{cond:lambda}
 \log \frac{d(T^{-1}_*(R^n_*(\nu|_{X^{(n)}})))}{d(R^n_*(\nu|_{X^{(n)}}))}=\sum_{1\leq i\leq m}e^{\lambda_i n}h_i\text{ for all }n\geq 0.
 \end{equation}
{To simplify notation, from now on we will identify functions on $X$ that are constant atoms of the partition $\mathcal P$ with vectors in $\R^{\mathcal A}$. Condition \eqref{cond:lambda} is naturally satisfied provided that we assume
\begin{equation}\label{cond:lambdaind}
\log \frac{d(T^{-1}_*\nu)}{d\nu}=\sum_{1\leq i\leq m}h_i\quad\text{and}\quad(h_i)_A=e^{\lambda_i}\cdot h_i\circ R\ \text{ for all }\ 1\leq i \leq m.
\end{equation}
Recall that $\psi_A$ is the renormalization of $\psi$ given by \eqref{def:phiA}. Indeed, if $\psi=\log \frac{d(T^{-1}_*\nu)}{d\nu}$, then, by \eqref{eq:render}, we have
\begin{align*}
\log \frac{d(T^{-1}_*(R^n_*(\nu|_{X^{(n)}})))}{d(R^n_*(\nu|_{X^{(n)}}))}&=\log \frac{d(R^n_*(T^{-1}_{X^{(n)}})_*(\nu|_{X^{(n)}}))}{d(R^n_*(\nu|_{X^{(n)}}))}\\
&=\log \frac{d((T^{-1}_{X^{(n)}})_*(\nu|_{X^{(n)}}))}{d((\nu|_{X^{(n)}}))}\circ R^{-n}=\psi_{X^{(n)}}\circ R^{-n}.
\end{align*}
It follows that
\[\log \frac{d(T^{-1}_*(R^n_*(\nu|_{X^{(n)}})))}{d(R^n_*(\nu|_{X^{(n)}}))}=\sum_{i=1}^m(h_i)_{X^{(n)}}\circ R^{-n}=\sum_{i=1}^m e^{\lambda_i}h_i.\]
}

 Let us consider any atom $T^jR^{-1}P_\alpha\in\mathcal{P}^{(1)}$ and any integer $0\leq k<q_\alpha$. As $0\leq j<q_\alpha$, the map $\log\frac{d((T^{-(k-j)})_*R^n_*(\nu|_{X^{(n)}}))}{d(R^n_*\nu|_{X^{(n)}})}$ is constant on $T^jR^{-1}P_\alpha$ and is equal to
 \begin{align}\label{eq:iksy}
 \sum_{1\leq i\leq m}e^{\lambda_i n}x^{\alpha,j,k}_i(\nu),\text{ where }
 x^{\alpha,j,k}_i(\nu)=
 \left\{
 \begin{array}{rl}
 \sum_{j\leq p<k}h_{i,\beta(\alpha,p)}&\text{ if } k\geq j\\
 -\sum_{k\leq p<j}h_{i,\beta(\alpha,p)}&\text{ if } k< j.
 \end{array}
 \right.
 \end{align}
 Indeed, if $k>j$, then
 \begin{align*}
 \log\frac{d((T^{-(k-j)})_*R^n_*(\nu|_{X^{(n)}}))}{d(R^n_*\nu|_{X^{(n)}})}|_{T^jR^{-1}P_\alpha}&=
 \sum_{j\leq p<k}\log\frac{d((T^{-1})_*R^n_*(\nu|_{X^{(n)}}))}{d(R^n_*\nu|_{X^{(n)}})}|_{T^pR^{-1}P_\alpha}\\
 &=
 \sum_{j\leq p<k}\log\frac{d((T^{-1})_*R^n_*(\nu|_{X^{(n)}}))}{d(R^n_*\nu|_{X^{(n)}})}|_{P_{\beta(\alpha,p)}}\\
 &=
 \sum_{j\leq p<k}\sum_{1\leq i\leq m}e^{\lambda_i n}h_{i,\beta(\alpha,p)}=\sum_{1\leq i\leq m}e^{\lambda_i n}x^{\alpha,j,k}_i(\nu).
 \end{align*}
 If $k<j$, then
 \begin{align*}
 \log\frac{d((T^{-(k-j)})_*R^n_*(\nu|_{X^{(n)}}))}{d(R^n_*\nu|_{X^{(n)}})}|_{T^jR^{-1}P_\alpha}&=-\log\frac{d((T^{-(j-k)})_*R^n_*(\nu|_{X^{(n)}}))}{d(R^n_*\nu|_{X^{(n)}})}|_{T^kR^{-1}P_\alpha}\\
 &=
 -\sum_{1\leq i\leq m}e^{\lambda_i n}\sum_{k\leq p<j}h_{i,\beta(\alpha,p)}.
 \end{align*}

 A vector $(x_1,x_2,\ldots,x_m)\in\R^m$ is \emph{lexicographically dominated} (lex-dominated) if there exists $1\leq i\leq m+1$ such that
 \[x_j=0\text{ for }j<i\text{ and } |x_i|>\sum_{j>i}|x_j|.\]
 If $(x_1,x_2,\ldots,x_m)\in\R^m$ is lex-dominated, then
 \begin{equation}\label{eq:lexdom}
 \sgn(\sum_{1\leq j\leq m}x_j)=\sgn(x_i)\text{ and }(e^{\lambda_1 n}x_1,e^{\lambda_2 n}x_2,\ldots,e^{\lambda_m n}x_m)\text{ is lex-dominated.}
 \end{equation}
 Moreover, for any $(x_1,x_2,\ldots,x_m)\in\R^m$ we can find $n\geq 1$ large enough such that $(e^{\lambda_1 n}x_1,e^{\lambda_2 n}x_2,\ldots,e^{\lambda_m n}x_m)$ is lex-dominated.

 We will consider measures $\nu$ which are \emph{super-dominated}, this is
 \begin{align}
 \label{def:super1}
 &(x_i^{\alpha,j,k}(\nu))_{1\leq i\leq m}\text{ is lex-dominated for all }\alpha\in\mathcal{A}\text{ and }0\leq j, k<q_\alpha,\text{ and}\\
 \label{def:super2}
 &\text{if }x_1^{\alpha,j,k}(\nu)\neq 0,\text{ then }|x_1^{\alpha,j,k}(\nu)|>\sum_{2\leq i\leq m}\sum_{\beta\in\mathcal{A}}q_\alpha|h_{i,\beta}|.
 \end{align}
 {Note that, by \eqref{eq:iksy}, for any measure $\nu$ satisfying \eqref{cond:lambda} by taking $n\geq 0$ large enough, we may guarantee that $R^n_*(\nu|_{X^{(n)}})$ is super-dominated by replacing the vectors $h_1,\ldots,h_m$ with $e^{\lambda_1n}h_1,\ldots, e^{\lambda_mn}h_m$. Indeed, for every $\alpha\in\A$ and $0\le k,j\le q_{\alpha}$, we have $x_i^{\alpha,j,k}(R^n_*(\nu|_{X^{(n)}}))=e^{\lambda_in}x_i^{\alpha,j,k}(\nu)$, and, since $\lambda_1>\ldots>\lambda_m$, the existence of such $n$ follows.}

 \begin{theorem}\label{thm: imperfect_divergence}
{Assume that $T:X\to X$ is a uniquely ergodic homeomorphism, the atoms of the partition $\mathcal P$ are clopen, and $\operatorname{diam}X^{(n)}\to 0$ as $n\to\infty$.
 Suppose that $\varrho$ is the unique $T$-invariant measure and $\mu$ is conformal measure such that $R_\mu:(X^\theta,\mu^\theta)\to (X^\theta,\mu^\theta)$ is well defined and ergodic.}
 Let $\nu$ be a Borel measure such that $T_*\nu\sim \nu$ and $\frac{d(T^{-1}_*\nu)}{d\nu}$ is constant on the atoms of the partition $\mathcal{P}$.
 Assume that there are vectors $h_1,h_2,\ldots,h_m\in \R^{\mathcal{A}}\setminus\{0\}$ and exponents $\lambda_1>\lambda_2>\ldots>\lambda_m$ such that $\lambda_1>0$, {$\int_Xh_1\,d\varrho=0$,} and
 \begin{equation}\label{cond:lambda1}
 \log \frac{d(T^{-1}_*\nu)}{d\nu}=\sum_{1\leq i\leq m}h_i\quad\text{and}\quad(h_i)_A=e^{\lambda_i}\cdot (h_i\circ R)\ \text{ for all }\ 1\leq i \leq m.
 \end{equation}
 Then for $\mu$-a.e.\ $x\in X$,
 \begin{equation}\label{eq:finconvinf}
 \lim_{n\to\infty}-\frac{1}{n}\log \nu(P^{(n)}(x))=+\infty.
 \end{equation}
 \end{theorem}

 \begin{proof}
 First, note that it is enough to restrict ourselves to considering super-dominated measures $\nu$.
 Indeed, suppose that \eqref{eq:finconvinf} holds whenever $\nu$ is super-dominated.
 As we have already noted, for any measure $\nu$ satisfying \eqref{cond:lambda1} there exists $n_0\geq 0$ such that $R^{n_0}_*(\nu|_{X^{(n_0)}})$ is super-dominated.
 Then, for $\mu^\theta$-a.e.\ $(x,r)\in X^\theta$, we have
 \[-\frac{1}{n}\log R^{n_0}_*(\nu|_{X^{(n_0)}})(P^{(n)}(\pi(R_\mu^{n_0}(x,r))))\to+\infty.\]
 In view of \eqref{eq:passP1},
 \begin{align*}
 P^{(n+n_0)}(x) = T^jR^{-n_0}P^{(n)}(\pi(R^{n_0}_{\mu}(x,r)))\text{ for some }0\leq j<\max_{\alpha\in\mathcal{A}}q_\alpha^{(n_0)}.
 \end{align*}
 Let $C:=\|\log\frac{d(T^{-1}_*\nu)}{d\nu}\|_{\sup}$. Then
 \begin{align*}
 \nu(P^{(n+n_0)}(x))&=\nu(T^jR^{-n_0}P^{(n)}(\pi(R^{n_0}_{\mu}(x,r))))\geq e^{-jC}\nu(R^{-n_0}P^{(n)}(\pi(R_\mu^{n_0}(x,r))))\\
 &\geq e^{-C\max_{\alpha\in\mathcal{A}}q_\alpha^{(n_0)}}
 R^{n_0}_*(\nu|_{X^{(n_0)}})(P^{(n)}(\pi(R_\mu^{n_0}(x,r)))).
 \end{align*}
 This gives \eqref{eq:finconvinf}.

 From now on, we will assume that the measure $\nu$ is super-dominated.
{As we have already noted, \eqref{cond:lambda1} implies
\begin{equation*}
 \log \frac{d(T^{-1}_*(R^n_*(\nu|_{X^{(n)}})))}{d(R^n_*(\nu|_{X^{(n)}}))}=\sum_{1\leq i\leq m}e^{\lambda_i n}h_i\quad\text{for all}\quad n\geq 0.
\end{equation*}
}
Then, by \eqref{eq:lexdom}, for any atom $T^jR^{-1}P_\alpha\in\mathcal{P}^{(1)}$ and any $0\leq k<q_\alpha$, we have $\log\frac{d(T^{-(k-j)}_*\nu)}{d\nu}$ and $\log \frac{d((T^{-(k-j)})_*R^n_*(\nu|_{X^{(n)}}))}{d(R^n_*\nu|_{X^{(n)}})}$ are constant on $T^jR^{-1}P_\alpha$ and have the same sign. Moreover, \eqref{eq:lank1} holds with
 \[e^{\lambda_{k-j}^{(n)}}=\frac{\sum_{1\leq i\leq m}e^{\lambda_i n}x^{\alpha,j,k}_i(\nu)}{\sum_{1\leq i\leq m}x^{\alpha,j,k}_i(\nu)}\quad\text{on}\quad T^jR^{-1}P_\alpha.\]

 {Let us consider a conformal measure $\nu_1$ such that $\log \frac{d(T^{-1})_*\nu_1}{d\nu_1}=h_1\neq 0$. As $T$ is a uniquely ergodic homeomorphism and $h_1$ is continuous with zero mean, the existence of this conformal measure follows directly from Proposition~\ref{prop:conformal_measures_general}. Moreover, in view of \eqref{eq:render}, we have
\begin{align*}
\log \frac{d(T^{-1}_*(R_*(\nu_1|_{A})))}{d(R_*(\nu_1|_{A}))}&=\log \frac{d(R_*(T^{-1}_{A})_*(\nu_1|_{A}))}{d(R_*(\nu_1|_{A}))}
=\log \frac{d((T^{-1}_{A})_*(\nu_1|_{A}))}{d((\nu_1|_{A}))}\circ R\\&=(h_1)_{A}\circ R=e^{\lambda_1}h_1=e^{\lambda_1}\log \frac{d(T^{-1})_*\nu_1}{d\nu_1},
\end{align*}
so the measure $\nu_1$ satisfies the condition of perfect scaling by renormalization \eqref{eq:condlambda}.}
Since $\nu_1$ is not $T$-invariant, in view of Lemma~\ref{lem:Xn}, we have $g^{\nu_1}_A$ is non-zero, maybe replacing $A$ by $X^{(n)}$.
 It follows that, there exist $\alpha\in\mathcal{A}$ and $0\leq j\neq k<q_\alpha$ such that $\beta(\alpha,j)=\beta(\alpha,k)=\beta$ and
 \[x^{\alpha,j,k}_1(\nu)=\log \frac{d(T^{-(k-j)})_*\nu_1}{d\nu_1}\quad\text{on}\quad T^jR^{-1}P_\alpha\quad\text{ is positive.}\]
 By \eqref{def:super1}, we have $\sum_{1\leq i\leq m}x^{\alpha,j,k}_i(\nu)>0$.

 Let $l=l(x,r)$ for $x\in T^jR^{-1}P_\alpha$.
 Then, $0\leq l<q_\alpha$, $\beta(\alpha,l)=\beta$, and
 \begin{align*}
 g^\nu_A(x,r)&=\sum_{1\leq i\leq m}x^{\alpha,j,l}_i(\nu)=\log\frac{\nu(T^lR^{-1}P_\alpha)}{\nu(T^jR^{-1}P_\alpha)}\\
 &\geq\log\frac{\nu(T^kR^{-1}P_\alpha)}{\nu(T^jR^{-1}P_\alpha)}=
 \!\sum_{1\leq i\leq m}x^{\alpha,j,k}_i(\nu)>0.
 \end{align*}
 Moreover, in view of \eqref{def:super2},
 \begin{align*}
 x^{\alpha,j,l}_1(\nu)&\geq x^{\alpha,j,k}_1(\nu)+\sum_{2\leq i\leq m}x^{\alpha,j,k}_i(\nu)-\sum_{2\leq i\leq m}x^{\alpha,j,l}_i(\nu)\\
 &\geq
 x_1^{\alpha,j,k}(\nu)-\sum_{2\leq i\leq m}\sum_{\beta\in\mathcal{A}}q_\alpha|h_{i,\beta}|>0.
 \end{align*}
 As $\lambda_1>0$, it follows that if $x\in T^jR^{-1}P_\alpha$, then
 \[e^{\lambda^{(n)}(x,r)}=e^{\lambda_{l-j}^{(n)}(x,r)}=\frac{\sum_{1\leq i\leq m}e^{\lambda_i n}x^{\alpha,j,l}_i(\nu)}{\sum_{1\leq i\leq m}x^{\alpha,j,l}_i(\nu)}\to +\infty.\]
 Denote by $B\subset X^\theta$ the suspension over $T^jR^{-1}P_\alpha$. Then $g^\nu_A$ is positive on $B$ and
 \[\lim_{n\to\infty}\inf\{e^{\lambda^{(n)}(x,r)}\mid (x,r)\in B\}=+\infty.\]
 Using Proposition~\ref{thm:inf3}, we obtain \eqref{eq:finconvinf}.
 \end{proof}

\section{Markov property}\label{sec:Markov}
Let $T:X\to X$ be a Borel bijection and $\mu$ be a Borel measure such that $T_*\mu\sim \mu$ and $\mu$ is ergodic for $T$. Let $\mathcal{P}=(P_\alpha)_{\alpha\in \mathcal A}$ be a finite Borel partition such that $\phi = \log\frac{d(T^{-1}_*\mu)}{d\mu}$ is constant on the atoms of the partition $\mathcal{P}$. We denote by $\mu_\alpha$, $\alpha\in\mathcal A$ the $\mu$-measures of atoms $P_\alpha$, $\alpha\in\mathcal A$, and by $\phi_\alpha$, $\alpha\in\mathcal A$ the values of $\phi$ on atoms $P_\alpha$, $\alpha\in\mathcal A$. Suppose that $\mu$ is the unique $\phi$-conformal measure for $T$. Let $R:A\to X$ be a Borel bijection such that $(R^{-1})_*\mu=e^{-\rho_\mu}\mu|_{A}$ with $\rho_\mu=-\log\mu(A)$. Recall that, by Lemma~\ref{lem:renmeas}, this is equivalent to assuming that $\phi\circ R=\phi_A$. Suppose that for every $\alpha\in\mathcal A$ there exists $q_\alpha\in\N$ which is the first return time for $T$ of all elements of $P^{(1)}_\alpha:=R^{-1}P_\alpha\subset A$ to $A$. Then
\[\mathcal{P}^{(1)}:=\{T^jP^{(1)}_\alpha\mid (\alpha,j)\in\Sigma\}, \quad\text{with}\quad \Sigma:=\{(\alpha,j)\mid \alpha\in\mathcal{A},0\leq j<q_\alpha\},\]
is a Borel partition of $X$. Suppose that
\begin{equation*}
\text{for any $(\alpha,j)\in\Sigma$ there exists $\beta(\alpha,j)\in\mathcal{A}$ such that $T^jP^{(1)}_\alpha\subset P_{\beta(\alpha,j)}$.}
\end{equation*}
Denote by $M=[M_{\alpha\beta}]_{\alpha,\beta\in\mathcal A}$ the incidence matrix for the renormalization map $R$, i.e.
\[M_{\alpha\beta}:=\#\{0\leq j<q_\alpha\mid \beta(\alpha,j)=\beta\}.\]
Then the condition $\phi\circ R=\phi_A$ is equivalent to $M\phi=\phi$, where $\phi$ is here treated as the vector $(\phi_\alpha)_{\alpha\in\mathcal A}\in \R^{\mathcal A}$.

Suppose that $\theta:X\to\R_{>0}$ is a function constant on atoms of the partition $\mathcal P$ such that $\int_X\theta\, d\mu=1$ and $\theta_A\circ R^{-1}=e^{\rho_\mu}\theta$. Denote by $\theta_\alpha$, $\alpha\in\mathcal A$ the values of $\theta$ on atoms $P_\alpha$, $\alpha\in\mathcal A$. Let us consider the matrix $M(\phi)=[M(\phi)_{\alpha\beta}]_{\alpha,\beta\in\mathcal A}$ given by
\[M(\phi)_{\alpha\beta}:=\sum_{\substack{0\leq j<q_\alpha\\
\beta(\alpha,j)=\beta}} e^{\sum_{0\leq k <j}\phi_{\beta(\alpha,k)}}.\]
Then the condition $(R^{-1})_*\mu=e^{-\rho_\mu}\mu|_{A}$ gives $\mu M(\phi)=e^{\rho_{\mu}}\mu$, where $\mu$ is here treated as the vector $(\mu_\alpha)_{\alpha\in\mathcal A}\in \R_{>0}^{\mathcal A}$, and the condition $\theta_A\circ R^{-1}=e^{\rho_\mu}\theta$ is equivalent to $M(\phi)\theta=e^{\rho_{\mu}}\theta$, where $\theta$ is here treated as the vector $(\theta_\alpha)_{\alpha\in\mathcal A}\in \R_{>0}^{\mathcal A}$.

\begin{remark}
Note that if the incidence $M$ is positive, i.e.\ all its entries are positive, then $M(\phi)$ is positive and by the Perron-Frobenius theorem $e^{\rho_\mu}$ is the principal (Perron-Frobenius) eigenvalue of $M(\phi)$ and the vectors $\mu$ and $\theta$ are its the right and the left Perron-Frobenius eigenvectors respectively.
\end{remark}

Let us consider the partition $\mathcal{Q}^{(1)}=\{Q^{(1)}_{(\alpha,j)}\mid (\alpha,j)\in\Sigma\}$ defined by
\[Q^{(1)}_{(\alpha,j)}=(T^jP^{(1)}_\alpha)\times[0,\theta_{\beta(\alpha,j)}).\]
Under the assumptions we have already made, we defined in Section~\ref{sc:renormalization_map} (see \eqref{def:Rmu}) the Borel bijection $R_\mu:X^\theta\to X^\theta$ given by
\[R_\mu(x,r)=(R(T^{-j}x),e^{-\rho_\mu}(e^{S_{(\alpha,j)}\phi} r+S^{\phi}_{(\alpha,j)}\theta))\text{ if } (x,r)\in Q^{(1)}_{\alpha,j}=(T^jP^{(1)}_\alpha)\times[0,\theta_{\beta(\alpha,j)}),\]
where
\[S_{(\alpha,j)}\phi=S_j\phi|_{P^{(1)}_\alpha}=\sum_{0\leq k<j}\phi_{\beta(\alpha,k)}\text{ and }S^{\phi}_{(\alpha,j)}\theta=S_j^{\phi}\theta|_{P^{(1)}_\alpha}=\sum_{0\leq k<j}e^{S_{(\alpha,k)}\phi}\theta_{\beta(\alpha,k)}.\]		
Recall that the map $R_{\mu}:X^\theta\to X^\theta$ preserves the probability measure $\mu^\theta$.

Let us consider the coding map $\mathfrak{q}:X^\theta\to\Sigma$ associated to the partition $\mathcal{Q}^{(1)}$, i.e.\ $\mathfrak{q}(x,r)=(\alpha,j)$
if and only if $(x,r)\in Q^{(1)}_{(\alpha,j)}$. Denote by $\widehat{\Sigma}$ the subset of all $\big((\alpha_j,i_j)\big)_{j\in\Z}$ such that $\beta(\alpha_j,i_j)={\alpha_{j-1}}$ for all $j\in\Z$. Then $\widehat{\Sigma}\subset \Sigma^\Z$ is a subset invariant under the left shift $\sigma$ and $\sigma:\widehat{\Sigma}\to\widehat{\Sigma}$ is a shift of finite type.

\begin{lemma}\label{lem:markow-chain}
Under the assumptions made in Section~\ref{sec:Markov}, the process $(\mathfrak{q}\circ R_\mu^n)_{n\in\Z}$ on $(X^\theta,\mu^\theta)$ is a stationary Markov chain. Moreover, if
\begin{equation}\label{eq:tozero}
\lim_{n\to\infty}\max_{J\in\mathcal P^{(n)}}\mu(J)=0\quad\text{and}\quad\lim_{n\to\infty}\max_{J\in\mathcal P^{(n)}}\operatorname{diam}(J)=0,
\end{equation}
then the map $\mathfrak{Q}:X^\theta\to \widehat{\Sigma}$ given by $\mathfrak{Q}(x,r)=(\mathfrak{q}(R_\mu^n(x,r)))_{n\in\Z}$ establishes a measure-preserving isomorphism between $R_\mu$ on $(X^\theta,\mu^\theta)$ and the Markov shift $\sigma$ on $(\widehat{\Sigma},\mathfrak{m})$, where $\mathfrak{m}:=\mathfrak{Q}_*(\mu^{\theta})$ is Markov measure with the transition matrix $\mathfrak M=[\mathfrak M_{(\alpha,i),(\beta,j)}]_{(\alpha,i),(\beta,j)\in\Sigma}$ given by
\begin{equation}\label{eq:trmat}
\mathfrak M_{(\alpha,i),(\beta,j)}=
\left\{
\begin{array}{cc}
0 &\text{if }\beta(\beta,j)\neq\alpha,\\
\mu\big(T^jP^{(1)}_\beta|P_\alpha\big) &\text{if }\beta(\beta,j)=\alpha.
\end{array}
\right.
\end{equation}
\end{lemma}
\begin{proof}
First note that for any pair $(\alpha_0,i_0)$, $(\alpha_1,i_1)$ in $\Sigma$, we have
\begin{equation}\label{eq:inter}
Q^{(1)}_{(\alpha_0,i_0)}\cap R^{-1}_\mu Q^{(1)}_{(\alpha_1,i_1)}=
\left\{
\begin{array}{cl}
\emptyset &\text{if }\beta(\alpha_1,i_1)\neq\alpha_0,\\
\!T^{i_0}R^{-1}T^{i_1}R^{-1}P_{\alpha_1}\times[0,\theta_{\beta(\alpha_0,i_0)})\! &\text{if }\beta(\alpha_1,i_1)=\alpha_0.
\end{array}
\right.
\end{equation}
Indeed, $(x,r)\in Q^{(1)}_{(\alpha_0,i_0)}\cap R^{-1}_\mu Q^{(1)}_{(\alpha_1,i_1)}$ iff $(x,r)\in T^{i_0}R^{-1}P_{\alpha_0}\times[0,\theta_{\beta(\alpha_0,i_0)})$ and
\begin{equation}\label{eq:seccon}
(RT^{-i_0}x,e^{-\rho_\mu}(e^{S_{(\alpha_0,i_0)}\phi} r+S^{\phi}_{(\alpha_0,i_0)}\theta))=R_\mu(x,r)\in T^{i_1}R^{-1}P_{\alpha_1}\times[0,\theta_{\beta(\alpha_1,i_1)}).
\end{equation}
Moreover, $RT^{-i_0}x\in P_{\alpha_0}$ and $RT^{-i_0}x\in T^{i_1}R^{-1}P_{\alpha_1}\subset P_{\beta(\alpha_1,i_1)}$. It follows that if $\beta(\alpha_1,i_1)\neq\alpha_0$, then the intersection is empty. Assume that $\beta(\alpha_1,i_1)=\alpha_0$. Then \eqref{eq:seccon} is equivalent to
\[x\in T^{i_0}R^{-1}T^{i_1}R^{-1}P_{\alpha_1}\text{ with }r\in e^{-S_{(\alpha_0,i_0)}\phi}(e^{\rho_\mu}[0,\theta_{\alpha_0})-S^{\phi}_{(\alpha_0,i_0)}\theta),\]
where the last interval contains $[0,\theta_{\beta(\alpha_0,i_0)})$. Indeed, as $i_0<q_{\alpha_0}$ and $\theta_A\circ R^{-1}=e^{\rho_\mu}\theta$, we have
\begin{gather*}
S^{\phi}_{(\alpha_0,i_0)}\theta+e^{S_{(\alpha_0,i_0)}\phi}\theta_{\beta(\alpha_0,i_0)}=S^{\phi}_{i_0+1}\theta|_{P^{(1)}_{\alpha_0}}
\leq S^{\phi}_{q_{\alpha_0}}\theta|_{P^{(1)}_{\alpha_0}}\\
=\theta_A|_{R^{-1}P_{\alpha_0}}=\theta_A\circ R^{-1}|_{P_{\alpha_0}}=e^{\rho_\mu}\theta_{\alpha_0}.
\end{gather*}
As $S^{\phi}_{(\alpha_0,i_0)}\theta\geq 0$, it follows that
\[[0,\theta_{\beta(\alpha_0,i_0)})\subset e^{-S_{(\alpha_0,i_0)}\phi}(e^{\rho_\mu}[0,\theta_{\alpha_0})-S^{\phi}_{(\alpha_0,i_0)}\theta).\]
This gives the second line in \eqref{eq:inter}. Moreover, using the form of the map $R_{\mu}$ (see \eqref{eq:seccon}), we also have
\[R_\mu Q^{(1)}_{(\alpha_0,i_0)}\cap Q^{(1)}_{(\alpha_1,i_1)}=T^{i_1}P^{(1)}_{\alpha_1}\times (e^{-\rho_\mu+S_{(\alpha_0,i_0)}\phi}[0,\theta_{\beta(\alpha_0,i_0)})+c).\]
Repeating the above arguments inductively, we get that for any sequence $((\alpha_j,i_j))_{j=0}^n$ in $\Sigma$ such that $\beta(\alpha_j,i_j)=\alpha_{j-1}$ for $1\leq j\leq n$, we have
\begin{align}
\label{eq:prod1}
\bigcap_{j=0}^nR_{\mu}^{-j}Q^{(1)}_{(\alpha_j,i_j)}&=T^{i_0}R^{-1}\ldots T^{i_n}R^{-1}P_{\alpha_n}\times[0,\theta_{\beta(\alpha_0,i_0)})\\
\bigcap_{j=0}^nR_{\mu}^{n-j}Q^{(1)}_{(\alpha_j,i_j)}&=T^{i_n}P^{(1)}_{\alpha_n}\times
\left(e^{-n\rho_\mu+\sum_{j=0}^{n-1}S_{(\alpha_j,i_j)}\phi}[0,\theta_{\beta(\alpha_0,i_0)})+c_n\right).
\end{align}
As $T^{i_0}R^{-1}\ldots T^{i_n}R^{-1}P_{\alpha_n}\in \mathcal{P}^{n+1}$ and the $\mu^\theta$-measure of the above two sets are equal, we get
\begin{align}
\label{eq:width}
\operatorname{width}\Big(\bigcap_{j=0}^nR_{\mu}^{-j}Q^{(1)}_{(\alpha_j,i_j)}\Big)&\leq \max_{J\in\mathcal P^{(n+1)}}\operatorname{diam}(J),\\
\label{eq:hight}
\operatorname{height}\Big(\bigcap_{j=0}^nR_{\mu}^{n-j}Q^{(1)}_{(\alpha_j,i_j)}\Big)&\leq \frac{\max_{\alpha\in\mathcal{A}}\theta_\alpha}{\min_{J\in\mathcal{P}^{(1)}}\mu(J)}
\max_{J\in\mathcal P^{(n+1)}}\mu(J).
\end{align}
In view of \eqref{eq:prod1}, we have
\begin{align*}
&\mu^{\theta}\left(\mathfrak{q}\circ R_\mu^{n} \in (\alpha_{n},i_{n}) \Bigm| \bigcap_{j=0}^{n-1} \mathfrak{q}\circ R_\mu^{j} \in (\alpha_{j},i_{j})\right)=
\frac{\mu^\theta\left(\bigcap_{j=0}^{n}R_{\mu}^{-j}Q^{(1)}_{(\alpha_j,i_j)}\right)}{\mu^\theta\left(\bigcap_{j=0}^{n-1}R_{\mu}^{-j}Q^{(1)}_{(\alpha_j,i_j)}\right)}\\
&=\frac{\mu^\theta\left(T^{i_0}R^{-1}\ldots T^{i_n}R^{-1}P_{\alpha_n}\times[0,\theta_{\beta(\alpha_0,i_0)})\right)}{\mu^\theta\left(T^{i_0}R^{-1}\ldots T^{i_{n-1}}R^{-1}P_{\alpha_{n-1}}\times[0,\theta_{\beta(\alpha_0,i_0)})\right)}
=\frac{\mu\left(T^{i_0}R^{-1}\ldots T^{i_n}R^{-1}P_{\alpha_n}\right)}{\mu\left(T^{i_0}R^{-1}\ldots T^{i_{n-1}}R^{-1}P_{\alpha_{n-1}}\right)}.
\end{align*}
Recall that for any $0\leq l< k\leq n$ we have $T^{i_{l+1}}R^{-1}\ldots T^{i_k}R^{-1}P_{\alpha_k}\subset P_{\beta(\alpha_{l+1},i_{l+1})}=P_{\alpha_{i_l}}$. Then $R^{-1}T^{i_{l+1}}R^{-1}\ldots T^{i_k}R^{-1}P_{\alpha_k}\subset P^{(1)}_{\alpha_{i_l}}$, and hence, by $(R^{-1})_*\mu=e^{-\rho_\mu}\mu|_{A}$, we have
\begin{align*}
\mu\left(T^{i_{l}}R^{-1}T^{i_{l+1}}R^{-1}\ldots T^{i_k}R^{-1}P_{\alpha_k}\right)&=e^{S_{(\alpha_l,i_l)}\phi}
\mu\left(R^{-1}T^{i_{l+1}}R^{-1}\ldots T^{i_k}R^{-1}P_{\alpha_k}\right)\\
&=e^{-\rho_\mu+S_{(\alpha_l,i_l)}\phi}
\mu\left(T^{i_{l+1}}R^{-1}\ldots T^{i_k}R^{-1}P_{\alpha_k}\right).
\end{align*}
It follows that
\[\frac{\mu\left(T^{i_0}R^{-1}\ldots T^{i_n}R^{-1}P_{\alpha_n}\right)}{\mu\left(T^{i_0}R^{-1}\ldots T^{i_{n-1}}R^{-1}P_{\alpha_{n-1}}\right)}
=\frac{\mu\left( T^{i_n}R^{-1}P_{\alpha_n}\right)}{\mu\left( P_{\alpha_{n-1}}\right)}=\mu(T^{i_n}P^{(1)}_{\alpha_n}|P_{\alpha_{n-1}}).\]
This shows that $(\mathfrak{q}\circ R_\mu^n)_{n\in\Z}$ on $(X^\theta,\mu^\theta)$ is a stationary Markov chain with the transition matrix given by \eqref{eq:trmat}.

Suppose that the condition \eqref{eq:tozero} holds. Then, by \eqref{eq:width} and \eqref{eq:hight}, for any sequence $\big((\alpha_j,i_j)\big)_{j\in\Z}$ in $\widehat \Sigma$, this gives
\[\lim_{n\to\infty}\operatorname{diam}\Big(\bigcap_{j=-n}^nR_{\mu}^{-j}Q^{(1)}_{(\alpha_j,i_j)}\Big)=0.\]
It follows that $\mathfrak{Q}:(X^\theta,\mu^\theta)\to (\widehat{\Sigma},\mathfrak{m})$ is a measure-preserving isomorphism.
\end{proof}

Finally, we obtain a result on the ergodicity of the renormalization process used to deduce the main results of this paper.

\begin{lemma}\label{lem:irrMar}
Suppose that the incidence matrix $M$ is positive. Then the stationary Markov chain $(\mathfrak{q}\circ R_\mu^n)_{n\in\Z}$ is irreducible, in particular is ergodic. If additionally we assume that \eqref{eq:tozero} holds, then $R_\mu$ is also ergodic.
\end{lemma}

\begin{proof}
It suffices to show that all entries of the matrix are $\mathfrak M^2 $ positive. Take any pair $(\alpha_0,i_0),(\alpha_2,i_2) \in \Sigma$ and denote $\beta(\alpha_2,i_2) = \alpha_1 \in \mathcal{A}$.
Since $B_{\alpha_1\alpha_0} \geq 1$,
there exists $0\leq i_1 <q_{\alpha_1}$ satisfying $T^{i_1}R^{-1}P_{\alpha_1} \subset P_{\alpha_0}$. Hence $\beta(\alpha_2,i_2) = \alpha_1$ and $\beta(\alpha_1,i_1) = \alpha_0$, so
\begin{align*}
\mathfrak M^2_{(\alpha_0,i_0)(\alpha_2,i_2)}&=
\sum_{\alpha \in \mathcal{A}}\sum_{0\leq i<q_{\alpha}} \mathfrak M_{(\alpha_0,i_0)(\alpha,i)}\cdot \mathfrak M_{(\alpha,i),(\alpha_2,i_2)}\\
& \geq \mathfrak M_{(\alpha_0,i_0)(\alpha_1,i_1)}\cdot \mathfrak M_{(\alpha_1,i_1),(\alpha_2,i_2)} > 0.
\end{align*}
The last inequality follows directly from \eqref{eq:trmat} and the choice of $(\alpha_1,i_1)$.
\end{proof}

\section{Information content of invariant measures of AIETs - proof of Proposition~\ref{prop: centrallimit} and Proposition~\ref{prop: unstablelimit}}\label{sec: HDproof}
In this section, we will apply the results of previous sections to obtain the convergences of information content, which are required in the proofs of the main results concerning the Hausdorff dimension of invariant measures of AIETs. In fact, the proofs of Proposition~\ref{prop: centrallimit} and Proposition~\ref{prop: unstablelimit} consist in justifying that we can use Theorem~\ref{thm:perfect_zero} and Theorem~\ref{thm: imperfect_divergence}, respectively. For this purpose, we need to build a dictionary that identifies the concepts used in both theorems. Since the dictionary of concepts in the proofs of both propositions is the same, and only the final arguments differ, we have combined both proofs into one.

\begin{proof}[Proof of Proposition~\ref{prop: centrallimit} and Proposition~\ref{prop: unstablelimit}]
Let $f\in\operatorname{Aff}(T,\omega)$ be an AIET of hyperbolic periodic type, semi-conjugated to a self-similar IET $T=(\pi,\lambda)$, with the vector of logarithms of slopes $\omega$ and let $\mu_f$ be the invariant measure of $f$. Moreover, let $\widehat h$ be the semi-conjugacy between $f$ and $\widehat T$, as introduced in Section~\ref{subs: AIET and conformal} (see \eqref{eq:semiconjugacy_cantor}). Let $M$ be the self-similarity matrix of $T$ of period $N$, with $\rho_T$ being the logarithm of its Perron-Frobenius eigenvalue, i.e.\ $\lambda M=e^{\rho_T}\lambda$. Let $\theta\in \R^{\mathcal A}_{>0}$ be the right Perron-Frobenius eigenvector of $M$, i.e.\ $M \theta=e^{\rho_T}\theta$

To apply the results of previous sections, we first present a list of notions, which in the considered case correspond to the abstract objects in Sections~\ref{sec: skewproducts},~\ref{sec: renormalization}, and~\ref{sec: dynpar}. We take the following objects:
\begin{itemize}
\item for $(X,T)$ we take $(K,\widehat T)$ - the Cantor model of the IET $T$, described in Section~\ref{sc:cantor_model};
\item for $\mu$ we substitute the measure $\widehat h_{*}\mu_f$, where $\mu_f$ is the unique $f$-invariant measure. Then $\mu$ is the unique $\widehat T$-invariant measure equal to $(i^+)_*(Leb)=(i^-)_*(Leb)$ (here $Leb$ is seen as the invariant measure of $T$, so using the notation from Section~\ref{subs: AIET and conformal}, we have $\mu=\widehat{Leb}$);
\item by $\widehat T$-invariance of $\mu$, we have $\phi=0$;
\item for the set $A$, we take $\widehat I^{(N)}:=\overline{i^+(I^{(N)})}$;
\item the mapping $r_{A}$ is given by $r_{A}(x)=q_{\alpha}^{(N)}$ for $x\in \widehat I_{\alpha}^{(N)}:=\overline{i^+(I^{(N)}_{\alpha})}$, $\alpha\in\A$ or in other words, $r_{A}$ describes the heigths of Rokhlin towers obtained by Rauzy-Veech induction in a single period;
\item the partition $\mathcal P$ is $(\widehat I_{\alpha})_{\alpha\in\A}$;
\item then the dynamical partitions $\mathcal P^{(n)}$ coincides with $\widehat{\mathcal Q}_{N\cdot n}$, in particular, they consist of clopen sets;
\item the rescaling $R:A\to X$ is taken to be $\widehat R_{\rho_T}$ defined by \eqref{def: hat R}, then
\[\rho_\mu=-\log\mu(\widehat I^{(N)})=-\log Leb(I^{(N)})=\rho_{T};\]
\item the rescaled domains are
\[
X^{(n)}=\widehat R^n_{\rho_T}(K)=\bigcup_{\alpha\in\A}\widehat I^{(n N)}_{\alpha};
\]
\item as $\theta$ we take the right Perron-Frobenius eigenvector of $M$, identifying $\theta:=\widehat{\phi}_\theta$, we have $\theta_A(x)=S^\phi_{q_{\alpha}^{(N)}}\theta(x)=S_{q_{\alpha}^{(N)}}\theta(x)=(M\theta)_\alpha=e^{\rho_T}\theta_\alpha$ for $x\in R^{-1}\widehat{I}_\alpha$, which gives $\theta_A\circ R^{-1}=e^{\rho_\mu}\theta$;
\item as $\theta_A\circ R^{-1}=e^{\rho_\mu}\theta$, we can define the renormalization map $R_\mu:X^\theta\to X^\theta$; and
\item finally take $\nu:=\widehat{\nu}_\omega=\widehat h_*(Leb)$ the $\widehat{\phi}_\omega$-conformal measure for $\widehat T$ and $\psi:=\widehat{\phi}_\omega$, if $\omega$ is of unstable type, and
 take $\nu:=\widehat{\nu}_{\omega_c}$ the $\widehat{\phi}_{\omega_c}$-conformal measure for $\widehat T$ and $\psi:=\widehat{\phi}_{\omega_c}$, if $\omega$ is of central stable type, where $\omega=\omega_c+\omega_s$ is the decomposition into invariant and stable vectors.
\end{itemize}

We begin by checking whether $R_{\mu}$ is ergodic, and to do it, we want to apply Lemma~\ref{lem:irrMar}. Thus, we need to verify its assumptions. First, by \eqref{RVmatrix meaning}, we have that the incidence matrix $M$ of $R$ is equal to the self-similarity matrix $M$ for $T$. Since $T$ is hyperbolically self-similar, $M$ is positive.
Now we need to verify that the assumption \eqref{eq:tozero} holds. Note that, by Lemma~\ref{lem:uniformlimit}, we have
the first convergence in \eqref{eq:tozero}. The second convergence follows directly from \eqref{diampart}. Thus, by Lemma~\ref{lem:irrMar}, the renormalization map $ R_{\mu}$ is ergodic.

\medskip

\noindent
\textbf{Central-stable case.}
Suppose that $\omega$ is of central-stable type, so we pass to the proof of Proposition~\ref{prop: centrallimit}. As $\omega_c$ is invariant for $M$, for every $x\in R^{-1}\widehat{I}_\alpha$, we have
\[(\widehat{\phi}_{\omega_c})_{A}(x)=S_{q_{\alpha}^{(N)}}\widehat{\phi}_{\omega_c}(x)=(M\omega_c)_\alpha=(\omega_c)_\alpha. \]
It follows that $\psi_A\circ R^{-1}=(\widehat{\phi}_{\omega_c})_A\circ R^{-1}=\widehat{\phi}_{\omega_c}=\psi$. Thus, the measure $\nu$ satisfies the perfect rescaling condition \eqref{eq:condlambda} with $\lambda=0$.
Applying Theorem~\ref{thm:perfect_zero}, we get
\[ \lim_{n\to\infty}-\frac{1}{n}\log \nu(P^{(n)}(x))=H^u_\nu(R_\mu)\quad\text{for $\mu$-a.e. }x\in K.\]
As $\omega_c$ is of central type, by Proposition~\ref{prop:conformal_measures}, the measures $\nu_{\omega_c}$ and $\widehat \nu_{\omega_c}$ are continuous. Thus, the map $i_+:I\to K$ establishes a measure-theoretical isomorphism between $T$ on $(I,Leb)$ and $\widehat T$ on $(X,\mu)$ such that $(i_+)_*(\nu_{\omega_c})=\widehat \nu_{\omega_c}=\nu$. Since $i_+$ preserves the dynamical partition, it follows that
\[ \lim_{n\to\infty}-\frac{1}{n}\log \nu_{\omega_c}(P_T^{(n)}(x))=H^u_\nu(R_\mu)\quad\text{for $Leb$-a.e. }x\in I.\]
In view of \eqref{eq:partitiopassintocentral} (see Lemma~\ref{lem:from_cs_to_c}), we have
\[ \lim_{n\to\infty}-\frac{1}{n}\log \nu_{\omega}(P_T^{(n)}(x))=H^u_\nu(R_\mu)\quad\text{for $Leb$-a.e. }x\in I.\]
As $h:I\to I$ establishes a measure-theoretical isomorphism between $f$ on $(I,\mu_f)$ and $T$ on $(I,Leb)$, which preserves the dynamical partition, such that $(h^{-1})_*(\nu_{\omega})=Leb$, this gives
\[ \lim_{n\to\infty}-\frac{1}{n}\log Leb(P_f^{(n)}(x))=H^u_\nu(R_\mu)\quad\text{for $\mu_f$-a.e. }x\in I.\]
Moreover, by Theorem~\ref{thm:difference_conformal}, we have $H^u_\nu(R_\mu)>\rho_\mu=\rho_T$, whenever $\omega_c\neq 0$.

We finish by showing that $H^u_\nu(R_\mu)=\mathcal G(T,\omega)$, where $\mathcal{G}(T, \omega)$ is defined by \eqref{eq: HD_invariant_formula}. In view of Theorem~\ref{thm:perfect_zero},
\[H^u_\nu(R_\mu)=-\sum_{\alpha\in \mathcal{A}}\sum_{0\leq i<q_\alpha}\log\frac{\nu(T^iR^{-1}P_\alpha)}{\nu(P_{\alpha})}\mu(T^iR^{-1}P_\alpha)\theta_{\beta(\alpha,i)}.\]
Using again the fact that the map $i_+:I\to K$ establishes a measure-theoretical isomorphism between $T$ on $(I,Leb)$ and $\widehat T$ on $(X,\mu)$ such that $(i_+)_*(\nu_{\omega_c})=\widehat \nu_{\omega_c}=\nu$, we get
\[H^u_\nu(R_\mu)=-\sum_{\alpha\in \mathcal{A}}\sum_{0\leq i<q_\alpha}\log\frac{\nu_{\omega_c}(T^iR^{-1}I_\alpha)}{\nu_{\omega_c}(I_{\alpha})}Leb(T^iR^{-1}I_\alpha)\theta_{\beta(\alpha,i)}.\]
In view of \eqref{eq: Rnuvc}, we have
\begin{equation}\label{eq: nuomc1}
\nu_{\omega_c}(I_{\alpha})=e^{\rho_{c}}\nu_{\omega_c}(R^{-1}I_{\alpha}),
\end{equation}
where $\rho_c$ is the logarithm of the Perron-Frobenius eigenvalue of $M^{(N)}_{\pi,\lambda,\omega_c}$. As
\[\sum_{\alpha\in \mathcal{A}}\sum_{0\leq i<q_\alpha}Leb(T^iR^{-1}I_\alpha)\theta_{\beta(\alpha,i)}=\langle\lambda,\theta\rangle=1,\]
this gives
\[H^u_\nu(R_\mu)=\rho_c-\sum_{\alpha\in \mathcal{A}}\sum_{0\leq i<q_\alpha}\log\frac{\nu_{\omega_c}(T^iR^{-1}I_\alpha)}{\nu_{\omega_c}(R^{-1}I_{\alpha})}Leb(T^iR^{-1}I_\alpha)\theta_{\beta(\alpha,i)}.\]
As $T$ is self-similar, we have $Leb(T^iR^{-1}I_\alpha)=Leb(I^{(N)}_\alpha)=e^{-\rho_T}\lambda_\alpha$. Moreover, as $\nu_{\omega_c}$ is $\phi^T_{\omega_c}$-conformal, we have
\begin{equation}\label{eq: nuomc2}
\frac{\nu_{\omega_c}(T^iR^{-1}I_\alpha)}{\nu_{\omega_c}(R^{-1}I_{\alpha})}=\frac{\nu_{\omega_c}(T^iI^{(N)}_\alpha)}{\nu_{\omega_c}(I^{(N)}_{\alpha})}=e^{S_i\phi^T_{\omega_c}|_{I^{(N)}_\alpha}}.
\end{equation}
Finally, this gives $H^u_\nu(R_\mu)=\mathcal{G}(T, \omega)$ and completes the proof of Proposition~\ref{prop: centrallimit}.

\medskip

\noindent
\textbf{Unstable case.}
Suppose that $\omega$ is of unstable type. Recall $\nu:=\widehat{\nu}_\omega=\widehat h_*(Leb)$ (see Section~\ref{subs: AIET and conformal}) is the $\widehat{\phi}_\omega$-conformal measure for $\widehat T$ and $\psi:=\widehat{\phi}_\omega$.
Then
\[\log\frac{d(T^{-1})_*\nu}{d\nu}=\widehat \phi_{\omega}\quad\text{ with }\quad\omega=\sum_{i=1}^mv_i,\]
where $v_i$, $1\leq i\leq m$, are right eigenvectors (with eigenvalues $e^{\lambda_1}>e^{\lambda_2}>\ldots>e^{\lambda_m}$) of the self-similarity matrix $M$ different than the maximal one. In particular, $\langle v_i,\lambda\rangle=0$ for all $1\leq i\leq m$. As $\omega$ is of unstable type, we have $\lambda_1>0$. Let $h_i=\widehat \phi_{v_i}$ for $1\leq i\leq m$. Then for any $x\in P_\alpha$, we have
\[(h_i)_A(R^{-1}x)=S_{q^{(N)}_\alpha}h_i(R^{-1}x)=(Mv_i)_\alpha=e^{\lambda_i}v_i=e^{\lambda_i}h_i(x).\]
Thus $(h_i)_A\circ R^{-1}=e^{\lambda_i}h_i$.
Finally, note that by \eqref{diampart} we have
\[
\textup{diam}(X^{(n)})=\max_{\alpha\in\A}\textup{diam}\Big(\widehat I^{((n-1)\cdot N)}_{\alpha}\Big)\to 0 \qquad\text{as}\quad n\to\infty.
\]
Therefore, the map $T$ and the measure $\nu$ meet all the assumptions of Theorem~\ref{thm: imperfect_divergence}.
Thus,
\begin{equation*}\label{eq: reductiontoCantor_unstable}
\lim_{n\to\infty}-\frac{1}{n}\log \widehat \nu_\omega\big(\widehat{\mathcal Q}_{nN}(x)\big) = \infty\qquad\text{for $\widehat h_*(\mu_f)$-a.e }x\in K.
\end{equation*}
In view of \eqref{eq: equalmeasuresonpartitions}, this gives
\begin{equation*}\label{eq: reductiontoCantor_unstable1}
\lim_{n\to\infty}-\frac{1}{n}\log Leb\big(\mathcal Q^f_{nN}(x)\big) = \infty\qquad\text{for $\mu_f$-a.e }x\in I,
\end{equation*}
which completes the proof of
 Proposition~\ref{prop: unstablelimit}.
\end{proof}

\section{Information content of conformal measures of AIETs - proof of Proposition~\ref{prop: centrallimit_conformal}}\label{sec: HDproof_conformal}
We turn our attention now to the main result concerning the Hausdorff dimension of conformal measures of IETs. The proof of Proposition~\ref{prop: centrallimit_conformal} is very similar to the proof of Proposition~\ref{prop: centrallimit} and uses the technical results obtained in previous sections. However, these proofs are not identical. Hence, to avoid confusion, we present fully the proof of Proposition~\ref{prop: centrallimit_conformal}.

\begin{proof}[Proof of Proposition~\ref{prop: centrallimit_conformal}]
Let $T=(\pi,\lambda)$ be a hyperbolically self-similar IET with period $N$ and self-similarity matrix $M$. Let also $\rho_T$ be the logarithm of its Perron-Frobenius eigenvalue, i.e.\ $\lambda M=e^{\rho_T}\lambda$.

Consider a vector $\omega\in \R^{\mathcal A}$ of central-stable type and let $\nu_{\omega}$ be the unique $\phi_{\omega}^T$-conformal measure, whose existence and uniqueness follows from Proposition~\ref{prop:conformal_measures}. Moreover, consider the decomposition $\omega=\omega_c+\omega_s$ into its central and stable parts. Let $\nu_{\omega_c}$ be the unique $\phi_{\omega_c}^T$-conformal measure for $T$. It follows also from Proposition~\ref{prop:conformal_measures}, that both $\nu_\omega$ and $\nu_{\omega_c}$ are continuous.

Let us consider the matrix $M(\omega_c)=M^{(N)}_{\pi,\lambda,\omega_c}$. Let $\rho_c>0$ be the logarithm of the Perron-Frobenius eigenvalue of $M(\omega_c)$. Let $\ell^c,\theta^c \in \R^\A_+$ be the unique left and right Perron-Frobenius eigenvector of $M(\omega_c)$ satisfying $| \ell^c |= 1$ and $\langle \ell^c, \theta^c \rangle = 1$. In view of Lemma~\ref{lem:from_cs_to_c}, we have $\ell^c_\alpha=\nu_{\omega_c}(I_\alpha)$ for $\alpha\in\mathcal A$.

As in the previous section, we present first the list of objects corresponding to the notions in the abstract setting given in Sections~\ref{sec: skewproducts},~\ref{sec: renormalization}, and~\ref{sec: dynpar}. We take the following:
\begin{itemize}
\item for $(X,T)$ we take $(K,\widehat T)$ - the Cantor model of the IET $T$, described in Section~\ref{sc:cantor_model};
\item for $\mu$ we substitute the measure $\widehat{\nu}_{\omega_c}=(i^+)_{*}\nu_{\omega_c}$. Then $\mu=\widehat{\nu}_{\omega_c}$ is the unique $\widehat \phi_{\omega_c}$-conformal measure for $\widehat T$, where the uniqueness also follows from Proposition~\ref{prop:conformal_measures};
\item we take $\phi=\widehat \phi_{\omega_c}$;
\item for the set $A$, we take $\widehat I^{(N)}:=\overline{i^+(I^{(N)})}$;
\item the mapping $r_{A}$ is given by $r_{A}(x)=q_{\alpha}^{(N)}$ for $x\in \widehat I_{\alpha}^{(N)}:=\overline{i^+(I^{(N)}_{\alpha})}$, $\alpha\in\A$ or in other words, $r_{A}$ describes the heigths of Rokhlin towers obtained by Rauzy-Veech induction in a single period;
\item the partition $\mathcal P$ is $(\widehat I_{\alpha})_{\alpha\in\A}$;
\item then the dynamical partitions $\mathcal P^{(n)}$ coincides with $\widehat{\mathcal Q}_{N\cdot n}$;
\item the rescaling $R$ is taken to be $\widehat R_{\rho_T}$, then, by \eqref{eq: Rnuvc}, we have
\[\rho_\mu=-\log\widehat{\nu}_{\omega_c}(\widehat I^{(N)})=-\log \nu_{\omega_c}(I^{(N)})=\rho_{c};\]
\item the rescaled domains are
\[
X^{(n)}=\widehat R^n_{\rho_T}(K)=\bigcup_{\alpha\in\A}\widehat I^{(n\cdot N)}_{\alpha};
\]
\item taking $\theta:=\widehat{\phi}_{\theta^c}$, we have $\theta_A(x)=S^{\phi}_{q_{\alpha}^{(N)}}\theta(x)=(M(\omega_c)\theta^c)_\alpha=e^{\rho_c}\theta^c_\alpha$ for $x\in R^{-1}\widehat{I}_\alpha$, which gives $\theta_A\circ R^{-1}=e^{\rho_\mu}\theta$;
\item as $\theta_A\circ R^{-1}=e^{\rho_\mu}\theta$, we can define the renormalization map $R_\mu=R_{\widehat \nu_{\omega_c}}$ on $X^\theta=K^{\widehat \phi_{\theta^{c}}}$; and
\item finally, take $\nu:=\widehat \nu_{\omega_c}$, that is, the same measure as for $\mu$, and take $\psi:=\widehat{\phi}_{\omega_c}$.
\end{itemize}
Proposition~\ref{prop: centrallimit_conformal} will be deduced from the results of Sections~\ref{sec: perfectlyscaled}. In order to verify that we can apply these results, we need to first check whether $R_{\widehat \nu_{\omega_c}}$ is ergodic.
In order to deduce ergodicity of $R_{\widehat \nu_{\omega_c}}$, we want to apply Lemma~\ref{lem:irrMar}. Thus, we need to verify its assumptions. First, by \eqref{RVmatrix meaning}, we have that the incidence matrix $M$ of $R$ is equal to the self-similarity matrix $M$ of $T$. Since $T$ is self-similar, $M$ is positive.

Now we need to verify that the assumption \eqref{eq:tozero} holds. Note that by Lemma~\ref{lem:uniformlimit} and the fact that $\nu_{\omega_c}$ is continuous, we have
the first convergence in \eqref{eq:tozero}. The second convergence follows from \eqref{diampart}. Thus, by Lemma~\ref{lem:irrMar}, the renormalization $R_{\widehat \nu_{\omega_c}}$ is ergodic.

As we have already shown in the proof of Proposition~\ref{prop: centrallimit}, the measure $\nu=\widehat \nu_{\omega_c}$ satisfies the perfect rescaling condition \eqref{eq:condlambda} with $\lambda=0$.
Applying Theorem~\ref{thm:perfect_zero}, we get
\[
\lim_{n\to\infty}-\frac{1}{n}\log \nu(P^{(n)}(x))=H^u_\mu(R_\mu)\quad\text{for $\mu$-a.e. }x\in K.
\]
As $\omega_c$ is of central type, by Proposition~\ref{prop:conformal_measures}, the measures $\nu_{\omega_c}$ and $\widehat \nu_{\omega_c}$ are continuous. Thus, the map $i_+:I\to K$ establishes a measure-theoretical isomorphism between $T$ on $(I,\nu_{\omega_c})$ and $\widehat T$ on $(K,\widehat \nu_{\omega_c})$. It follows that
\[
\lim_{n\to\infty}-\frac{1}{n}\log \nu_{\omega_c}(P_T^{(n)}(x))=H^u_\mu(R_\mu)\quad\text{for $\nu_{\omega_c}$-a.e. }x\in I.
\]
In view of \eqref{eq:partitiopassintocentral} (see Lemma~\ref{lem:from_cs_to_c}) and the fact that the measures $\nu_{\omega_c}$ and $\nu_{\omega}$ are equivalent (see Corollary~\ref{cor: bded_density}), we have
\[
\lim_{n\to\infty}-\frac{1}{n}\log \nu_{\omega}(P_T^{(n)}(x))=H^u_\mu(R_\mu)\quad\text{for $\nu_{\omega}$-a.e. }x\in I.
\]
Moreover, by Theorem~\ref{thm:difference_self-similar} applied to $\varrho=\widehat{Leb}$, we have $H^u_\mu(R_\mu)<\rho_{\widehat{Leb}}=\rho_T$, whenever $\omega_c\neq 0$.

We finish by showing that $H^u_\mu(R_\mu)=\mathcal H(T,\omega)$, where $\mathcal{H}(T, \omega)$ is defined by \eqref{eq: HD_conformal_formula}. In view of Theorem~\ref{thm:perfect_zero},
\[
H^u_\mu(R_\mu)=-\sum_{\alpha\in \mathcal{A}}\sum_{0\leq i<q_\alpha}\log\frac{\widehat \nu_{\omega_c}(T^iR^{-1}P_\alpha)}{\widehat \nu_{\omega_c}(P_{\alpha})}\widehat \nu_{\omega_c}(T^iR^{-1}P_\alpha)\theta^c_{\beta(\alpha,i)}.
\]
Using again the fact that the map $i_+:I\to K$ establishes a measure-theoretical isomorphism between $T$ on $(I,\nu_{\omega_c})$ and $\widehat T$ on $(K,\widehat \nu_{\omega_c})$, we get
\[
H^u_\mu(R_\mu)=-\sum_{\alpha\in \mathcal{A}}\sum_{0\leq i<q_\alpha}\log\frac{\nu_{\omega_c}(T^iR^{-1}I_\alpha)}{\nu_{\omega_c}(I_{\alpha})}\nu_{\omega_c}(T^iR^{-1}I_\alpha)\theta^c_{\beta(\alpha,i)}.
\]
In view of \eqref{eq: nuomc1} and \eqref{eq: nuomc2}, we have
\[
\nu_{\omega_c}(T^iR^{-1}I_\alpha)=e^{-\rho_c}e^{S_i\phi^T_{\omega_c}|_{I^{(N)}_\alpha}}\nu_{\omega_c}(I_\alpha)=e^{-\rho_c}e^{S_i\phi^T_{\omega_c}|_{I^{(N)}_\alpha}}\ell^c_\alpha.
\]
As
\[
\sum_{\alpha\in \mathcal{A}}\sum_{0\leq i<q_\alpha}\nu_{\omega_c}(T^iR^{-1}I_\alpha)\theta^c_{\beta(\alpha,i)}=\sum_{\beta\in \mathcal{A}}\nu_{\omega_c}(I_\beta)\theta^c_{\beta}=\langle\ell^c,\theta^c\rangle=1,
\]
this gives
\[
H^u_\mu(R_\mu)=\rho_{c}-e^{-\rho_c}\sum_{\alpha\in \mathcal{A}}\sum_{0\leq i<q_\alpha}\big({S_i\phi^T_{\omega_c}|_{I^{(N)}_\alpha}}\big)e^{{S_i\phi^T_{\omega_c}|_{I^{(N)}_\alpha}}}\ell^c_\alpha\theta^c_{\beta(\alpha,i)}=\mathcal{H}(T,\omega),
\]
which completes the proof of Proposition~\ref{prop: centrallimit_conformal}.
\end{proof}

\section{H\"older exponents - proof of Theorem~\ref{thm: main2}}\label{sec: reg conj}
In this section, we show that by using the machinery developed in previous sections, we can find maximal H\"older regularity the semi-conjugacy $h$ between $f$ and $T$, i.e.\ the continuous map $h:I\to I$ such that $h\circ f=T\circ h$. More precisely, we determine the supremal regularity of the semi-conjugacy, i.e., the supremum of H\"older exponents $\gamma$ for which $h$ is $\gamma$-H\"older. As before, assume that $T$ is hyperbolically self-similar, with period $N$ and the self-similarity matrix $M$.

\subsection{Stable log-slope vector.}
Assume first that the log-slope vector $\omega$ of $f$ is of stable type. Recall also that $\alpha(\omega)$ denotes the modulus  of the logarithm of the maximal eigenvalue, whose corresponding eigenvector appears in the decomposition of $\omega$ w.r.t. the base of eigenvectors. Then, by Proposition~\ref{prop:stableconj}, $h$ is a homeomorphism such that
\begin{itemize}
\item $\mathfrak{H}(h)=1+\frac{\alpha(\omega)}{\rho_T}$ if $\alpha(\omega)<\rho_T$, and
\item $h:I\to I$ is a $C^\infty$-diffeomorphism if $\alpha(\omega)=\rho_T$.
\end{itemize}
This gives parts \eqref{numb:Thm 1.3 (1)} and \eqref{numb:Thm 1.3 (2)} in Theorem~\ref{thm: main2}.

\subsection{Regularity of the conjugacy for AIETs with central-stable log-slope vector.} \label{sec: regulh}
Suppose that the log-slope vector $\omega$ of $f$ is of central-stable type. Recall that in this case, the semi-conjugacy $h$ is actually a conjugacy.
Assume now that $\omega=\omega_c+\omega_s$, where $\omega_c$ is an invariant vector for $M$ and $\omega_s$ is of stable type. By Proposition~\ref{prop:conjcs}, any AIET $f_c\in \operatorname{Aff}(T,\omega_c)$ is $C^1$-conjugated to $f$. Thus, we can reduce the proof to the simpler case where $\omega=\omega_c$, i.e.\ $\omega$ is an invariant vector of $M$.

From now on, we will usually assume that $\omega$ is an invariant vector. Consider the family of dynamical partitions $(\mathcal P^{(n)})_{n\geq 0}=(\mathcal Q_{n\cdot N}^{f})_{n\ge 0}$ of $f$ given by the Rauzy-Veech induction. Then $\mathcal P^{(0)}=(I^{(0)}_\alpha)_{\alpha\in\mathcal A}=(I^{(0)}_\alpha(f))_{\alpha\in\mathcal A}$ is the partition $[0,1)$ into intervals affinely transformed by $f$. For any $n\geq 1$ let
\begin{align*}
\zeta_n:&=\frac{1}{n}\max_{(x,r)\in[0,1)^\theta}\left(-\sum_{0\leq j<n}\log\frac{Leb(P^{(1)}(\pi R^j_{\mu_f}(x,r)))}{Leb(P^{(0)}(\pi R^j_{\mu_f}(x,r)))}\right)\\
&=\frac{1}{n}\max_{\substack{((\alpha_j,i_j))_{j=0}^{n}\\ \beta(\alpha_j,i_j)=\alpha_{j-1}}}\left(-\sum_{1\leq j\leq n}\log\frac{Leb(f^{i_j}I_{\alpha_j}^{(1)})}{Leb(I_{\beta(\alpha_j,i_j)})}\right)>0.
\end{align*}
As $(m+n)\zeta_{m+n}\leq m\zeta_{m}+n\zeta_{n}$, the sequence $(\zeta_n)_{n\ge 0}$ converges and let
\begin{equation*}
\zeta^f=\zeta:=\lim_{n\to\infty}\zeta_n=\inf_{n\geq 1}\zeta_n,
\end{equation*}
recalling that $\beta(\alpha,i)$ is such that $f^i(I^{(1)}_{\alpha})\subseteq I^{(0)}_{\beta(\alpha,i)}$.
For any AIET $f\in \operatorname{Aff}(T,\omega)$ of hyperbolic periodic type, semi-conjugated to a self-similar IET $T=(\pi,\lambda)$, with the log-slope vector $\omega$ of central-stable type we define
\begin{equation}\label{def:zeta}
\zeta^f:=\zeta^{f_c},
\end{equation}
where $ f_c\in \operatorname{Aff}(T,\omega_c)$ is an AIET related to the invariant vector $\omega_c$ in the decomposition of $\omega=\omega_c+\omega_s$ into an invariant and a stable type vectors.

If $\omega$ is an invariant vector (again), then standard weak-limit arguments show that
\begin{equation}\label{eq:def:zeta}
\zeta^f=\max\Big\{\int_{\widehat{\Sigma}}- \log\frac{Leb(f^{i_0}I_{\alpha_0}^{(1)})}{Leb(I_{\beta(\alpha_0,i_0)})}\,d\lambda\big((\alpha_j,i_j)\big)_{j\in\Z}\mid \lambda\in \Lambda(\widehat{\Sigma}, \sigma)\Big\},
\end{equation}
where $\Lambda(\widehat{\Sigma}, \sigma)$ is the simplex of $\sigma$-invariant probability measures on $\widehat{\Sigma}$ {(recall that it is a compact set)}.
{Note that if $\nu\in \Lambda(\widehat{\Sigma}, \sigma)$ is the measure which maximizes the integral in the definition of $\zeta$, then we may assume that $\nu$ is ergodic. This follows from the fact that all $\sigma$-invariant measures are convex combinations of ergodic measures.}

Since $\mathfrak{m}:=\mathfrak{Q}_*(\mu_f^{\theta})\in \Lambda(\widehat{\Sigma}, \sigma)$, where $\mathfrak Q(x,r)=(\mathfrak q(R^n_{\mu_f}(x,r))_{n\in\Z}$ associates to a point a sequence of indices of partition obtained from renormalization $R_{\mu_f}$, we have
\begin{equation}\label{eq: zetagreaterthanH}
\begin{split}
\zeta&\geq - \log\frac{Leb(f^{i_0}I_{\alpha_0}^{(1)})}{Leb(I_{\beta(\alpha_0,i_0)})}\,d\mathfrak{m}\big((\alpha_j,i_j)\big)_{j\in\Z}\\
&=\int_{I^\theta}- \log\frac{Leb(P^{(1)}(\pi(x,r)))}{Leb(P^{(0)}(\pi (x,r)))}\,d\mu_f^\theta(x,r)\\
& \geq-\sum_{\alpha\in \mathcal{A}}\sum_{0\leq i<q_\alpha}\log\frac{Leb(f^iI^{(1)}_\alpha)}{Leb(I^{(0)}_{\beta(\alpha,i)})}\mu_f(f^iI^{(1)}_\alpha)\theta_{\beta(\alpha,i)}=\mathcal G(T,\omega).
\end{split}
\end{equation}

Recall that $\rho_{\mu_f}>0$ denotes the logarithm of the Perron-Frobenius eigenvalue of $M$ - the self-similarity matrix of $f$, and, in view of Theorem~\ref{thm:difference_conformal}, we have $\mathcal G(T,\omega)>\rho_{\mu_f}$ whenever $\omega\neq 0$ (or equivalently $f$ is not an IET). At the end of this section, we will also prove that $\zeta>\mathcal G(T,\omega)$. As $\frac{\rho_{\mu_f}}{\mathcal G(T,\omega)}$ is the Hausdorff dimension of the unique $f$-invariant measure $\mu_f$, we get $\dim_{H}(\mu_f)=\frac{\rho_{\mu_f}}{\mathcal G(T,\omega)}>\frac{\rho_{\mu_f}}{\zeta}>0$. First, we prove that $\frac{\rho_{\mu_f}}{\zeta}$ the supremum of H\"older exponents of the conjugacy $h$ between $f$ and $T$. The proof splits into the two following propositions.

\begin{proposition}\label{prop: centralHolderestimate}
The conjugacy $h$ between $f$ and $T$ is $\gamma$-H\"older for every
\[
0<\gamma< \frac{\rho_{\mu_f}}{\zeta}.
\]
\end{proposition}
\begin{proof}
 Since $\lim_{x\to 1^-}h(x)=1$, the map $h^{-1}$ is uniformly continuous. Hence, since $\max_{J\in \mathcal P^{(n)}} Leb(h(J))$ decays exponentially with $n$, we have
\begin{equation}\label{eq:lentozero}
\lim_{n\to\infty}\max\{Leb(J)\mid J\in \mathcal P^{(n)}\}=0.
\end{equation}

Let $N$ be a natural number such that
\[\frac{-\log 2+(n-1)\rho_{\mu_f}}{n\cdot \zeta_n -\log \left(\min_{J\in \mathcal{P}^{(0)}}Leb(J)\right)}>\gamma\quad \text{for all}\quad n>N.\]
Let $\delta=\delta_N:=\min\{Leb(J)\mid J\in \mathcal P^{(N)}\}>0$. In view of \eqref{eq:lentozero}, there exists $n> N$ such that
\[
n:=\min\{m\in \mathbb N\mid \#\{[x,y]\cap \partial \mathcal P^{(m)}\}\ge 2 \},
\]
for any $x<y$ such that $|x-y|<\delta$.
Then there exist intervals $I^{-}_{n-1},I^{+}_{n-1}\in \mathcal P^{(n-1)}$ and an interval $I_{n}\in \mathcal P^{(n)}$ such that
\begin{equation}\label{eq: intervalinclusion}
 I_{n}\subseteq [x,y]\subseteq I^{-}_{n-1}\cup I^{+}_{n-1}.
 \end{equation}
We want to prove that for every $0<\gamma< \frac{\rho_{\mu_f}}{\zeta}$, we have
\begin{equation}\label{eq: Holder goal}
|h(x)-h(y)|\le |x-y|^{\gamma}\quad \text{for all $x,y$ such that}\quad |x-y|<\delta,
\end{equation}
which by \eqref{eq:projected_invariant_measure} is equivalent to
\begin{equation}\label{eq: Holder reduction}
\frac{-\log\mu_f([x,y])}{-\log Leb([x,y])}\ge \gamma,
\end{equation}
where $\mu_f$ is the unique probability $f$-invariant measure.
To prove the above inequality, we will prove that for every sufficiently small $\epsilon>0$ there exists $N$ large enough such that for every $x<y$ satisfying $|x-y|<\delta_N$, the following holds
\begin{equation}\label{eq: second Holder reduction}
 \frac{-\log\mu_f([x,y])}{-\log Leb([x,y])}\ge \frac{\rho_{\mu_f}}{\zeta}-\epsilon.
 \end{equation}
By, \eqref{eq: intervalinclusion}, we have that
\begin{align}\label{eq: Holder numerator estimate}
\begin{aligned}
-\log\mu_f([x,y])&\ge -\log \mu_f(I^{-}_{n-1}\cup I^{+}_{n-1})\ge -\log(2\cdot e^{-(n-1)\rho_{\mu_f}}\cdot \max_{J\in\mathcal P^{(0)}}\mu(J))\\
&\ge-\log 2+(n-1)\rho_{\mu_f}.
\end{aligned}
\end{align}
For any $x_0\in I_n$, we have $I_{n}=\mathcal{P}^{(n)}(x_0)$. Then, in view of \eqref{eq:infla=0}, for any choice $(x_0,r)\in I^\theta$, we have
\begin{align}\label{eq: Holder denominator estimate}
 \begin{split}
 -\log Leb([x,y])&\le -\log Leb(I_n)= -\log \nu({P}^{(n)}(x_0))\\
& =-\sum_{0\leq i<n}\log\frac{\nu({P}^{(1)}(\pi(R^i_\mu(x_0,r))))}{\nu({P}(\pi(R^{i}_\mu(x_0,r))))}-\log \nu({P}(x_0))\\
&\leq n\cdot \zeta_n-\log \left(\min_{J\in \mathcal{P}^{(0)}}Leb(J)\right).
 \end{split}
 \end{align}
By combining \eqref{eq: Holder numerator estimate} and \eqref{eq: Holder denominator estimate}, we get
\[
\frac{-\log\mu([x,y])}{-\log Leb([x,y])}\ge \frac{-\log 2+(n-1)\rho_{\mu_f}}{n\cdot \zeta_n -\log \left(\min_{J\in \mathcal{P}^{(0)}}Leb(J)\right)},
\]
with $n>N$. By the choice of $N$, we have
\[\frac{-\log 2+(n-1)\rho_{\mu_f}}{n\cdot \zeta_n -\log \left(\min_{J\in \mathcal{P}^{(0)}}Leb(J)\right)}>\frac{\rho_{\mu_f}}{\zeta}-\epsilon=\gamma\quad \text{for all}\quad n>N,\]
by taking $\epsilon:= \frac{\rho_{\mu_f}}{\zeta}-\gamma>0$.
Hence, we get \eqref{eq: Holder reduction}, whenever $|x-y|<\delta_N$ and \eqref{eq: Holder goal} with $\delta=\delta_N$. It follows that taking $C:=(1+\delta^{-1})^{1-\gamma}$, we have
\begin{equation}\label{eq:finHol}
|h(x)-h(y)|\leq C|x-y|^\gamma\quad \text{for all}\quad x,y\in[0,1).
\end{equation}
Indeed, if $x<y$ and $|x-y|\geq \delta_N$, then we choose $x=x_0<x_1<\ldots<x_{k-1}<x_k=y$ such that $|x_j-x_{j-1}|=|x-y|/k$ for all $1\leq j\leq k$ and $|x-y|/k<\delta_N\leq |x-y|/(k-1)$. Then
\[|h(x)-h(y)|\leq \sum_{j=1}^k|h(x_j)-h(x_{j-1})|\leq \sum_{j=1}^k|x_j-x_{j-1}|^{\gamma}=k\Big(\frac{|x-y|}{k}\Big)^\gamma=k^{1-\gamma}|x-y|^\gamma.\]
As $\delta_N\leq |x-y|/(k-1)\leq 1/(k-1)$, we have $k\leq 1+\delta_N^{-1}$, which gives \eqref{eq:finHol}.
\end{proof}

To show that the estimate in Proposition~\ref{prop: centralHolderestimate} is optimal, we prove the following result.
\begin{proposition}\label{prop:optH}
There exists an infinite sequence of intervals $[x_n,y_n]\subset I$, such that
\begin{equation}\label{eq:limfrlog}
\lim_{n\to\infty}|x_n-y_n|=0\quad\text{and}\quad\lim_{n\to\infty}\frac{-\log\mu_f([x_n,y_n])}{-\log Leb([x_n,y_n])}=\frac{\rho_{\mu_f}}{\zeta}.
\end{equation}
\end{proposition}
\begin{proof}
For any $n\geq 1$, choose $(\xi_n,r_n)\in I^\theta$ such that
\[\zeta_n=-\frac{1}{n}\sum_{0\leq j<n}\log\frac{Leb(P^{(1)}(\pi R^j_{\mu_f}(\xi_n,r_n)))}{Leb(P^{(0)}(\pi R^j_{\mu_f}(\xi_n,r_n)))}.\]
Then, in view of \eqref{eq:infla=0}, we have
\[-\log Leb(P^{(n)}(\xi_n))=n\zeta_n-\log Leb(P(\xi_n)).\]
Let $[x_n,y_n)=P^{(n)}(\xi_n)$. It follows that
\[n\zeta_n\leq -\log Leb([x_n,y_n))\leq n\zeta_n-\log\min_{J\in\mathcal{P}^{(0)}}Leb(J).\]
As $\mu_f$ is $f$-invariant and $(R^{-1})_*\mu_f=e^{-\rho_\mu}\mu_f|_{A}$, we have
\begin{gather*}
n\rho_{\mu_f}\leq n\rho_{\mu_f}-\log\max_{J\in\mathcal{P}^{(0)}}\mu_f(J)\leq -\log \mu_f([x_n,y_n))\\
=-\log\mu_f(P^{(n)}(\xi_n))\leq n\rho_{\mu_f}-\log\min_{J\in\mathcal{P}^{(0)}}\mu_f(J).
\end{gather*}
As $\mu_f$ is continuous, it follows that
\[
\frac{n\rho_{\mu_f}}{n\zeta_n-\log\min_{J\in\mathcal{P}^{(0)}}Leb(J)}\le \frac{-\log\mu_f([x_n,y_n])}{-\log Leb([x_n,y_n])}\le \frac{n\rho_{\mu_f}-\log\min_{J\in\mathcal{P}^{(0)}}\mu_f(J)}{n\zeta_n}.
\]
Since $\zeta_n\to\zeta$, this gives the second part of \eqref{eq:limfrlog}. As $[x_n,y_n]\in\mathcal{P}^{(n)}$, by \eqref{eq:lentozero}, we also have $|x_n-y_n|\to 0$, which finishes the proof of the proposition.
\end{proof}
For any function $f:[0,1]\to\R$ denote by $\mathfrak{H}(h)$ its supreme H\"older exponent given by
\[\mathfrak{H}(h):=\sup\{\alpha>0\mid \exists_{C>0}\forall_{x,y\in[0,1]}|h(y)-h(x)|\leq C|y-x|^\alpha\}.\]

\begin{proposition}\label{cor:she}
Let $T=(\pi,\lambda)$ be a self-similar IET of $d\ge 2$ intervals and let $M$ be its positive self-similarity matrix. Let $f\in \operatorname{Aff}(T,\omega)$ be an AIET, semi-conjugated to $T$, with vector of logarithms of slopes $\omega$. If $\omega\neq 0$ is a central-stable vector of the matrix $M$ and $h:[0,1]\to[0,1]$ is the conjugacy of $T$ and $f$, then
\[\mathfrak{H}(h)=\frac{\rho_T}{\zeta^f}\quad\text{and}\quad 0<\frac{\rho_T}{\zeta^f}<\dim_H(\mu_f),\]
where $\rho_T$ is the logarithm of the Perron-Frobenius eigenvalue of $M$, $\zeta^f$ is given by \eqref{def:zeta} and $\mu_f$ is the unique invariant measure of $f$.
\end{proposition}

\begin{proof}
As $\omega$ is a central-stable vector, we have $\omega=\omega_s+\omega_c$, where $\omega_s$ is a stable vector and $\omega_c\neq 0$ is a central eigenvector of $M$. As we have already note, $f$ is $C^1$-conjugated to $f_c\in \operatorname{Aff}(T,\omega_c)$. If $h_c:[0,1]\to[0,1]$ is the conjugacy of $T$ and $f_c$ and $\mu_{f_c}$ is the unique $f_c$-invariant measure, then ${h_c}^{-1}\circ h$ is a $C^1$-diffeomorphism of $[0,1]$ conjugating $f$ and $f_c$ such that $({h_c}^{-1}\circ h)_*\mu_{f_c}=\mu_f$. It follows that
\[\mathfrak{H}(h)=\mathfrak{H}(h_c)\quad\text{and}\quad \dim_H(\mu_f)=\dim_H(\mu_{f_c}).\]
In view of Proposition~\ref{prop: centralHolderestimate}, we have $\mathfrak{H}(h_c)\geq \frac{\rho_T}{\zeta^{f_c}}$. Suppose, contrary to our claim, that $\mathfrak{H}(h_c)> \frac{\rho_T}{\zeta^{f_c}}$. Then there exists $\gamma > \frac{\rho_T}{\zeta^{f_c}}$ such that $h_c$ is $\gamma$-H\"older, so there exists $C>0$ such that $|h_c(y)-h_c(x)|\leq C|y-x|^\gamma$ for all $x,y\in[0,1]$. As $h_c(x)=\mu_{f_c}[0,x]$ and the measure $\mu_{f_c}$ is continuous, this gives
\[\frac{-\log \mu_{f_c}[x,y]}{-\log Leb[x,y]}\geq \gamma +\frac{-\log C}{-\log|x-y|}\quad \text{if}\quad x<y.\]
In view of Proposition~\ref{prop:optH}, there exists an infinite sequence of intervals $[x_n,y_n]\subset [0,1]$, such that
\[
\lim_{n\to\infty}|x_n-y_n|=0\quad\text{and}\quad\lim_{n\to\infty}\frac{-\log\mu_{f_c}([x_n,y_n])}{-\log Leb([x_n,y_n])}=\frac{\rho_{T}}{\zeta^{f_c}}.
\]
It follows that
\[\frac{\rho_{T}}{\zeta^{f_c}}=\lim_{n\to\infty}\frac{-\log\mu_{f_c}([x_n,y_n])}{-\log Leb([x_n,y_n])}\geq \gamma,\]
which is in contradiction with the choice of $\gamma > \frac{\rho_T}{\zeta^{f_c}}$. By definition, $\zeta^f=\zeta^{f_c}$. This gives
\[\mathfrak{H}(h)=\mathfrak{H}(h_c)=\frac{\rho_T}{\zeta^{f_c}}=\frac{\rho_T}{\zeta^f}.\]
The inequality $\frac{\rho_T}{\zeta^f}<\dim_H(\mu_{f_c})=\dim_H(\mu_f)$ follows directly from $\zeta^{f_c}>H^u_{Leb}(R_{\mu_{f_c}})$ which will be proven in the rest of the section, see Proposition~\ref{prop:zeta>}.
\end{proof}

\subsection{Maximizing measures}
In this section, we will give some arguments about maximizing measures that help derive an effective formula for computing $\zeta$ and in proving the inequality $\zeta>\mathcal G(T,\omega)$ that was already used in the proof of Proposition~\ref{cor:she}.
The results presented in this section will also be used to establish regularity of the inverse $h^{-1}$ in Section~\ref{sec:regofinverse}.

From now on, we will assume again that $f\in \operatorname{Aff}(T,\omega)$, where $\omega\neq 0$ is an invariant vector of $M$. Some notation and arguments that we will use are borrowed from \cite{Ch-Ga-Ug}.

Let us consider a directed graph $\mathfrak{G}_{T}=(\Sigma,\mathcal{E})$ with the set of vertices
\[\Sigma=\{(\alpha,j)\mid \alpha\in\mathcal A,\ 0\leq j<q_\alpha\},\]
and the set of arrows
\[\mathcal E=\{(\alpha_0,i_0)\to (\alpha_1,i_1)\mid \beta(\alpha_1,i_1)=\alpha_0\}.\]
This is the graph associated with the shift of finite type $\sigma:\widehat{\Sigma}\to \widehat \Sigma$ defined in Section~\ref{sec:Markov}.

Denote by $\Pi_T$ the set of finite paths in $\mathfrak{G}_{T}$.
For any $\varrho\in \Pi_T$ let $|\varrho|$ be the length of the path.
An elementary loop is a cyclic path in $\mathfrak{G}_{T}$ with no repeated vertices.
Denote by $\mathcal{C}^{el}_T$ the set of elementary loops. Notice that this set is finite.
Any loop $C$ (a cyclic path)
in $\mathfrak{G}_{T}$ can be represented as a union of elementary loops, $C=C_1\ldots C_n$, where
each arrow in $C$ appears in exactly one elementary loops $C_i$, $1\leq i\leq n$, thus $|C|=\sum_{i=1}^n|C_i|$. Denote by $\mathcal{C}_T$ the set of cyclic paths.

Let $\vartheta_+,\vartheta_-:\mathcal E\to\R_{\geq 0}$ be given by
\[\vartheta_{\pm}((\alpha_0,i_0)\to (\alpha_1,i_1)):=\pm\log \frac{Leb(f^{i_1} I^{(1)}_{\alpha_1})}{Leb(I^{(0)}_{\alpha_0})},\]
which can be further extended to $\vartheta_{\pm}:\Pi_T\to\R_{\geq 0}$ by
\[\vartheta_\pm(\varrho):=\sum_{i=1}^n\vartheta_\pm(\varrho_i)\quad\text{if}\quad \varrho=\varrho_1\ldots\varrho_n\in \Pi_T,\ \text{where }\varrho_1,\ldots,\varrho_n\in \mathcal{E}. \]
While $\zeta_-$ will be useful in this section, $\zeta_+$ will be used in the next section to compute H\"older exponent of the inverse of $h$.

For simplicity of notation, we will usually write $\vartheta$ instead of $\vartheta_\pm$ for the rest of this section.
Let
\[\zeta_n(\vartheta):=\frac{1}{n}\sup\{\vartheta(\varrho)\mid \varrho\in\Pi_T,|\varrho|=n\}\quad\text{and}\quad\zeta(\vartheta):=\lim_{n\to\infty}\zeta_n(\vartheta)=\inf_{n\geq 1}\zeta_n(\vartheta).\]
Note that $\zeta^f=\zeta(\vartheta_-)$, where $\zeta^f$ is the quantity defined by \eqref{eq:def:zeta}. Let
\[\zeta_{\mathcal{C}}(\vartheta):=\sup_{C\in \mathcal{C}_T}\frac{\vartheta(C)}{|C|}\quad\text{and}\quad\zeta^{el}_{\mathcal{C}}(\vartheta):=\max_{C\in \mathcal{C}^{el}_T}\frac{\vartheta(C)}{|C|}.\]

\begin{lemma}\label{lem: allzetassame}
All three numbers $\zeta(\vartheta)$, $\zeta_{\mathcal{C}}(\vartheta)$ and $\zeta^{el}_{\mathcal{C}}(\vartheta)$ are equal.
\end{lemma}

\begin{proof}
By definition, $\zeta^{el}_{\mathcal{C}}(\vartheta)\leq \zeta_{\mathcal{C}}(\vartheta)$. Moreover, for every $C\in \mathcal{C}_T$ we can decompose $C=C_1\ldots C_n$ such that $C_i\in \mathcal{C}^{el}_T$, $1\leq i\leq n$.
Then
\[\frac{\vartheta(C)}{|C|}=\sum_{i=1}^n\frac{|C_i|}{|C|}\frac{\vartheta(C_i)}{|C_i|}\leq \zeta^{el}_{\mathcal{C}}(\vartheta).\]
Hence, $\zeta^{el}_{\mathcal{C}}(\vartheta)= \zeta_{\mathcal{C}}(\vartheta)$.

As $\zeta(\vartheta)$ is the maximum integrals of $\vartheta$ along all $\sigma$-invariant measures on $\widehat \Sigma$ and $\frac{\vartheta(C)}{|C|}$ is the integral of $\vartheta$ for the invariant measure on the periodic orbit given by a cyclic path $C\in \mathcal{C}_T$, we get $\zeta^{el}_{\mathcal{C}}(\vartheta)= \zeta_{\mathcal{C}}(\vartheta)\leq \zeta(\vartheta)$.

For any $n\geq 1$, let $\varrho^n=\varrho_1\ldots\varrho_n$ be a path in $\mathfrak{G}_{T}$ maximizing $\vartheta(\varrho^n)$, i.e.\ such that
$\frac{1}{n}\vartheta(\varrho^n)=\zeta_n(\vartheta)$. Similarly as in the proof of Lemma~\ref{lem:irrMar}, since $M$ is a positive matrix, we can extend the path $\varrho^n=\varrho_1\ldots\varrho_n$ to a cyclic path $\widetilde{\varrho}^n=\varrho_1\ldots\varrho_n\varrho_{n+1}\varrho_{n+2}$. Then
\begin{align*}
\zeta_n(\vartheta)&=\frac{1}{n}\vartheta(\varrho^n)\leq \frac{1}{n}\vartheta(\widetilde{\varrho}^n)+\frac{2}{n}\|\vartheta\|_{\sup}\\&
= \frac{n+2}{n}\frac{\vartheta(\widetilde{\varrho}^n)}{|\widetilde{\varrho}^n|}+\frac{2}{n}\|\vartheta\|_{\sup}\leq \frac{n+2}{n}\zeta_{\mathcal{C}}(\vartheta)+\frac{2}{n}\|\vartheta\|_{\sup}.
\end{align*}
It follows that
\[\zeta(\vartheta)=\lim_{n\to\infty}\zeta_n(\vartheta)\leq \zeta_{\mathcal{C}}(\vartheta),\]
which completes the proof.
\end{proof}

\begin{remark}
As the set of elementary loops $\mathcal{C}^{el}_T$ is finite, the previous lemma gives an effective formula for counting the quantity $\zeta^f$:
\begin{equation}
\zeta^f=\zeta(\vartheta_-)=\max_{C\in \mathcal{C}^{el}_T}\frac{\vartheta_-(C)}{|C|}.
\end{equation}
\end{remark}

For any finite path $\varrho=\varrho_1\ldots\varrho_N\in\Pi_T$ denote by $[\varrho]\subset \widehat\Sigma$ the cylinder containing all sequences $(a_j)_{j\in\Z}$ such that $(a_{j-1}\to a_{j})=\varrho_j$ for all $1\leq j\leq N$.

\begin{lemma}\label{lem:maxcyl}
Let $\lambda$ be a $\sigma$-invariant ergodic measure on $\widehat\Sigma$ which maximizes the integral of $\vartheta$, i.e.\ $\int_{\widehat\Sigma}\vartheta\,d\lambda=\zeta(\vartheta)$.
If $C$ is a cyclic path such that $\lambda([C])>0$, then the path $C\in \mathcal{C}_T$ is $\vartheta$-maximizing i.e.\ $\frac{\vartheta(C)}{|C|}=\zeta_{\mathcal{C}}(\vartheta)$.
\end{lemma}

\begin{proof}
Let $C=\varrho_1\ldots\varrho_N$ with $\lambda([C])>0$. By the ergodicity of $\lambda$, for $\lambda$-a.e.\ $a=(a_j)_{j\in\Z}$, we have
\begin{gather}\label{eq:tozeta}
\lim_{n\to\infty}\frac{1}{n}\sum_{j=1}^n\vartheta(a_{j-1}\to a_j)=\int\vartheta\,d\lambda=\zeta(\vartheta),\\
\lim_{n\to\infty}\frac{1}{n}\#\{0\leq j<n\mid \sigma^ja\in[C]\}=\lambda([C])>0.\nonumber
\end{gather}
It follows that for a.e.\ $a\in[C]$ there exists an increasing sequence $(n_k)_{k\geq 0}$ such that $n_0=0$, $n_{k+1}-n_k\geq N$, $\sigma^{n_k}a\in[C]$ for all $k\geq 0$ and $\liminf_{k\to\infty}\frac{k}{n_k}\geq \lambda([C])/N>0$.
For every $k\geq 1$ let $C^k$ be the cyclic path of length $n_k+N$ determined by the sequence $(a_j)_{j=0}^{n_k+N}$. Then
\[C^k=CC_1CC_2\ldots CC_kC,\]
where $C_1,C_2,\ldots,C_k$ are cyclic paths (maybe empty). In view of \eqref{eq:tozeta}, $\frac{\vartheta(C^k)}{|C^k|}\to\zeta(\vartheta)$ as $k\to\infty$.
Moreover,
\[\frac{\vartheta(C^k)}{|C^k|}=\frac{(k+1)N}{n_k+N}\frac{\vartheta(C)}{|C|}+\sum_{j=1}^k\frac{|C_j|}{|C^k|}\frac{\vartheta(C_j)}{|C_j|}\leq \frac{(k+1)N}{n_k+N}\frac{\vartheta(C)}{|C|}+ \Big(1-\frac{(k+1)N}{n_k+N}\Big)\zeta(\vartheta),\]
where the last inequality follows from Lemma~\ref{lem: allzetassame}.
It follows that
\[\frac{\vartheta(C)}{|C|}\geq \zeta(\vartheta)- \frac{n_k+N}{(k+1)N}\Big(\zeta(\vartheta)-\frac{\vartheta(C^k)}{|C^k|}\Big).\]
As $\limsup\frac{n_k}{k}$ is finite and $\frac{\vartheta(C^k)}{|C^k|}\to \zeta(\vartheta)$, we get $\frac{\vartheta(C)}{|C|}\geq \zeta(\vartheta)$. On the other hand, $\frac{\vartheta(C)}{|C|}\leq \zeta_{\mathcal{C}}(\vartheta)=\zeta(\vartheta)$, which completes the proof.
\end{proof}

\begin{proposition}\label{prop:zeta0}
Let $T=(\pi,\lambda)$ be a self-similar IET of $d\ge 2$ intervals and let $M$ be its positive self-similarity matrix. Suppose that $\omega\neq 0$ is an invariant vector of the matrix $M$.
Let $\lambda$ be a $\sigma$-invariant ergodic measure on $\widehat\Sigma$ such that $\lambda([\varrho])>0$ for any $\varrho\in \Pi_T$. Then $\int_{\widehat\Sigma}\vartheta\,d\lambda<\zeta(\vartheta)$.
\end{proposition}

\begin{proof}
Suppose that, contrary to our claim, $\int_{\widehat\Sigma}\vartheta\,d\lambda=\zeta(\vartheta)$.

We will show that for any $n\geq 1$ and $\alpha^{(0)},\alpha^{(n)}\in\mathcal A$ if $f^{i_1}I^{(n)}_{\alpha^{(n)}}\subset I^{(0)}_{\alpha^{(0)}}$ and $f^{i_2}I^{(n)}_{\alpha^{(n)}}\subset I^{(0)}_{\alpha^{(0)}}$ for some $0\leq i_1<i_2<q^{(n)}_{\alpha^{(n)}}$, then $Leb(f^{i_1}I^{(n)}_{\alpha^{(n)}})=Leb(f^{i_2}I^{(n)}_{\alpha^{(n)}})$.

Take any $x_1\in f^{i_1}I^{(n)}_{\alpha^{(n)}}$ and $x_2\in f^{i_2}I^{(n)}_{\alpha^{(n)}}$. Let $(\alpha_1^{(j)},i_1^{(j)})_{j=1}^n$ and $(\alpha_2^{(j)},i_2^{(j)})_{j=1}^n$ be sequences in $\Sigma$ such that taking any $r\in[0,\theta_{\alpha^{(0)}}]$ and $\alpha^{(0)}_1=\alpha^{(0)}_2=\alpha^{(0)}$, we have
\begin{gather*}
P^{(0)}(\pi(R^j_{\mu_f}(x_1,r)))=I^{(0)}_{\alpha_1^{(j)}}, \quad P^{(0)}(\pi(R^j_{\mu_f}(x_2,r)))=I^{(0)}_{\alpha_2^{(j)}}\quad \text{for}\quad 0\leq j\leq n \\
P^{(1)}(\pi(R^{j-1}_{\mu_f}(x_1,r)))=f^{i_1^{(j)}}I^{(1)}_{\alpha_1^{(j)}},\quad P^{(1)}(\pi(R^{j-1}_{\mu_f}(x_2,r)))=f^{i_2^{(j)}}I^{(1)}_{\alpha_2^{(j)}}\quad \text{for}\quad 1\leq j\leq n.
\end{gather*}
Then $\beta(\alpha^{(j)}_1,i^{(j)}_1)=\alpha^{(j-1)}_1$ and $\beta(\alpha^{(j)}_2,i^{(j)}_2)=\alpha^{(j-1)}_2$ for $1\leq j\leq n$ and $\alpha^{(n)}_1=\alpha^{(n)}_2=\alpha^{(n)}$. As the incidence matrix for the renormalization map $R_{\mu_f}$ is positive, we can choose $(\alpha^{(n+1)},i^{(n+1)})\in\Sigma$ such that $\beta(\alpha^{(n+1)},i^{(n+1)})=\alpha^{(n)}$ and $0\leq i_1^{(0)},i_2^{(0)}<q_{\alpha^{(0)}}$ such that $\beta(\alpha_1^{(0)},i_1^{(0)})=\alpha^{(n+1)}$ and $\beta(\alpha_2^{(0)},i_2^{(0)})=\alpha^{(n+1)}$, thus obtaining two loops.

As $\vartheta=\vartheta_\pm$, in view of \eqref{eq:infla=0}, we have
\begin{align*}
\pm\log\frac{\big|f^{i_1}I^{(n)}_{\alpha^{(n)}}\big|}{\big|I^{(0)}_{\alpha^{(0)}}\big|}=\pm\sum_{j=1}^n\log\frac{\big|f^{i^{(j)}_1}I^{(1)}_{\alpha_1^{(j)}}\big|}{\big|I^{(0)}_{\alpha_1^{(j-1)}}\big|}=
\sum_{j=1}^n\vartheta\big((\alpha_1^{(j-1)},i^{(j-1)}_1)\to(\alpha_1^{(j)},i^{(j)}_1)\big),\\
\pm\log\frac{|f^{i_2}I^{(n)}_{\alpha^{(n)}}|}{|I^{(0)}_{\alpha^{(0)}}|}=\pm\sum_{j=1}^n\log\frac{|f^{i^{(j)}_2}I^{(1)}_{\alpha_1^{(j)}}|}{|I^{(0)}_{\alpha_2^{(j-1)}}|}=
\sum_{j=1}^n\vartheta\big((\alpha_2^{(j-1)},i^{(j-1)}_2)\to(\alpha_2^{(j)},i^{(j)}_2)\big).
\end{align*}
Let us consider two cyclic paths $C^1$, $C^2$ in $\mathcal{C}_T$ given respectively by
\begin{align*}
(\alpha_1^{(0)},i^{(0)}_1)\to(\alpha_1^{(1)},i^{(1)}_1)\to\cdots\to(\alpha_1^{(n)},i^{(n)}_1)\to(\alpha^{(n+1)},i^{(n+1)})\to(\alpha_1^{(0)},i^{(0)}_1),\\
(\alpha_2^{(0)},i^{(0)}_2)\to(\alpha_2^{(1)},i^{(1)}_2)\to\cdots\to(\alpha_2^{(n)},i^{(n)}_2)\to(\alpha^{(n+1)},i^{(n+1)})\to(\alpha_2^{(0)},i^{(0)}_2).
\end{align*}
Then
\begin{align*}
\vartheta(C^1)&=\sum_{j=1}^n\vartheta\big((\alpha_1^{(j-1)},i^{(j-1)}_1)\to(\alpha_1^{(j)},i^{(j)}_1)\big)\\
&\quad +\vartheta\big((\alpha_1^{(n)},i^{(n)}_1)\to(\alpha^{(n+1)},i^{(n+1)})\big)+
\vartheta\big((\alpha^{(n+1)},i^{(n+1)})\to(\alpha_1^{(0)},i^{(0)}_1)\big)\\
& = \pm\log\frac{\big|f^{i_1}I^{(n)}_{\alpha^{(n)}}\big|}{\big|I^{(0)}_{\alpha^{(0)}}\big|} \pm\log\frac{\big|f^{i^{(n+1)}}I^{(1)}_{\alpha^{(n+1)}}\big|}{\big|I^{(0)}_{\alpha^{(n)}}\big|}\pm\log\frac{\big|f^{i^{(0)}_1}I^{(1)}_{\alpha_1^{(0)}}\big|}{\big|I^{(0)}_{\alpha^{(n+1)}}\big|},
\end{align*}
and similarly
\begin{align*}
\vartheta(C^2) = \pm\log\frac{\big|f^{i_2}I^{(n)}_{\alpha^{(n)}}\big|}{\big|I^{(0)}_{\alpha^{(0)}}\big|} \pm\log\frac{\big|f^{i^{(n+1)}}I^{(1)}_{\alpha^{(n+1)}}\big|}{\big|I^{(0)}_{\alpha^{(n)}}\big|}\pm\log\frac{\big|f^{i^{(0)}_2}I^{(1)}_{\alpha_2^{(0)}}\big|}{\big|I^{(0)}_{\alpha^{(n+1)}}\big|}.
\end{align*}
As by assumption $\lambda([C^1])$, $\lambda([C^2])$ are positive, by Lemma~\ref{lem:maxcyl}, we have
\[\vartheta(C^1)=\vartheta(C^2)=(n+2)\zeta(\vartheta).\]
In view of the previous two formulas for $\vartheta(C^1)$ and $\vartheta(C^2)$, this gives
\begin{equation}\label{eq:pairs}
\big|f^{i_1}I^{(n)}_{\alpha^{(n)}}\big| \big|f^{i^{(0)}_1}I^{(1)}_{\alpha_1^{(0)}}\big| = \big|f^{i_2}I^{(n)}_{\alpha^{(n)}}\big| \big| f^{i^{(0)}_2}I^{(1)}_{\alpha_2^{(0)}} \big| .
\end{equation}

We will repeat the same argument for shorter paths. As the incidence matrix for the renormalization map $R_{\mu_f}$ is positive, we can choose $\underline{i}^{(n+1)}$ such that $(\alpha^{(n+1)},\underline{i}^{(n+1)})\in\Sigma$ and $\beta(\alpha^{(n+1)},\underline{i}^{(n+1)})=\alpha^{(0)}$. Let's consider two short cyclic paths
\begin{gather*}
\underline{C}^1=\Big((\alpha_1^{(0)},i^{(0)}_1)\to(\alpha^{(n+1)},\underline{i}^{(n+1)})\to(\alpha_1^{(0)},i^{(0)}_1)\Big),\\
\underline{C}^2=\Big((\alpha_2^{(0)},i^{(0)}_2)\to(\alpha^{(n+1)},\underline{i}^{(n+1)})\to(\alpha_2^{(0)},i^{(0)}_2)\Big).
\end{gather*}
As by assumption $\lambda([\underline{C}^1])$, $\lambda([\underline{C}^2])$ are positive, by Lemma~\ref{lem:maxcyl}, we have
$\vartheta(\underline{C}^1)=\vartheta(\underline{C}^2)=2\zeta(\vartheta)$. As
\begin{gather*}
\vartheta(\underline{C}^1)=\pm\log\frac{\big|f^{\underline{i}^{(n+1)}}I^{(1)}_{\alpha^{(n+1)}}\big|}{\big|I^{(0)}_{\alpha^{(0)}}\big|}\pm\log\frac{\big|f^{i^{(0)}_1}I^{(1)}_{\alpha_1^{(0)}}\big|}{\big|I^{(0)}_{\alpha^{(n+1)}}\big|},\\
\vartheta(\underline{C}^2)=
\pm\log\frac{\big|f^{\underline{i}^{(n+1)}}I^{(1)}_{\alpha^{(n+1)}}\big|}{\big|I^{(0)}_{\alpha^{(0)}}\big|}\pm\log\frac{\big|f^{i^{(0)}_2}I^{(1)}_{\alpha_2^{(0)}}\big|}{\big|I^{(0)}_{\alpha^{(n+1)}}\big|},
\end{gather*}
this gives $\big|f^{i^{(0)}_1}I^{(1)}_{\alpha^{(0)}}\big|=\big|f^{i^{(0)}_2}I^{(1)}_{\alpha^{(0)}}\big|$. In view of \eqref{eq:pairs}, we have
\[\big|f^{i_1}I^{(n)}_{\alpha^{(n)}}\big|=\big|f^{i_2}I^{(n)}_{\alpha^{(n)}}\big|\quad\text{whenever}\quad f^{i_1}I^{(n)}_{\alpha^{(n)}}\subset I^{(0)}_{\alpha^{(0)}}\quad\text{and}\quad f^{i_2}I^{(n)}_{\alpha^{(n)}}\subset I^{(0)}_{\alpha^{(0)}}.\]
It follows that for any $n\geq 1$ the map $g^{Leb}_{I^{(n)}}:[0,1]^\theta\to\R_{\geq 0}$ given by \eqref{def:gA} is zero. Therefore, by Lemma~\ref{lem:Xn}, the Lebesgue measure on $[0,1]$ is $f$-invariant, which is in contradiction with the non-triviality of the log-slope vector $\omega$. This shows that $\int_{\widehat\Sigma}\vartheta\,d\lambda<\zeta(\vartheta)$.
\end{proof}

\begin{proposition}\label{prop:zeta>}
Let $T=(\pi,\lambda)$ be a self-similar IET of $d\ge 2$ intervals and let $M$ be its positive self-similarity matrix. Let $f\in \operatorname{Aff}(T,\omega)$ be an AIET, semi-conjugated to $T$, with vector of logarithms of slopes $\omega$. If $\omega\neq 0$ is an invariant vector of the matrix $M$, then $\zeta^f>\mathcal G(T,\omega)$.
\end{proposition}

\begin{proof}
Recall that the measure $\mathfrak{m}$, constructed in Section~\ref{sec:Markov}, is a $\sigma$-invariant ergodic Markov measure on $\widehat\Sigma$ such that $\mathfrak{m}([\varrho])>0$ for any $\varrho\in \Pi_T$ and
\[\int_{\widehat\Sigma}\vartheta_-\,d\mathfrak{m}=-\sum_{(\alpha,i)\in\Sigma}\log\frac{Leb(f^iI^{(1)}_\alpha)}{Leb(I^{(0)}_{\beta(\alpha,i)})}\mu_f(f^iI^{(1)}_\alpha)\theta_{\beta(\alpha,i)}=H^u_{Leb}(R_{\mu_{f}}).\]
In view of Proposition~\ref{prop:zeta0}, this gives $\mathcal G(T,\omega)=\int_{\widehat\Sigma}\vartheta_-\,d\mathfrak{m}<\zeta(\vartheta_-)=\zeta$.
\end{proof}

\subsection{Unstable log-slope vector.} Finally, assume that $\omega$ has a non-trivial unstable component in the decomposition w.r.t.\ the basis of eigenvectors of $M$. To show that then $h$ is not H\"older, we first show the following general fact.
\begin{lemma}\label{lem:holder_upper_bound}
Let $\nu$ be any Borel probability on $\R$ such that for some $x_0\in I$ inequality $\liminf_{\epsilon \rightarrow 0 } \frac{\log \nu[x_0-\epsilon, x_0+\epsilon)}{\log \epsilon} \leq \alpha$ holds for some $\alpha > 0 $. Then the distribution function given by the formula $h(y)=\nu[0,y)$ has H\"older exponent bounded from above by $\alpha$. In particular, if $\operatorname{dim}_H(\nu)=0$, then $h$ is not H\"older.
\begin{proof}
 Suppose that for some $\gamma>0$
 $h$ is $\gamma$-H\"older i.e. there exist $C>0$ such that
 \begin{equation*}
 \nu[x,y) = h(y)-h(x) < C(y-x)^{\gamma},
 \end{equation*}
 holds for all $x<y \in I$. In particular, by taking logarithm, we get
 \begin{equation*}
 \log \nu[x_0-\varepsilon,x_0+\varepsilon) <\log (C(2\varepsilon)^{\gamma})=\gamma\log(2\varepsilon)+\log(C),
 \end{equation*}
 for sufficiently small $\varepsilon>0 $.
 Dividing both sides by $\log(2\varepsilon)<0$ gives
 \begin{equation*}
 \frac{\log \nu[x_0-\varepsilon,x_0+\varepsilon)}{\log(2\varepsilon)} > \gamma+\frac{\log(C)}{\log(2\varepsilon)},
 \end{equation*}
 which, after passing to the limit as $\varepsilon \rightarrow 0 $, yields $\alpha \geq \gamma$.
\end{proof}
\end{lemma}

The following Proposition deals with the final case of Theorem~\ref{thm: main2}.								
\begin{proposition}
If $\omega$ is of unstable type, any semi-conjugacy between $f$ and $(\pi, \lambda)$ is not H\"older continuous.
\begin{proof}
 Since we have $h_*\mu_f = Leb$ and $h$ is a non-decreasing surjection, the semi-conjugacy $h:I\rightarrow I$ is a distribution function of the measure $\mu_f$, i.e., it is given by the formula $h(x)=\mu_f[0,x)$. Since by case (3) in Theorem~\ref{thm: main1}, we have $\operatorname{dim}(\mu_f)=0$,
 the claim follows from Lemma~\ref{lem:holder_upper_bound}.
\end{proof}
\end{proposition}

\section{Regularity of the inverse of conjugacy for AIETs with central-stable log-slope vector}\label{sec:regofinverse}

Suppose that the log-slope vector $\omega$ of $f\in \textup{Aff}(T,\omega)$ is of central-stable type, with $T$ being hyperbolically self-similar, with period $N$ and self-similarity matrix $M$. Recall that in this case, the semi-conjugacy $h$ is actually a conjugacy, and we can ask about the regularity of its inverse $h^{-1}:I\to I$.
Assume now that $\omega=\omega_c+\omega_s$, where $\omega_c$ is an invariant vector for $M$ and $\omega_s$ is of stable type. As in the previous section, by Proposition~\ref{prop:conjcs}, any AIET $f_c\in \operatorname{Aff}(T,\omega_c)$ is $C^1$-conjugated to $f$. Thus, we can reduce the proof to the simpler case where $\omega=\omega_c$, i.e.\ $\omega$ is an invariant vector of $M$.

\medskip								

From now on, we will usually assume that $\omega$ is an invariant vector.
Let $\nu=\nu_\omega$ be the unique $\phi_{\omega}$-conformal measure for the IET $T$. Then $h^{-1}(x)=\nu([0,x])$ for any $x\in[0,1]$. 	Consider the family of dynamical partitions $(\mathcal Q^{(n)})_{n\geq 0}$ of $T$ given by the Rauzy-Veech induction. Then $\mathcal Q^{(0)}=(I^{(0)}_\alpha)_{\alpha\in\mathcal A}=(I^{(0)}_\alpha(T))_{\alpha\in\mathcal A}$ is the partition $[0,1)$ into intervals exchanged by $T$ and the vector $\nu:=(\nu(I^{(0)}_\alpha))_{\alpha\in\mathcal{A}}$ is the left Perron-Frobenius eigenvector of the matrix $M(\omega):=M_{\pi,\lambda,\omega}^{(N)}$ defined in Section~\ref{subs: cocycles}. Denote by $\rho_\omega=\rho_{\nu}>0$ the logarithm of the Perron-Frobenius eigenvalue of $M(\omega)$. Then $\nu(I^{(n)}_\alpha)=\nu(R^{-n}I^{(0)}_\alpha)=e^{-n\rho_\nu}\nu(I^{(0)}_\alpha)$ for any $n\geq 1$ and $\alpha\in\mathcal{A}$. 	Let $\theta=(\theta_\alpha)_{\alpha\in\mathcal A}$ be the right Perron -Frobenius eigenvector of $M_{\pi,\lambda,\omega}^{(N)}$. Then $R_\nu:[0,1)^\theta\to [0,1)^\theta$ is a Borel automorphism preserving the measure $\nu^\theta$. By Lemmas~\ref{lem:markow-chain}~and~\ref{lem:irrMar}, $R_\nu$ on $([0,1)^\theta,\nu^\theta)$
is isomorphic with the Markov shift $\sigma$ on $(\widehat{\Sigma},\mathfrak{n})$ via the map $\mathfrak{Q}:[0,1)^\theta\to\widehat{\Sigma}$, where $\mathfrak{n}:=\mathfrak{Q}_*(\nu^{\theta})$, and $\mathfrak{n}([C])>0$ for any $ $.

For any $n\geq 1$ let
\begin{align*}
\xi_n:&=\frac{1}{n}\min_{(x,r)\in[0,1)^\theta}\left(-\sum_{0\leq j<n}\log\frac{\nu(Q^{(1)}(\pi R^j_{\nu}(x,r)))}{\nu(Q^{(0)}(\pi R^j_{\nu}(x,r)))}\right)\\
&=\frac{1}{n}\min_{\substack{((\alpha_j,i_j))_{j=0}^{n}\\ \beta(\alpha_j,i_j)=\alpha_{j-1}}}\left(-\sum_{1\leq j\leq n}\log\frac{\nu(T^{i_j}I_{\alpha_j}^{(1)})}{\nu(I_{\beta(\alpha_j,i_j)})}\right)>0.
\end{align*}
As $(m+n)\xi_{m+n}\geq m\xi_{m}+n\xi_{n}$, the sequence $(\xi_n)_n$ converges and let
\begin{equation}\label{def:xi}
\xi=\xi^f:=\lim_{n\to\infty}\xi_n=\sup_{n\geq 1}\xi_n>0.
\end{equation}

For any AIET $f\in \operatorname{Aff}(T,\omega)$ of hyperbolic periodic type, semi-conjugated to a self-similar IET $T=(\pi,\lambda)$, with the log-slope vector $\omega$ of central-stable type we define
\begin{equation}\label{def:def:xi}
 \xi^f:=\xi^{f_c},
\end{equation}
where $f_c\in \operatorname{Aff}(T,\omega_c)$ is an AIET related to the invariant vector $\omega_c$ in the decomposition of $\omega=\omega_c+\omega_s$ into an invariant and a stable type vectors.

If $\omega$ is an invariant vector (again), standard weak-limit arguments show that
\[\xi=\min\Big\{\int_{\widehat{\Sigma}}- \log\frac{\nu(T^{i_0}I_{\alpha_0}^{(1)})}{\nu(I_{\beta(\alpha_0,i_0)})}\,d\lambda\big((\alpha_j,i_j)\big)_{j\in\Z}\mid \lambda\in \Lambda(\widehat{\Sigma}, \sigma)\Big\},\]
where $\Lambda(\widehat{\Sigma}, \sigma)$ is the simplex of $\sigma$-invariant probability measures on $\widehat{\Sigma}$. As $\mathfrak{n}:=\mathfrak{Q}_*(\nu^{\theta})\in \Lambda(\widehat{\Sigma}, \sigma)$, we have
\begin{equation}\label{eq: zetagreaterthanH1}
\begin{split}
\xi&\leq - \log\frac{\nu(T^{i_0}I_{\alpha_0}^{(1)})}{\nu(I_{\beta(\alpha_0,i_0)})}\,d\mathfrak{n}\big((\alpha_j,i_j)\big)_{j\in\Z}\\
&=\int_{I^\theta}- \log\frac{\nu(Q^{(1)}(\pi(x,r)))}{\nu(Q^{(0)}(\pi (x,r)))}\,d\nu^\theta(x,r)\\
& =-\sum_{\alpha\in \mathcal{A}}\sum_{0\leq i<q_\alpha}\log\frac{\nu(T^iI^{(1)}_\alpha)}{\nu(I^{(0)}_{\beta(\alpha,i)})}\nu(T^iI^{(1)}_\alpha)\theta_{\beta(\alpha,i)}=\mathcal H(T,\omega),
\end{split}
\end{equation}
where $\mathcal H(T,\omega)$ is defined in \ref{eq: HD_conformal_formula}

As $h:[0,1]\to[0,1]$ is a conjugacy between $f$ and $T$ with $\nu=h_*(Leb)$, we have $h(I^{(0)}_\beta(f))=I^{(0)}_\beta(T)$ and $h(f^iI^{(1)}_\alpha(f))=T^iI^{(1)}_\alpha(T)$. Thus
\[\log\frac{Leb(f^iI^{(1)}_\alpha(f))}{Leb(I^{(0)}_\beta(f))}=\log\frac{\nu(T^iI^{(1)}_\alpha(T))}{\nu(I^{(0)}_\beta(T))}.\]
It follows that
\[\xi^f=-\zeta(\vartheta_+)\quad\text{and}\quad \mathcal H(T,\omega)=-\int_{\widehat \Sigma}\vartheta_+\,d\mathfrak{n}.\]
In view of Proposition~\ref{prop:zeta0}, this gives
\[\xi^f<\mathcal H(T,\omega)\quad\text{whenever}\quad\omega\neq 0.\]
Moreover, Lemma~\ref{lem: allzetassame} gives an effective formula for counting $\xi^f$:
\begin{equation}\label{def:xi1}
\xi^f=-\zeta(\vartheta_+)=-\max_{C\in \mathcal{C}^{el}_T}\frac{\vartheta_+(C)}{|C|}=\min_{C\in \mathcal{C}^{el}_T}\frac{\vartheta_-(C)}{|C|}.
\end{equation}

Recall that $\rho_T$ is the logarithm of the Perron-Frobenius eigenvalue of the self-similarity matrix $M$ for the IET $T$. As $\frac{\mathcal H(T,\omega)}{\rho_{T}}$ is the Hausdorff dimension of the unique $\phi_\omega$-conformal measure $\nu$ for $T$, we get
\begin{equation}\label{eq:dimxi}
\dim_{H}(\nu)=\frac{\mathcal H(T,\omega)}{\rho_{T}}>\frac{\xi^f}{\rho_{T}}>0.
\end{equation}
In this section we prove that $\frac{\xi^f}{\rho_{T}}$ the supremum of H\"older exponents of the inverse $h^{-1}$. The proof splits into the following two propositions.
\begin{proposition}\label{prop: centralHolderestimate1}
The inverse $h^{-1}$ is $\gamma$-H\"older for every
\[
0<\gamma< \frac{\xi}{\rho_{T}}.
\]
\end{proposition}
\begin{proof}
Let $N$ be a natural number large enough, to be specified later. Let
\[\delta=\delta_N:=\min\{|J|\mid J\in \mathcal Q^{(N)}\}=e^{-N\rho_T}\min\{|I_{\alpha}|\mid \alpha\in \mathcal A\}>0.\]
Since
\[\max\{|J|\mid J\in \mathcal Q^{(n)}\}=e^{-n\rho_T}\max\{|I_{\alpha}|\mid \alpha\in \mathcal A\}\to 0\quad{as}\quad n\to\infty,\]
for any $x<y$ with $|x-y|<\delta$ there exists $n> N$ such that
\[
n:=\min\{m\in \mathbb N\mid \#\{[x,y]\cap \partial \mathcal P^{(m)}\}\ge 2 \}.
\]
Then there exist intervals $I^{-}_{n-1},I^{+}_{n-1}\in \mathcal Q^{(n-1)}$ and an interval $I_{n}\in \mathcal Q^{(n)}$ such that
\begin{equation}\label{eq: intervalinclusion1}
 I_{n}\subseteq [x,y]\subseteq I^{-}_{n-1}\cup I^{+}_{n-1}.
 \end{equation}
We want to prove that for every $0<\gamma< \frac{\xi}{\rho_{T}}$, we have
\begin{equation*}\label{eq: Holder goal1}
|h^{-1}(x)-h^{-1}(y)|\le |x-y|^{\gamma}\quad \text{for all $x,y$ such that}\quad |x-y|<\delta,
\end{equation*}
which is equivalent to
\begin{equation*}\label{eq: Holder reduction1}
\frac{-\log\nu([x,y])}{-\log Leb([x,y])}\ge \gamma,
\end{equation*}
where $\nu$ is the unique probability $\phi_\omega$-conformal measure.
To prove the above inequality, we will prove that for every sufficiently small $\epsilon>0$ there exists $N$ large enough such that for every $x<y$ satisfying $|x-y|<\delta_N$, the following holds
\begin{equation}\label{eq: second Holder reduction1}
 \frac{-\log\nu([x,y])}{-\log Leb([x,y])}\ge \frac{\xi}{\rho_{T}}-\epsilon.
 \end{equation}
First note that
\begin{align}\label{eq: Holder numerator estimate1}
\begin{aligned}
-\log Leb([x,y])&\leq -\log Leb(I_{n})\leq -\log(e^{-n\rho_{T}}\cdot \min_{\alpha\in\mathcal A}|I_\alpha|)\\
&=n\rho_{T}+\max_{\alpha\in\mathcal A}\log|I_\alpha|^{-1}.
\end{aligned}
\end{align}
For any $x_0\in I^{\pm}_{n-1}$, we have $I^{\pm}_{n-1}={Q}^{(n-1)}(x_0)$. Moreover, in view of \eqref{eq:infla=0}, for any choice $(x_0,r)\in I^\theta$, we have
\begin{align*}
 -\log \nu({Q}^{(n-1)}(x_0))
& =-\sum_{0\leq i<n-1}\log\frac{\nu({Q}^{(1)}(\pi(R^i_\nu(x_0,r))))}{\nu({Q}(\pi(R^{i}_\nu(x_0,r))))}-\log \nu({Q}(x_0))\\
&\geq (n-1) \xi_{n-1}.
 \end{align*}
It follows that
\begin{equation}\label{eq: Holder denominator estimate1}
-\log \nu([x,y])\geq -\log 2\cdot \max \{\nu(I^+_{n-1}),\nu(I^-_{n-1})\}\geq -\log 2+ (n-1) \xi_{n-1}.
\end{equation}
By combining \eqref{eq: Holder numerator estimate1} and \eqref{eq: Holder denominator estimate1}, we get
\[
\frac{-\log\nu([x,y])}{-\log Leb([x,y])}\geq \frac{ -\log 2+ (n-1) \xi_{n-1}}{n\rho_{T}+\max_{\alpha\in\mathcal A}\log|I_\alpha|^{-1}},
\]
with $n>N$. Choose $N$ large enough such that
\[\frac{ -\log 2+ (n-1) \xi_{n-1}}{n\rho_{T}+\max_{\alpha\in\mathcal A}\log|I_\alpha|^{-1}}>\frac{\xi}{\rho_{T}}-\epsilon=\gamma\quad \text{for all}\quad n>N,\]
by taking $\epsilon:= \frac{\xi}{\rho_{T}}-\gamma>0$.
Hence, we get $|h^{-1}(x)-h^{-1}(y)|\leq |x-y|^\gamma$, whenever $|x-y|<\delta_N$. Following the arguments from the end of the proof of Proposition~\ref{prop: centralHolderestimate}, we get
\begin{equation*}\label{eq:finHol_inv}
|h^{-1}(x)-h^{-1}(y)|\leq (1+\delta^{-1})^{1-\gamma}|x-y|^\gamma\quad \text{for all}\quad x,y\in[0,1).
\end{equation*}
\end{proof}

\begin{proposition}\label{prop:optH1}
There exists an infinite sequence of intervals $[x_n,y_n]\subset I$, such that
\begin{equation}\label{eq:limfrlog1}
\lim_{n\to\infty}|x_n-y_n|=0\quad\text{and}\quad\lim_{n\to\infty}\frac{-\log\nu([x_n,y_n])}{-\log Leb([x_n,y_n])}=\frac{\xi}{\rho_{T}}.
\end{equation}
\end{proposition}
\begin{proof}
For any $n\geq 1$, choose $(z_n,r_n)\in I^\theta$ such that
\[\xi_n=-\frac{1}{n}\sum_{0\leq j<n}\log\frac{\nu(Q^{(1)}(\pi R^j_{\nu}(z_n,r_n)))}{\nu(Q^{(0)}(\pi R^j_{\nu}(z_n,r_n)))}.\]
Then, in view of \eqref{eq:infla=0}, we have
\[-\log \nu(Q^{(n)}(z_n))=n\xi_n-\log \nu(Q(z_n)).\]
Let $[x_n,y_n)=Q^{(n)}(z_n)$. It follows that
\[n\xi_n\leq -\log \nu([x_n,y_n))\leq n\xi_n+\max_{\alpha\in\mathcal A}\log\nu(I_\alpha)^{-1}.\]
If $[x_n,y_n)=T^jI^{(n)}_\alpha$, then $Leb([x_n,y_n))=|I^{(n)}_\alpha|=e^{-n\rho_T}|I_\alpha|$. Thus
\begin{gather*}
n\rho_T\leq -\log Leb([x_n,y_n))\leq n\rho_T+\max_{\alpha\in\mathcal A}\log|I_\alpha|^{-1}.
\end{gather*}
As $\nu$ is continuous, it follows that
\[
\frac{n\xi_n}{ n\rho_T+\max_{\alpha\in\mathcal A}\log|I_\alpha|^{-1}}
\leq \frac{-\log\nu([x_n,y_n])}{-\log Leb([x_n,y_n])}\leq \frac{n\xi_n+\max_{\alpha\in\mathcal A}\log\nu(I_\alpha)^{-1}}{n\rho_T}.
\]
Since $\xi_n\to\xi$, this gives \eqref{eq:limfrlog1}.
\end{proof}

\begin{proposition}\label{cor:she1}
Let $T=(\pi,\lambda)$ be a self-similar IET of $d\ge 2$ intervals and let $M$ be its positive self-similarity matrix. Let $f\in \operatorname{Aff}(T,\omega)$ be an AIET, semi-conjugated to $T$, with the log-slope vector $\omega$. If $\omega\neq 0$ is a central-stable vector of the matrix $M$ and $h:[0,1]\to[0,1]$ is the conjugacy of $T$ and $f$, then
\[\mathfrak{H}(h^{-1})=\frac{\xi^f}{\rho_T}\quad\text{and}\quad 0<\frac{\xi^f}{\rho_T}<\dim_H(\nu_\omega),\]
where $\rho_T$ is the logarithm of the Perron-Frobenius eigenvalue of $M$, $\xi^f$ is given by \eqref{def:def:xi} and $\nu_\omega$ is the unique $\phi_\omega$-conformal measure for $T$.
\end{proposition}

\begin{proof}
As $\omega$ is a central-stable vector, we have $\omega=\omega_s+\omega_c$, where $\omega_s$ is a stable vector and $\omega_c\neq 0$ is a central eigenvector of $M$. As we have already noted, $f$ is $C^1$-conjugated to $f_c\in \operatorname{Aff}(T,\omega_c)$. If $h_c:[0,1]\to[0,1]$ is the conjugacy of $T$ and $f_c$, then ${h_c}^{-1}\circ h$ is a $C^1$-diffeomorphism of $[0,1]$ conjugating $f$ and $f_c$. It follows that
\[\mathfrak{H}(h)=\mathfrak{H}(h_c).\]
By definition,
\[\frac{d((T^{-1})_*\nu_{\omega_c})}{d\nu_{\omega_c}}=e^{\phi_{\omega_c}}\quad\text{and}\quad \frac{d((T^{-1})_*\nu_{\omega})}{d\nu_{\omega}}=e^{\phi_{\omega}}=e^{\phi_{\omega_c}}e^{\phi_{\omega_s}}.\]
As $\omega_s$ is of stable type, in view of Proposition~\ref{prop:cohom}, the exists a continuous solution $v:I\to\R$ to the cohomological equation $\phi_{\omega_s}=v\circ T-v$. It follows that
\[\frac{d(T^{-1})_*\nu_\omega}{d\nu_\omega}=e^{\phi_{\omega_c}}e^{v\circ T-v},\quad\text{so}\quad \frac{d(T^{-1})_*(e^{-v}\nu_\omega)}{d(e^{-v}\nu_\omega)}=e^{\phi_{\omega_c}}.\]
As $\omega_c$ is of central-stable type, by Proposition~\ref{prop:conformal_measures}, the $\phi_{\omega_c}$-conformal measure is unique. Therefore, there exists real $c$ such that $d\nu_\omega=e^{v+c}d\nu_{\omega_c}$.
It follows that
\[\dim_H(\nu_\omega)=\dim_H(\nu_{\omega_c}).\]

In view of Proposition~\ref{prop: centralHolderestimate1}, we have $\mathfrak{H}(h_c^{-1})\geq \frac{\xi^{f_c}}{\rho_T}$. Suppose, contrary to our claim, that $\mathfrak{H}(h_c^{-1})> \frac{\xi^{f_c}}{\rho_T}$. Then there exists $\gamma > \frac{\xi^{f_c}}{\rho_T}$ such that $h_c^{-1}$ is $\gamma$-H\"older, so there exists $C>0$ such that $|h_c^{-1}(y)-h_c^{-1}(x)|\leq C|y-x|^\gamma$ for all $x,y\in[0,1]$. As $h_c^{-1}(x)=\nu_{\omega_c}[0,x]$ and the measure $\nu_{\omega_c}$ is continuous, this gives
\[\frac{-\log \nu_{\omega_c}[x,y]}{-\log Leb[x,y]}\geq \gamma +\frac{-\log C}{-\log|x-y|}\quad \text{if}\quad x<y.\]
In view of Proposition~\ref{prop:optH1}, there exists an infinite sequence of intervals $[x_n,y_n]\subset [0,1]$, such that
\[
\lim_{n\to\infty}|x_n-y_n|=0\quad\text{and}\quad\lim_{n\to\infty}\frac{-\log\nu_{\omega_c}([x_n,y_n])}{-\log Leb([x_n,y_n])}=\frac{\xi^{f_c}}{\rho_{T}}.
\]
It follows that
\[\frac{\xi^{f_c}}{\rho_{T}}=\lim_{n\to\infty}\frac{-\log{\nu_{\omega_c}}([x_n,y_n])}{-\log Leb([x_n,y_n])}\geq \gamma,\]
which is in contradiction with the choice of $\gamma > \frac{\xi^{f_c}}{\rho_{T}}$. By definition, $\xi^f=\xi^{f_c}$. This gives
\[\mathfrak{H}(h^{-1})=\mathfrak{H}(h_c^{-1})=\frac{\xi^{f_c}}{\rho_{T}}=\frac{\xi^{f}}{\rho_{T}}.\]
Finally, the inequality
\[\frac{\xi^{f}}{\rho_{T}}=\frac{\xi^{f_c}}{\rho_{T}}<\dim_H(\nu_{\omega_c})=\dim_H(\nu_\omega),\] follows directly from \eqref{eq:dimxi}.
\end{proof}

\section*{Acknowledgements}
 This research was partially supported by the Narodowe Centrum Nauki Grant Preludium Bis 2023/50/O/ST1/00045. The last author has received funding from the European Union's Horizon 2020 research and innovation programme under the Marie Sk\l odowska-Curie grant agreement No. 101154283.

\appendix
{
\section{On the regularity of conjugacies between AIETs}\label{sec: appA}

We begin this appendix with a result describing the regularity of solutions to the cohomological equation $\phi_\omega=v\circ T-v$ for piecewise constant functions $\phi_\omega$ corresponding to vectors $\omega$ of stable type. The existence of continuous solutions was noticed by Marmi-Moussa-Yoccoz in \cite{marmi_cohomological_2005}. Next, Marmi-Yoccoz in \cite{marmi_holder_2016} improved the regularity of the solutions by showing that they are H\"older. A precise analysis of the regularity of the solutions was recently presented in \cite{fraczek_kim_partII}. The result we need is directly taken from \cite{fraczek_kim_partII}.

\begin{proposition}\label{prop:cohom}
Suppose that $T$ is an IET of hyperbolic periodic type. Let $\omega\in\R^{\mathcal A}$ be a non-zero log-slope vector of stable type. Then there exists a H\"older continuous solution $v:I\to\R$ to the cohomological equation $\phi_\omega=v\circ T-v$ such that
\begin{itemize}
\item $\mathfrak{H}(v)=\frac{\alpha(\omega)}{\rho_T}$ if $\alpha(\omega)<\rho_T$, and
\item $v(x)=ax$ for some non-zero real $a$ if $\alpha(\omega)=\rho_T$.
\end{itemize}
\end{proposition}		

\begin{proof}
Let
\[e^{-\lambda_1}<e^{-\lambda_2}<\ldots<e^{-\lambda_g}<1<e^{\lambda_g}<\ldots<e^{\lambda_2}<e^{\lambda_1}\]
be all eigenvalues of the self-similarity matrix $M=M(T)$. Of course, $\lambda_1=\rho_T$. Assume that $h_{-i}\in\R^{\mathcal A}$, $1\leq i\leq g$, is a basis of the stable subspace such that $Mh_{-i}=e^{-\lambda_i}h_{-i}$, for $1\leq i\leq g$. Then, the IET $T$ satisfies the Full Filtration Diophantine Condition (FFDC) introduced in \cite[Section~3]{fraczek_kim_partII}. Moreover, by Corollary~6.7 and Remark~6.9 in \cite{fraczek_kim_partII}, for every $1\leq i\leq g$ there exists a H\"older continuous solution $v_{-i}:I\to\R$ to the cohomological equation $\phi_{h_{-i}}=v_{-i}\circ T-v_{-i}$. Moreover,
\begin{itemize}
\item if $1<i\leq g$ then $\mathfrak{H}(v_{-i})=\frac{\lambda_i}{\lambda_1}=\frac{\lambda_i}{\rho_T}<1$, and
\item $v_{-1}(x)=ax$ for some non-zero real $a$.
\end{itemize}
Since $\omega$ is a linear combinations of the vectors $h_{-i}\in\R^{\mathcal A}$, $1\leq i\leq g$, this gives our claim.
\end{proof}

Proposition~\ref{prop:cohom} gives us a simple way to prove the regularity of a conjugacy (and its inverse) of an AIET with IET, in the case where the log-slope vector is of stable type. This, in turn, shows parts (1) and (2) in Theorems~\ref{thm: main2}~and~\ref{thm: main4}.				

\begin{proposition}\label{prop:stableconj}
Let $f\in \operatorname{Aff}(T,\omega)$ be an AIET of hyperbolic periodic type, semi-conjugated via $h:I\to I$ to a self-similar IET $T=(\pi,\lambda)$, with the log-slope vector $\omega\neq 0$ of stable type. Then $h$ is a homeomorphism such that
\begin{itemize}
\item $\mathfrak{H}(h)=\mathfrak{H}(h^{-1})=1+\frac{\alpha(\omega)}{\rho_T}$ if $\alpha(\omega)<\rho_T$, and
\item $h:I\to I$ is a $C^\infty$-diffeomorphism if $\alpha(\omega)=\rho_T$.
\end{itemize}
\end{proposition}	

\begin{proof}
Let $v:I\to\R$ be a continuous solution to the cohomological equation $\phi_\omega=v\circ T-v$.
Let $\nu:= h_*(Leb)$. Then $\nu$ is a $\phi_\omega$-conformal measure for $T$, and
\[\frac{d(T^{-1})_*\nu}{d\nu}=e^{\phi_\omega}=e^{v\circ T-v}.\]
It follows that
\[\frac{d(T^{-1})_*(e^{-v}\nu)}{d(e^{-v}\nu)}=1.\]
As $T$ is uniquely ergodic, we have $d \nu(x)=e^{v(x)+c}dx$ for some real $c$. Since $\nu$ is a continuous measure, $h$ must be a homomorphism. It follows that
\[\int_0^xe^{v(y)+c}dy=\nu([0,x])=Leb(h^{-1}([0,x]))=h^{-1}(x).\]
Therefore, $(h^{-1})'(x)=e^{v(x)+c}$, so $h$ is a $C^1$-diffeomorphism.

In view of Proposition~\ref{prop:cohom}, if $\alpha(\omega)=\rho_T$, then $(h^{-1})'(x)=e^{ax+c}$. Thus, $h$ is a $C^\infty$-diffeomorphism.

However, if $\alpha(\omega)<\rho_T$, then $(h^{-1})'(x)=e^{v(x)+c}$ with $v$ H\"older continuous and $\mathfrak{H}(v)=\frac{\alpha(\omega)}{\rho_T}$. Therefore,
\[\mathfrak{H}(h^{-1})=1+\mathfrak{H}(e^{v+c})=1+\mathfrak{H}(v)=1+\frac{\alpha(\omega)}{\rho_T}.\]
As $h'(x)=\frac{1}{(h^{-1})'(hx)}=e^{-v(h(x))-c}$ and $h$ is bi-Lipschitz, we have
\[\mathfrak{H}(h)=1+\mathfrak{H}(e^{-v\circ h-c})=1+\mathfrak{H}(v\circ h)=1+\mathfrak{H}(v)=1+\frac{\alpha(\omega)}{\rho_T}.\]
\end{proof}									

\begin{proposition}\label{prop:conjcs}
Suppose that $T$ is an IET of hyperbolic periodic type. Let $\omega_1,\omega_2\in\R^{\mathcal A}$ be log-slope vectors of central-stable type such that their difference $\omega:=\omega_1-\omega_2$ is of stable type. Let $f_1$, $f_2$ be two AIETs such that $f_1\in \operatorname{Aff}(T,\omega_1)$ and $f_2\in \operatorname{Aff}(T,\omega_2)$. Then $f_1$ and $f_2$ are $C^1$-conjugated.
\end{proposition}							

\begin{proof}
Let $h_i:I\to I$ be a semi-conjugacy of $f_i$ with the IET $T$, for $i=1,2$. As $\omega_i$ is of central-stable type, $h_i$ is a homeomorphism, for $i=1,2$. Let $\nu_i:= (h_i)_*(Leb)$, for $i=1,2$. Then $\nu_i$ is a $\phi_{\omega_i}$-conformal measure for $T$, and
\[\frac{d(T^{-1})_*\nu_1}{d\nu_1}=e^{\phi_{\omega_1}}\quad\text{and}\quad \frac{d(T^{-1})_*\nu_2}{d\nu_2}=e^{\phi_{\omega_2}}=e^{\phi_{\omega_1}}e^{\phi_{\omega}}.\]
As $\omega$ is of stable type, by Proposition~\ref{prop:cohom}, there exists a H\"older solution $v:I\to\R$ to the cohomological equation $\phi_\omega=v\circ T-v$. It follows that
\[\frac{d(T^{-1})_*\nu_2}{d\nu_2}=e^{\phi_{\omega_1}}e^{v\circ T-v},\quad\text{so}\quad \frac{d(T^{-1})_*(e^{-v}\nu_2)}{d(e^{-v}\nu_2)}=e^{\phi_{\omega_1}}.\]
As $\omega_1$ is of central-stable type, by Proposition~\ref{prop:conformal_measures}, the $\phi_{\omega_1}$-conformal measure is unique. Therefore, there exists real $c$ such that
\begin{equation}\label{eq:mu12}
d\nu_2=e^{v+c}d\nu_1.
\end{equation}
Let us consider the homeomorphism $h:I\to I$ conjugating $f_2$ and $f_1$ given by $h:=h_1^{-1}\circ h_2$. Then
\[h_*(Leb)=(h_1^{-1})_*((h_2)_*(Leb))=(h_1^{-1})_*(\nu_2),\]
so
\[h^{-1}(x)=\int_0^{h_1(x)}e^{v(y)+c}d\nu_1(y)=\int_0^{h_1(x)}e^{v(y)+c}d(h_1)_*(Leb)(y)=\int_0^{x}e^{v\circ h_1(y)+c}dy.\]
As $v\circ h_1$ is continuous, it follows that $h$ is a $C^1$-diffeomorphism.
\end{proof}	
As a direct consequence of the proof of Proposition~\ref{prop:conjcs} (see \eqref{eq:mu12}), we have the following result.
\begin{corollary}\label{cor: bded_density}
 Suppose that $T$ is an IET of hyperbolic periodic type. Let $\omega_1,\omega_2\in\R^{\mathcal A}$ be vectors of central-stable (or stable) type such that their difference $\omega:=\omega_1-\omega_2$ is of stable type. Let $\nu_{\omega_1}$ and $\nu_{\omega_2}$ be the unique $\phi_{\omega_1}$- and $\phi_{\omega_2}$-conformal measures. Then $\nu_{\omega_1}$ and $\nu_{\omega_2}$ are equivalent and the Radon-Nikodym derivative $\frac{d\nu_{\omega_1}}{d\nu_{\omega_2}}$ is a bounded continuous function, bounded away from 0.
\end{corollary}

In fact, Proposition~\ref{prop:conjcs} is also a direct consequence of Theorem~3.1 in \cite{berk_rigidity_2024}, which is much more general, and it is formulated for almost every $T$.	
However, due to the precise knowledge of the H\"older exponents of the conjugacies $h_i$ (see, Proposition~\ref{cor:she}) and the solution to the cohomological equation $v$ (see, Proposition~\ref{prop:cohom}), we are now able to improve the regularity of the conjugacy between $f_1$ and $f_2$, but only in the case of the hyperbolic periodic type. We believe that for almost every combinatorial rotation number, improving the regularity of conjugacy is not possible.

\begin{proposition}\label{prop:conjcsh}
Suppose that $T$ is an IET of hyperbolic periodic type. Let $\omega_1,\omega_2\in\R^{\mathcal A}$ be log-slope vectors of central-stable type such that their difference $\omega:=\omega_1-\omega_2$ is of stable type. Let $f_1$, $f_2$ be two AIETs such that $f_1\in \operatorname{Aff}(T,\omega_1)$ and $f_2\in \operatorname{Aff}(T,\omega_2)$. Then the conjugacy $h$ between $f_1$ and $f_2$ is such that
\[\mathfrak{H}(h), \mathfrak{H}(h^{-1})\geq 1+\frac{\alpha(\omega)}{\zeta^{f_1}}.\]
\end{proposition}		

\begin{proof}
As $h^{-1}(x)=\int_0^{x}e^{v\circ h_1(y)+c}dy$, we have
\[(h^{-1})'(x)=e^{v\circ h_1(x)+c}\quad\text{and}\quad h'(x)=e^{-v\circ h_1(hx)-c}=e^{-v\circ h_2(x)-c}.\]
Therefore,
\[\mathfrak{H}(h^{-1})=1+\mathfrak{H}(v\circ h_1)\geq 1+ \mathfrak{H}(v)\mathfrak{H}(h_1) \text{ and }
\mathfrak{H}(h)=1+\mathfrak{H}(v\circ h_2)\geq 1+\mathfrak{H}(v)\mathfrak{H}(h_2).\]
Recall that, by Propositions~\ref{cor:she}~and~\ref{prop:cohom}, we have
\[\mathfrak{H}(h_i)=\frac{\rho_T}{\zeta^{f_i}}\quad\text{with}\quad \zeta^{f_1}=\zeta^{f_2},\quad\text{and}\quad\mathfrak{H}(v)\geq\frac{\alpha(\omega)}{\rho_T}.\]
This completes the proof.
\end{proof}			
}							
\bibliographystyle{acm}

\begin{thebibliography}{10}

\bibitem{berk_rigidity_2024}
{\sc Berk, P., and Trujillo, F.}
\newblock Rigidity for piecewise smooth circle homeomorphisms and certain
 {GIETs}.
\newblock {\em Advances in Mathematics 441}, (2024), 109560.

\bibitem{bressaud_persistence_2010}
{\sc Bressaud, X., Hubert, P., and Maass, A.}
\newblock Persistence of wandering intervals in self-similar affine interval
 exchange transformations.
\newblock {\em Ergodic Theory and Dynamical Systems 30}, (2010),
 665--686.


\bibitem{camelier_affine_1997}
{\sc Camelier, R., and Gutierrez, C.}
\newblock Affine interval exchange transformations with wandering intervals.
\newblock {\em Ergodic Theory and Dynamical Systems 17}, (1997),
 1315--1338.

\bibitem{Ch-Ga-Ug}
{\sc Chazottes, J.-R., Gambaudo, J.-M., and Ugalde, E.}
\newblock Zero-temperature limit of one-dimensional {G}ibbs states via
 renormalization: the case of locally constant potentials.
\newblock {\em Ergodic Theory and Dynamical Systems 31}, (2011), 1109--1161.

\bibitem{cobo_piece-wise_2002}
{\sc Cobo, M.}
\newblock Piece-wise affine maps conjugate to interval exchanges.
\newblock {\em Ergodic Theory and Dynamical Systems 22}, (2002),
 375--407.

\bibitem{cobo_wandering_2018}
{\sc Cobo, M., Guti\'errez-Romo, R., and Maass, A.}
\newblock Wandering intervals in affine extensions of self-similar interval
 exchange maps: the cubic {Arnoux}-{Yoccoz} map.
\newblock {\em Ergodic Theory and Dynamical Systems 38}, (2018),
 2537--2570.

\bibitem{dzhalilov_circle_2006}
{\sc Dzhalilov, A., and Liousse, I.}
\newblock Circle homeomorphisms with two break points.
\newblock {\em Nonlinearity 19}, (2006), 1951--1968.


\bibitem{dzhalilov_singular_2009}
{\sc Dzhalilov, A., Liousse, I., and Mayer, D.}
\newblock Singular measures of piecewise smooth circle homeomorphisms with two
 break points.
\newblock {\em Discrete \& Continuous Dynamical Systems 24}, (2009), 381.


\bibitem{dzhalilov_invariant_1998}
{\sc Dzhalilov, A.~A., and Khanin, K.~M.}
\newblock On an invariant measure for homeomorphisms of a circle with a point
 of break.
\newblock {\em Functional Analysis and Its Applications 32}, (1998),
 153--161.

\bibitem{fraczek_kim_partII}
{\sc Fr\k{a}czek, K., and Kim, M.}
\newblock Solving the cohomological equation for locally hamiltonian flows,
 part II -- Global obstructions.
 \emph{Proceedings of the London Mathematical Society 131}, (2025), e70094.

\bibitem{ghazouani_local_2021}
{\sc Ghazouani, S.}
\newblock Local rigidity for periodic generalised interval exchange
 transformations.
\newblock {\em Inventiones mathematicae 226}, (2021), 467--520.

\bibitem{ghazouani_priori_2023}
{\sc Ghazouani, S., and Ulcigrai, C.}
\newblock A priori bounds for {GIETs}, affine shadows and rigidity of
 foliations in genus two.
\newblock {\em Publications math\'ematiques de l'IH\'ES 138}, (2023), 229--366.

\bibitem{herman_sur_1979}
{\sc Herman, M.~R.}
\newblock Sur la conjugaison diff\'erentiable des diff\'eomorphismes du cercle \`a
 des rotations.
\newblock {\em Publications Math\'ematiques de l'IH\'ES 49} (1979), 5--233.

\bibitem{hofbauer_hausdorff_1992}
{\sc Hofbauer, F., and Raith, P.}
\newblock The {Hausdorff} {dimension} of an {ergodic} {invariant} {measure} for
 a {piecewise} {monotonic} {map} of the {interval}.
\newblock {\em Canadian Mathematical Bulletin 35}, (1992), 84--98.

\bibitem{keane_interval_1975}
{\sc Keane, M.}
\newblock Interval exchange transformations.
\newblock {\em Mathematische Zeitschrift 141}, (1975), 25--31.

\bibitem{khanin_hausdorff_2017}
{\sc Khanin, K., and Koci\'c, S.}
\newblock Hausdorff dimension of invariant measure of circle diffeomorphisms
 with a break point.
\newblock {\em Ergodic Theory and Dynamical Systems 39}, (2019), 1331--1339.

\bibitem{Kingman}
{\sc Kingman, J. F.~C.}
\newblock A convexity property of positive matrices.
\newblock {\em Q. J. Math., Oxf. II. Ser. 12}, (1961), 283--284.

\bibitem{liousse_nombre_2005}
{\sc Liousse, I.}
\newblock Nombre de rotation, mesures invariantes et ratio set des
 hom\'eomorphismes affines par morceaux du cercle.
\newblock {\em Annales de l'Institut Fourier 55}, (2005), 431--482.

\bibitem{marmi_cohomological_2005}
{\sc Marmi, S., Moussa, P., and Yoccoz, J.-C.}
\newblock The cohomological equation for {Roth}-type interval exchange maps.
\newblock {\em Journal of the American Mathematical Society 18}, (2005),
 823--872.

\bibitem{marmi_affine_2010}
{\sc Marmi, S., Moussa, P., and Yoccoz, J.-C.}
\newblock Affine interval exchange maps with a wandering interval.
\newblock {\em Proceedings of the London Mathematical Society 100}, (2010),
 639--669.

\bibitem{marmi_linearization_2012}
{\sc Marmi, S., Moussa, P., and Yoccoz, J.-C.}
\newblock Linearization of generalized interval exchange maps.
\newblock {\em Annals of Mathematics 176}, (2012), 1583--1646.

\bibitem{marmi_holder_2016}
{\sc Marmi, S., and Yoccoz, J.-C.}
\newblock H\"older {regularity} of the {solutions} of the {cohomological}
 {equation} for {Roth} {type} {interval} {exchange} {maps}.
\newblock {\em Communications in Mathematical Physics 344}, (2016),
 117--139.

\bibitem{przytycki_hausdorff_1989}
{\sc Przytycki, F., and Urba\'nski, M.}
\newblock On the {Hausdorff} dimension of some fractal sets.
\newblock {\em Studia Mathematica 93}, (1989), 155--186.

\bibitem{rauzy_echanges_1979}
{\sc Rauzy, G.}
\newblock \'Echanges d'intervalles et transformations induites.
\newblock {\em Acta Arithmetica 34}, (1979), 315--328.

\bibitem{trujillo_hausdorff_2024}
{\sc Trujillo, F.}
\newblock On the {Hausdorff} dimension of invariant measures of piecewise
 smooth circle homeomorphisms.
\newblock {\em Ergodic Theory and Dynamical Systems 44}, (2024), 3599--3629.

\bibitem{trujillo_affine_2024}
{\sc Trujillo, F., and Ulcigrai, C.}
\newblock Affine interval exchange maps with a singular conjugacy to an {IET}.
\newblock To appear in {\em Annali Scuola Normale Superiore - Classe di Scienze} (2024),
\url{DOI: 10.2422/2036-2145.202305_001}.

\bibitem{veech_interval_1978}
{\sc Veech, W.~A.}
\newblock Interval exchange transformations.
\newblock {\em Journal d'Analyse Math\'ematique 33}, (1978), 222--272.

\bibitem{veech_gauss_1982}
{\sc Veech, W.~A.}
\newblock Gauss {measures} for {transformations} on the {space} of {interval}
 {exchange} {maps}.
\newblock {\em Annals of Mathematics 115}, (1982), 201--242.

\bibitem{yoccoz_echanges_2005}
{\sc Yoccoz, J.-C.}
\newblock \'Echanges d'intervalles.
\newblock Lecture notes (2005) available at \url{https://www.college-de-france.fr/media/jean-christophe-yoccoz/UPL8726_yoccoz05.pdf}


\bibitem{yoccoz_interval_2010}
{\sc Yoccoz, J.-C.}
\newblock Interval exchange maps and translation surfaces.
\newblock In {\em Homogeneous flows, moduli spaces and arithmetic}, vol.~10 of
 {\em Clay {Math}. {Proc}.} Amer. Math. Soc., Providence, RI, 2010, pp.~1--69.

\bibitem{zorich_finite_1996}
{\sc Zorich, A.}
\newblock Finite {Gauss} measure on the space of interval exchange
 transformations. {Lyapunov} exponents.
\newblock {\em Annales de l'Institut Fourier 46}, (1996), 325--370.

\end{thebibliography}

\end{document}